\documentclass[numbers=enddot,12pt,final,onecolumn,notitlepage]{scrartcl}%
\usepackage[headsepline,footsepline,manualmark]{scrlayer-scrpage}
\usepackage{amssymb}
\usepackage{amsmath}
\usepackage{amsthm}
\usepackage{framed}
\usepackage{comment}
\usepackage{color}
\usepackage[breaklinks=True]{hyperref}
\usepackage[sc]{mathpazo}
\usepackage[T1]{fontenc}
\usepackage{needspace}
\usepackage{tabls}
\providecommand{\U}[1]{\protect\rule{.1in}{.1in}}
\theoremstyle{definition}
\newtheorem{theo}{Theorem}[section]
\newenvironment{theorem}[1][]
{\begin{theo}[#1]\begin{leftbar}}
{\end{leftbar}\end{theo}}
\newtheorem{lem}[theo]{Lemma}
\newenvironment{lemma}[1][]
{\begin{lem}[#1]\begin{leftbar}}
{\end{leftbar}\end{lem}}
\newtheorem{prop}[theo]{Proposition}
\newenvironment{proposition}[1][]
{\begin{prop}[#1]\begin{leftbar}}
{\end{leftbar}\end{prop}}
\newtheorem{defi}[theo]{Definition}
\newenvironment{definition}[1][]
{\begin{defi}[#1]\begin{leftbar}}
{\end{leftbar}\end{defi}}
\newtheorem{remk}[theo]{Remark}
\newenvironment{remark}[1][]
{\begin{remk}[#1]\begin{leftbar}}
{\end{leftbar}\end{remk}}
\newtheorem{coro}[theo]{Corollary}
\newenvironment{corollary}[1][]
{\begin{coro}[#1]\begin{leftbar}}
{\end{leftbar}\end{coro}}
\newtheorem{conv}[theo]{Convention}
\newenvironment{convention}[1][]
{\begin{conv}[#1]\begin{leftbar}}
{\end{leftbar}\end{conv}}
\newtheorem{quest}[theo]{Question}
\newenvironment{question}[1][]
{\begin{quest}[#1]\begin{leftbar}}
{\end{leftbar}\end{quest}}
\newtheorem{warn}[theo]{Warning}

\newtheorem{conj}[theo]{Conjecture}

\newtheorem{exam}[theo]{Example}
\newenvironment{example}[1][]
{\begin{exam}[#1]\begin{leftbar}}
{\end{leftbar}\end{exam}}
\newenvironment{statement}{\begin{quote}}{\end{quote}}
\newenvironment{fineprint}{\begin{small}}{\end{small}}
\let\sumnonlimits\sum
\let\prodnonlimits\prod
\let\cupnonlimits\bigcup
\let\capnonlimits\bigcap
\renewcommand{\sum}{\sumnonlimits\limits}
\renewcommand{\prod}{\prodnonlimits\limits}
\renewcommand{\bigcup}{\cupnonlimits\limits}
\renewcommand{\bigcap}{\capnonlimits\limits}
\newcommand{\timesu}{\mathbin{\underline{\times}}}

\renewcommand{\leq}{\leqslant}
\renewcommand{\geq}{\geqslant}

\newtheoremstyle{plainsl}
{8pt plus 2pt minus 4pt}
{8pt plus 2pt minus 4pt}
{\slshape}
{0pt}
{\bfseries}
{.}
{5pt plus 1pt minus 1pt}
{}
\theoremstyle{plainsl}
\ihead{The Bhargava greedoid as a Gaussian elimination greedoid}
\ohead{page \thepage}
\begin{document}

\title{The Bhargava greedoid as a Gaussian elimination greedoid}
\author{Darij Grinberg\thanks{Drexel University, Korman Center, Room 263, 15 S 33rd
Street, Philadelphia PA, 19104, USA}}
\date{27 April 2024}
\maketitle

\begin{abstract}
\textbf{Abstract.} Inspired by Manjul Bhargava's theory of generalized
factorials, Fedor Petrov and the author have defined the \textit{Bhargava
greedoid} -- a greedoid (a matroid-like set system on a finite set) assigned
to any \textquotedblleft ultra triple\textquotedblright\ (a somewhat extended
variant of a finite ultrametric space). Here we show that the Bhargava
greedoid of a finite ultra triple is always a \textit{Gaussian elimination
greedoid} over any sufficiently large (e.g., infinite) field; this is a
greedoid analogue of a representable matroid. We find necessary and sufficient
conditions on the size of the field to ensure this.

\end{abstract}
\tableofcontents

\section*{***}

The notion of a \emph{greedoid} was coined in 1981 by Korte and Lov\'{a}sz,
and has since seen significant developments (\cite{KoLoSc91}, \cite{BjoZie92}%
). It is a type of set system (i.e., a set of subsets of a given ground set)
that is required to satisfy some axioms weaker than the matroid axioms -- so
that, in particular, the independent sets of a matroid form a greedoid.

In \cite{GriPet19}, Grinberg and Petrov have constructed a greedoid stemming
from Bhargava's theory of generalized factorials \cite[\S 2]{Bharga97}, albeit
in a setting significantly more general than Bhargava's. Roughly speaking, the
sets that belong to this greedoid are subsets of maximum perimeter (among all
subsets of their size) of a finite ultrametric space (which, in Bhargava's
work, was a Dedekind ring with a metric coming from a valuation).

More precisely, the setup is more general than that of an ultrametric space:
We consider a finite set $E$, a \emph{distance function} $d$ that assigns a
\textquotedblleft distance\textquotedblright\ $d\left(  e,f\right)  $ to any
pair $\left(  e,f\right)  $ of distinct elements of $E$, and a \emph{weight
function} $w$ that assigns a \textquotedblleft weight\textquotedblright%
\ $w\left(  e\right)  $ to each $e\in E$. The distances and weights are
required to belong to a totally ordered abelian group $\mathbb{V}$ (for
example, $\mathbb{R}$). The distances are required to satisfy the symmetry
axiom $d\left(  e,f\right)  =d\left(  f,e\right)  $ and the \textquotedblleft
ultrametric triangle inequality\textquotedblright\ $d\left(  a,b\right)
\leq\max\left\{  d\left(  a,c\right)  ,d\left(  b,c\right)  \right\}  $. In
this setting, any subset $S$ of $E$ has a well-defined \emph{perimeter},
obtained by summing the weights and the pairwise distances of all its
elements. The subsets $S$ of $E$ that have maximum perimeter (among all
$\left\vert S\right\vert $-element subsets of $E$) then form a greedoid, which
has been called the \emph{Bhargava greedoid} of $\left(  E,w,d\right)  $ in
\cite{GriPet19}. This greedoid is furthermore a strong greedoid \cite[Theorem
6.1]{GriPet19}, which implies in particular that for any given $k\leq
\left\vert E\right\vert $, the $k$-element subsets of $E$ that have maximum
perimeter are the bases of a matroid.

In the present paper, we prove that the Bhargava greedoid\ of $\left(
E,w,d\right)  $ is a \emph{Gaussian elimination greedoid} over any
sufficiently large (e.g., infinite) field. Roughly speaking, a Gaussian
elimination greedoid is a greedoid analogue of a representable
matroid\footnote{In particular, this entails that all the matroids mentioned
in the preceding paragraph are representable.}. We quantify the
\textquotedblleft sufficiently large\textquotedblright\ by providing a
sufficient condition for the size of the field. When all weights $w\left(
e\right)  $ are equal, we show that this condition is also necessary.

We note that the Bhargava greedoid can be seen to arise from an optimization
problem in phylogenetics: Given a finite set $E$ of organisms and an integer
$k\in\mathbb{N}$, we want to choose a $k$-element subset of $E$ that maximizes
some kind of biodiversity. Depending on the definition of biodiversity used,
the properties of the maximizing subsets can differ. It appears natural to
define biodiversity in terms of distances on the evolutionary tree (which is a
finite ultrametric space), and such a definition has been considered by
Moulton, Semple and Steel in \cite{MoSeSt06}, leading to the result that the
maximum-biodiversity sets form a strong greedoid. The Bhargava greedoid is an
analogue of their greedoid using a slightly different definition of
biodiversity\footnote{To be specific: We view the organisms as the leaves of
an evolutionary tree $\mathcal{T}$ that obeys a molecular clock assumption
(i.e., all its leaves have the same distance from the root). Then, the set $E$
of these organisms is equipped with a distance function (measuring distances
along the edges of the tree), which satisfies the \textquotedblleft
ultrametric triangle inequality\textquotedblright. We define the weight
function $w$ by setting $w\left(  e\right)  =0$ for all $e\in E$. Now, the
\emph{phylogenetic diversity} of a subset $S\subseteq E$ is defined to be the
sum of the edge lengths of the minimal subtree of $\mathcal{T}$ that connects
all leaves in $S$. This phylogenetic diversity is the measure of biodiversity
used in \cite{MoSeSt06}. Meanwhile, our notion of perimeter can also be seen
as a measure of biodiversity -- perhaps even a better one for sustainability
questions, as it rewards subsets that are roughly balanced across different
clades. To give a trivial example, a zoo optimized for phylogenetic diversity
might have dozens of mammals and only one bird, while this would unlikely be
considered optimal in terms of perimeter.
\par
The molecular clock assumption can actually be dropped, at the expense of
changing the weight function to account for different distances from the
root.}. The present paper potentially breaks this analogy by showing that the
Bhargava greedoid is a Gaussian elimination greedoid, whereas this is unknown
for the greedoid of Moulton, Semple and Steel. Whether the latter is a
Gaussian elimination greedoid as well remains to be understood\footnote{This
question might have algorithmic significance. At least for polymatroids,
representability can make the difference between a problem being NP-hard and
in P, as shown by Lov\'{a}sz in \cite{Lovasz80} for polymatroid matching.}, as
does the question of interpolating between the two notions of biodiversity.

\bigskip

This paper is self-contained (up to some elementary linear algebra), and in
particular can be read independently of \cite{GriPet19}.

The 12-page extended abstract \cite{GriPet20} summarizes the highlights of
both \cite{GriPet19} and this paper; it is thus a convenient starting point
for a reader interested in the subject.

\subsubsection*{Acknowledgments}

I thank Fedor Petrov for his (major) part in the preceding project
\cite{GriPet19} that led to this one, and an anonymous referee for helpful
comments and corrections.

This work has been started during a Leibniz fellowship at the Mathematisches
Forschungsinstitut Oberwolfach, and completed at the Institut Mittag-Leffler
in Djursholm; I thank both institutes for their hospitality.

\begin{fineprint}
This material is based upon work supported by the Swedish Research Council
under grant no. 2016-06596 while the author was in residence at Institut
Mittag-Leffler in Djursholm, Sweden during Spring 2020.
\end{fineprint}

\section{\label{sect.geg}Gaussian elimination greedoids}

\subsection{The definition}

\begin{convention}
Here and in the following, $\mathbb{N}$ denotes the set $\left\{
0,1,2,\ldots\right\}  $.
\end{convention}

\begin{convention}
If $E$ is any set, then $2^{E}$ will denote the powerset of $E$ (that is, the
set of all subsets of $E$).
\end{convention}

\begin{convention}
Let $\mathbb{K}$ be any field, and let $n\in\mathbb{N}$. Then, $\mathbb{K}%
^{n}$ shall denote the $\mathbb{K}$-vector space of all column vectors of size
$n$ over $\mathbb{K}$.
\end{convention}

We recall the definition of a Gaussian elimination greedoid:

\begin{definition}
\label{def.geg}Let $E$ be a finite set.

Let $m\in\mathbb{N}$ be such that $m\geq\left\vert E\right\vert $. Let
$\mathbb{K}$ be a field. For each $k\in\left\{  0,1,\ldots,m\right\}  $, let
$\pi_{k}:\mathbb{K}^{m}\rightarrow\mathbb{K}^{k}$ be the projection map that
removes all but the first $k$ coordinates of a column vector. (That is,
$\pi_{k}\left(
\begin{array}
[c]{c}%
a_{1}\\
a_{2}\\
\vdots\\
a_{m}%
\end{array}
\right)  =\left(
\begin{array}
[c]{c}%
a_{1}\\
a_{2}\\
\vdots\\
a_{k}%
\end{array}
\right)  $ for each $\left(
\begin{array}
[c]{c}%
a_{1}\\
a_{2}\\
\vdots\\
a_{m}%
\end{array}
\right)  \in\mathbb{K}^{m}$.)

For each $e\in E$, let $v_{e}\in\mathbb{K}^{m}$ be a column vector. The family
$\left(  v_{e}\right)  _{e\in E}$ will be called a \textit{vector family} over
$\mathbb{K}$.

Let $\mathcal{G}$ be the subset%
\[
\left\{  F\subseteq E\ \mid\ \text{the family }\left(  \pi_{\left\vert
F\right\vert }\left(  v_{e}\right)  \right)  _{e\in F}\in\left(
\mathbb{K}^{\left\vert F\right\vert }\right)  ^{F}\text{ is linearly
independent}\right\}
\]
of $2^{E}$. Then, $\mathcal{G}$ is called the \textit{Gaussian elimination
greedoid} of the vector family $\left(  v_{e}\right)  _{e\in E}$. It is
furthermore called a \textit{Gaussian elimination greedoid on ground set }$E$.
\end{definition}

\begin{example}
\label{exa.geg.matrix2}Let $\mathbb{K}=\mathbb{Q}$ and $E=\left\{
1,2,3,4,5\right\}  $ and $m=6$. Let $v_{1},v_{2},v_{3},v_{4},v_{5}%
\in\mathbb{K}^{6}$ be the columns of the $6\times5$-matrix%
\[
\left(
\begin{array}
[c]{ccccc}%
0 & 1 & 1 & 0 & 1\\
1 & 1 & 0 & 0 & 0\\
0 & 2 & 1 & 0 & 1\\
1 & 0 & 1 & 0 & 0\\
0 & 0 & 0 & 0 & 0\\
1 & 2 & 0 & 2 & 1
\end{array}
\right)  .
\]
Then, the Gaussian elimination greedoid of the vector family $\left(
v_{e}\right)  _{e\in E}=\left(  v_{1},v_{2},v_{3},v_{4},v_{5}\right)  $ is the
set%
\begin{align*}
&  \left\{  \varnothing,\left\{  2\right\}  ,\left\{  3\right\}  ,\left\{
5\right\}  ,\left\{  1,2\right\}  ,\left\{  1,3\right\}  ,\left\{
1,5\right\}  ,\left\{  2,3\right\}  ,\left\{  2,5\right\}  ,\right. \\
&  \ \ \ \ \ \ \ \ \ \ \left.  \left\{  1,2,3\right\}  ,\left\{
1,2,5\right\}  ,\left\{  1,2,3,5\right\}  \right\}  .
\end{align*}
For example, the $3$-element set $\left\{  1,2,5\right\}  $ belongs to this
greedoid because the family $\left(  \pi_{3}\left(  v_{e}\right)  \right)
_{e\in\left\{  1,2,5\right\}  }\in\left(  \mathbb{K}^{3}\right)  ^{\left\{
1,2,5\right\}  }$ is linearly independent (indeed, this family consists of the
vectors $\left(
\begin{array}
[c]{c}%
0\\
1\\
0
\end{array}
\right)  $, $\left(
\begin{array}
[c]{c}%
1\\
1\\
2
\end{array}
\right)  $ and $\left(
\begin{array}
[c]{c}%
1\\
0\\
1
\end{array}
\right)  $).
\end{example}

Our definition of a Gaussian elimination greedoid follows \cite[\S 1.3]%
{Knapp18}, except that we are using vector families instead of matrices (but
this is equivalent, since any matrix can be identified with the vector family
consisting of its columns) and we are talking about linear independence rather
than non-singularity of matrices (but this is again equivalent, since a square
matrix is non-singular if and only if its columns are linearly independent).
The same definition is given in \cite[\S IV.2.3]{KoLoSc91}.

\subsection{Context}

In the rest of Section \ref{sect.geg}, we shall briefly connect Definition
\ref{def.geg} with known concepts in the theory of greedoids. This is not
necessary for the rest of our work, so the impatient reader can well skip to
Section \ref{sect.Vultra}.

As the name suggests, Gaussian elimination greedoids are instances of
greedoids -- a class of set systems (i.e., sets of sets) characterized by some
simple axioms. We refer to Definition \ref{def.sg} below for the definition of
a greedoid, and to \cite{KoLoSc91} for the properties of such. A subclass of
greedoids that has particular interest to us are the \textit{strong
greedoids}; see, e.g., Section \ref{sect.plucker} below or \cite[\S 6.1]%
{GriPet19} or \cite[\S 2]{BrySha99} for their definition.\footnote{They also
appear in \cite[\S IX.4]{KoLoSc91} under the name of \textquotedblleft Gauss
greedoids\textquotedblright, but they are defined differently. (The
equivalence between the two definitions is proved in \cite[\S 2]{BrySha99}.)}
The following theorem is implicit in \cite[\S IX.4]{KoLoSc91}\footnote{A
partial proof of Theorem \ref{thm.geg.strong} also appears in \cite[\S 1.3]%
{Knapp18}. (Namely, two paragraphs above \cite[Example 1.3.15]{Knapp18}, it is
shown that $\mathcal{G}$ is a greedoid.)}:

\begin{theorem}
\label{thm.geg.strong}The Gaussian elimination greedoid $\mathcal{G}$ in
Definition \ref{def.geg} is a strong greedoid.
\end{theorem}

See Section \ref{sect.plucker} below for a proof of this theorem.

Matroids are a class of set systems more famous than greedoids; see
\cite{Oxley11} for their definition. We will not concern ourselves with
matroids much in this note, but let us remark one connection to Gaussian
elimination greedoids:\footnote{See \cite[\S 1.1]{Oxley11} for the definition
of a representable matroid.}

\begin{proposition}
\label{prop.geg.matroid-rep}Let $\mathcal{G}$ be a Gaussian elimination
greedoid on a ground set $E$. Let $k\in\mathbb{N}$. Let $\mathcal{G}_{k}$ be
the set of all $k$-element sets in $\mathcal{G}$. Then, $\mathcal{G}_{k}$ is
either empty or is the collection of bases of a representable matroid on the
ground set $E$.
\end{proposition}

See Section \ref{sect.matroid-rep} below for a proof of this proposition.

Proposition \ref{prop.geg.matroid-rep} justifies thinking of Gaussian
elimination greedoids as a greedoid analogue of representable matroids.

\section{\label{sect.Vultra}$\mathbb{V}$-ultra triples}

\begin{definition}
Let $E$ be a set. Then, $E\timesu E$ shall denote the subset $\left\{  \left(
e,f\right)  \in E\times E\ \mid\ e\neq f\right\}  $ of $E\times E$.
\end{definition}

\begin{convention}
\label{conv.V}Fix a totally ordered abelian group $\left(  \mathbb{V}%
,+,0,\leq\right)  $ (with ground set $\mathbb{V}$, group operation $+$, zero
$0$ and smaller-or-equal relation $\leq$). The total order on $\mathbb{V}$ is
required to be translation-invariant (i.e., if $a,b,c\in\mathbb{V}$ satisfy
$a\leq b$, then $a+c\leq b+c$).

We shall refer to the ordered abelian group $\left(  \mathbb{V},+,0,\leq
\right)  $ simply as $\mathbb{V}$. We will use the standard additive notations
for the abelian group $\mathbb{V}$; in particular, we will use the $\sum$ sign
for finite sums inside the group $\mathbb{V}$. We will furthermore use the
standard order-theoretical notations for the totally ordered set $\mathbb{V}$;
in particular, we will use the symbol $\geq$ for the reverse relation of
$\leq$ (that is, $a\geq b$ means $b\leq a$), and we will use the symbols $<$
and $>$ for the strict versions of the relations $\leq$ and $\geq$. We will
denote the largest element of a nonempty subset $S$ of $\mathbb{V}$ (with
respect to the relation $\leq$) by $\max S$. Likewise, $\min S$ will stand for
the smallest element of $S$.

We will keep this group $\mathbb{V}$ fixed throughout this paper.
\end{convention}

For almost all examples we are aware of, it suffices to set $\mathbb{V}$ to be
the abelian group $\mathbb{R}$, or even the smaller abelian group $\mathbb{Z}%
$. Nevertheless, we shall work in full generality, as it serves to separate
objects that would otherwise easily be confused.

\begin{definition}
\label{def.Vultra}A $\mathbb{V}$\textit{-ultra triple} shall mean a triple
$\left(  E,w,d\right)  $ consisting of:

\begin{itemize}
\item a set $E$, called the \textit{ground set} of this $\mathbb{V}$-ultra triple;

\item a map $w:E\rightarrow\mathbb{V}$, called the \textit{weight function} of
this $\mathbb{V}$-ultra triple;

\item a map $d:E\timesu E\rightarrow\mathbb{V}$, called the \textit{distance
function} of this $\mathbb{V}$-ultra triple, and required to satisfy the
following axioms:

\begin{itemize}
\item \textbf{Symmetry:} We have $d\left(  a,b\right)  =d\left(  b,a\right)  $
for any two distinct elements $a$ and $b$ of $E$.

\item \textbf{Ultrametric triangle inequality:} We have $d\left(  a,b\right)
\leq\max\left\{  d\left(  a,c\right)  ,d\left(  b,c\right)  \right\}  $ for
any three distinct elements $a$, $b$ and $c$ of $E$.
\end{itemize}
\end{itemize}

If $\left(  E,w,d\right)  $ is a $\mathbb{V}$-ultra triple and $e\in E$, then
the value $w\left(  e\right)  \in\mathbb{V}$ is called the \textit{weight} of
$e$.

If $\left(  E,w,d\right)  $ is a $\mathbb{V}$-ultra triple and $e$ and $f$ are
two distinct elements of $E$, then the value $d\left(  e,f\right)
\in\mathbb{V}$ is called the \textit{distance} between $e$ and $f$.
\end{definition}

\begin{example}
\label{exa.Vultra.a=bmodm}For this example, let $\mathbb{V}=\mathbb{Z}$, and
let $E$ be a subset of $\mathbb{Z}$. Let $m$ be any integer. Define a map
$w:E\rightarrow\mathbb{V}$ arbitrarily. Define a map $d:E\timesu
E\rightarrow\mathbb{V}$ by%
\[
d\left(  a,b\right)  =%
\begin{cases}
1, & \text{if }a\not \equiv b\operatorname{mod}m;\\
0, & \text{if }a\equiv b\operatorname{mod}m
\end{cases}
\ \ \ \ \ \ \ \ \ \ \text{for all }\left(  a,b\right)  \in E\timesu E.
\]
It is easy to see that $\left(  E,w,d\right)  $ is a $\mathbb{V}$-ultra triple.
\end{example}

\begin{example}
\label{exa.Vultra.bhar}For this example, let $\mathbb{V}=\mathbb{Z}$ again,
and let $E$ be a subset of $\mathbb{Z}$. Fix a prime number $p$. For each
nonzero integer $k$, let $v_{p}\left(  k\right)  $ denote the largest
$i\in\mathbb{N}$ such that $p^{i}\mid k$. (For instance, $v_{3}\left(
45\right)  =2$.)

Define a map $w:E\rightarrow\mathbb{V}$ arbitrarily. Define a map
$d:E\timesu
E\rightarrow\mathbb{V}$ by%
\[
d\left(  a,b\right)  =-v_{p}\left(  a-b\right)  \ \ \ \ \ \ \ \ \ \ \text{for
all }\left(  a,b\right)  \in E\timesu E.
\]
It is easy to see that $\left(  E,w,d\right)  $ is a $\mathbb{V}$-ultra triple.

More generally, we can replace $\mathbb{Z}$ by any integral domain, and
$v_{p}$ by any valuation on this integral domain, and obtain a $\mathbb{V}%
$-ultra triple, where $\mathbb{V}$ is the target of our valuation.
\end{example}

The notion of a $\mathbb{V}$-ultra triple generalizes the notion of an ultra
triple as defined in \cite{GriPet19}. More precisely, if $\mathbb{V}$ is the
additive group $\left(  \mathbb{R},+,0,\leq\right)  $ (with the usual addition
and the usual total order on $\mathbb{R}$), then a $\mathbb{V}$-ultra triple
is the same as what is called an \textquotedblleft ultra
triple\textquotedblright\ in \cite{GriPet19}. Several properties and examples
of ultra triples can be found in \cite{GriPet19}.

It is straightforward to adapt all the definitions and results stated in
\cite{GriPet19} for ultra triples to the more general setting of $\mathbb{V}%
$-ultra triples\footnote{There is one stupid exception: The definition of $R$
in \cite[Remark 8.13]{GriPet19} requires $\mathbb{V}\neq0$. But \cite[Remark
8.13]{GriPet19} is just a tangent without concrete use.}. Let us specifically
extend two definitions from \cite{GriPet19} to $\mathbb{V}$-ultra triples: the
definition of a perimeter (\cite[\S 3.1]{GriPet19}) and the definition of the
Bhargava greedoid (\cite[\S 6.2]{GriPet19}):

\begin{definition}
Let $\left(  E,w,d\right)  $ be a $\mathbb{V}$-ultra triple. Let $A$ be a
finite subset of $E$. Then, the \textit{perimeter} of $A$ (with respect to
$\left(  E,w,d\right)  $) is defined to be
\[
\sum_{a\in A}w\left(  a\right)  +\sum_{\substack{\left\{  a,b\right\}
\subseteq A;\\a\neq b}}d\left(  a,b\right)  \in\mathbb{V}.
\]
(Here, the second sum ranges over all \textbf{unordered} pairs $\left\{
a,b\right\}  $ of distinct elements of $A$.)

The perimeter of $A$ is denoted by $\operatorname*{PER}\left(  A\right)  $.
\end{definition}

For example, if $A=\left\{  p,q,r\right\}  $ is a $3$-element set, then%
\[
\operatorname*{PER}\left(  A\right)  =w\left(  p\right)  +w\left(  q\right)
+w\left(  r\right)  +d\left(  p,q\right)  +d\left(  p,r\right)  +d\left(
q,r\right)  .
\]

\begin{definition}
Let $S$ be any set, and let $k\in\mathbb{N}$. A $k$\textit{-subset} of $S$
means a $k$-element subset of $S$ (that is, a subset of $S$ having size $k$).
\end{definition}

\begin{definition}
\label{def.bhar-greed}Let $\left(  E,w,d\right)  $ be a $\mathbb{V}$-ultra
triple such that $E$ is finite. The \textit{Bhargava greedoid} of $\left(
E,w,d\right)  $ is defined to be the subset%
\begin{align*}
&  \left\{  A\subseteq E\ \mid\ A\text{ has maximum perimeter among all
}\left\vert A\right\vert \text{-subsets of }E\right\} \\
&  =\left\{  A\subseteq E\ \mid\ \operatorname*{PER}\left(  A\right)
\geq\operatorname*{PER}\left(  B\right)  \text{ for all }B\subseteq E\text{
satisfying }\left\vert B\right\vert =\left\vert A\right\vert \right\}
\end{align*}
of $2^{E}$.
\end{definition}

Some examples of Bhargava greedoids can be found in \cite[\S 6.2]{GriPet19}.
Here are two more:

\Needspace{6cm}

\begin{example}
\label{exa.bhargava.1}For this example, let $\mathbb{V}=\mathbb{Z}$ and
$E=\left\{  0,1,2,3,4\right\}  $. Define a map $w:E\rightarrow\mathbb{V}$ by
setting $w\left(  e\right)  =\max\left\{  e,1\right\}  $ for each $e\in E$.
(Thus, $w\left(  0\right)  =1$ and $w\left(  e\right)  =e$ for all $e>0$.)
Define a map $d:E\timesu E\rightarrow\mathbb{V}$ by setting
\[
d\left(  e,f\right)  =\min\left\{  3,\max\left\{  4-e,4-f\right\}  \right\}
\ \ \ \ \ \ \ \ \ \ \text{for all }\left(  e,f\right)  \in E\timesu E.
\]
Here is a table of values of $d$:%
\[%
\begin{tabular}
[c]{c||ccccc}%
$d$ & $0$ & $1$ & $2$ & $3$ & $4$\\\hline\hline
$0$ &  & $3$ & $3$ & $3$ & $3$\\
$1$ & $3$ &  & $3$ & $3$ & $3$\\
$2$ & $3$ & $3$ &  & $2$ & $2$\\
$3$ & $3$ & $3$ & $2$ &  & $1$\\
$4$ & $3$ & $3$ & $2$ & $1$ &
\end{tabular}
\ \ \ \ \ \ .
\]
It is easy to see that $\left(  E,w,d\right)  $ is a $\mathbb{V}$-ultra
triple. Let $\mathcal{F}$ be its Bhargava greedoid. Thus, $\mathcal{F}$
consists of the subsets $A$ of $E$ that have maximum perimeter among all
$\left\vert A\right\vert $-subsets of $E$. What are these subsets?

\begin{itemize}
\item Clearly, $\varnothing$ is the only $\left\vert \varnothing\right\vert
$-subset of $E$, and thus has maximum perimeter among all $\left\vert
\varnothing\right\vert $-subsets of $E$. Hence, $\varnothing\in\mathcal{F}$.

\item The perimeter of a $1$-subset $\left\{  e\right\}  $ of $E$ is just the
weight $w\left(  e\right)  $. Thus, the $1$-subsets of $E$ having maximum
perimeter among all $1$-subsets of $E$ are precisely the subsets $\left\{
e\right\}  $ where $e\in E$ has maximum weight. In our example, there is only
one $e\in E$ having maximum weight, namely $4$. Thus, the only $1$-subset of
$E$ having maximum perimeter among all $1$-subsets of $E$ is $\left\{
4\right\}  $. In other words, the only $1$-element set in $\mathcal{F}$ is
$\left\{  4\right\}  $.

\item What about $2$-element sets in $\mathcal{F}$ ? The perimeter
$\operatorname*{PER}\left\{  e,f\right\}  $ of a $2$-subset $\left\{
e,f\right\}  $ of $E$ is $w\left(  e\right)  +w\left(  f\right)  +d\left(
e,f\right)  $. Thus,%
\[
\operatorname*{PER}\left\{  0,4\right\}  =w\left(  0\right)  +w\left(
4\right)  +d\left(  0,4\right)  =1+4+3=8
\]
and similarly $\operatorname*{PER}\left\{  1,4\right\}  =8$ and
$\operatorname*{PER}\left\{  2,4\right\}  =8$ and $\operatorname*{PER}\left\{
3,4\right\}  =8$ and $\operatorname*{PER}\left\{  0,3\right\}  =7$ and
$\operatorname*{PER}\left\{  1,3\right\}  =7$ and $\operatorname*{PER}\left\{
2,3\right\}  =7$ and $\operatorname*{PER}\left\{  0,2\right\}  =6$ and
$\operatorname*{PER}\left\{  1,2\right\}  =6$ and $\operatorname*{PER}\left\{
0,1\right\}  =5$. Thus, the $2$-subsets of $E$ having maximum perimeter among
all $2$-subsets of $E$ are $\left\{  0,4\right\}  $ and $\left\{  1,4\right\}
$ and $\left\{  2,4\right\}  $ and $\left\{  3,4\right\}  $. So these four
sets are the $2$-element sets in $\mathcal{F}$.

\item Similarly, the $3$-element sets in $\mathcal{F}$ are $\left\{
0,1,4\right\}  $, $\left\{  0,3,4\right\}  $, $\left\{  1,3,4\right\}  $,
$\left\{  0,2,4\right\}  $ and $\left\{  1,2,4\right\}  $. They have perimeter
$15$, while all other $3$-subsets of $E$ have perimeter $14$ or $13$.

\item Similarly, the $4$-element sets in $\mathcal{F}$ are $\left\{
0,1,2,4\right\}  $ and $\left\{  0,1,3,4\right\}  $.

\item Clearly, $E$ is the only $\left\vert E\right\vert $-subset of $E$, and
thus has maximum perimeter among all $\left\vert E\right\vert $-subsets of
$E$. Hence, $E\in\mathcal{F}$.
\end{itemize}

Thus, the Bhargava greedoid of $\left(  E,w,d\right)  $ is%
\begin{align*}
\mathcal{F}  &  =\left\{  \varnothing,\left\{  4\right\}  ,\left\{
0,4\right\}  ,\left\{  1,4\right\}  ,\left\{  2,4\right\}  ,\left\{
3,4\right\}  ,\right. \\
&  \ \ \ \ \ \ \ \ \ \ \left.  \left\{  0,1,4\right\}  ,\left\{
0,3,4\right\}  ,\left\{  1,3,4\right\}  ,\left\{  0,2,4\right\}  ,\left\{
1,2,4\right\}  ,\left\{  0,1,2,4\right\}  ,\left\{  0,1,3,4\right\}
,E\right\}  .
\end{align*}

\end{example}

\begin{example}
\label{exa.bhargava.0}For this example, let $\mathbb{V}=\mathbb{Z}$ and
$E=\left\{  1,2,3\right\}  $. Define a map $w:E\rightarrow\mathbb{V}$ by
setting $w\left(  e\right)  =e$ for each $e\in E$. Define a map $d:E\timesu
E\rightarrow\mathbb{V}$ by setting $d\left(  e,f\right)  =1$ for each $\left(
e,f\right)  \in E\timesu E$. It is easy to see that $\left(  E,w,d\right)  $
is a $\mathbb{V}$-ultra triple. Let $\mathcal{F}$ be the Bhargava greedoid of
$\left(  E,w,d\right)  $. What is $\mathcal{F}$ ?

The same kind of reasoning as in Example \ref{exa.bhargava.1} (but simpler due
to the fact that all values of $d$ are the same) shows that%
\[
\mathcal{F}=\left\{  \varnothing,\left\{  3\right\}  ,\left\{  2,3\right\}
,\left\{  1,2,3\right\}  \right\}  .
\]

\end{example}

One thing we observed in both of these examples is the following simple fact:

\begin{remark}
\label{rmk.bhargava.EinF}Let $\left(  E,w,d\right)  $ be a $\mathbb{V}$-ultra
triple such that $E$ is finite. Let $\mathcal{F}$ be the Bhargava greedoid of
$\left(  E,w,d\right)  $. Then, $E\in\mathcal{F}$.
\end{remark}

\begin{proof}
[Proof of Remark \ref{rmk.bhargava.EinF}.]The set $E$ obviously has maximum
perimeter among all $\left\vert E\right\vert $-subsets of $E$ (since $E$ is
the only $\left\vert E\right\vert $-subset of $E$).

But $\mathcal{F}$ is the Bhargava greedoid of $\left(  E,w,d\right)  $. In
other words,
\[
\mathcal{F}=\left\{  A\subseteq E\ \mid\ A\text{ has maximum perimeter among
all }\left\vert A\right\vert \text{-subsets of }E\right\}
\]
(by Definition \ref{def.bhar-greed}). Hence, $E\in\mathcal{F}$ (since $E$ is a
subset of $E$ that has maximum perimeter among all $\left\vert E\right\vert
$-subsets of $E$). This proves Remark \ref{rmk.bhargava.EinF}.
\end{proof}

\section{\label{sect.thm-bh-geg}The main theorem}

In \cite[Theorem 6.1]{GriPet19}, it was proved that the Bhargava greedoid of
an ultra triple with finite ground set is a strong greedoid\footnote{See
\cite[\S 6.1]{GriPet19} for the definition of strong greedoids.}. More
generally, this holds for any $\mathbb{V}$-ultra triple with finite ground set
(and the same argument can be used to prove this). However, we shall prove a
stronger statement:

\begin{theorem}
\label{thm.bh-geg}Let $\left(  E,w,d\right)  $ be a $\mathbb{V}$-ultra triple
such that $E$ is finite. Let $\mathcal{F}$ be the Bhargava greedoid of
$\left(  E,w,d\right)  $. Let $\mathbb{K}$ be a field of size $\left\vert
\mathbb{K}\right\vert \geq\left\vert E\right\vert $. Then, $\mathcal{F}$ is
the Gaussian elimination greedoid of a vector family over $\mathbb{K}$.
\end{theorem}

We will spend the next few sections working towards a proof of this theorem.
First, however, let us extend it somewhat by strengthening the $\left\vert
\mathbb{K}\right\vert \geq\left\vert E\right\vert $ bound.

\section{\label{sect.clique-stronger}Cliques and stronger bounds}

For the rest of Section \ref{sect.clique-stronger}, we fix a $\mathbb{V}%
$-ultra triple $\left(  E,w,d\right)  $.

Let us define a certain kind of subsets of $E$, which we call \textit{cliques}.

\begin{definition}
\label{def.alpha-clique}Let $\alpha\in\mathbb{V}$. An $\alpha$\textit{-clique}
of $\left(  E,w,d\right)  $ will mean a subset $F$ of $E$ such that any two
distinct elements $a,b\in F$ satisfy $d\left(  a,b\right)  =\alpha$.
\end{definition}

\begin{definition}
\label{def.clique}A \textit{clique} of $\left(  E,w,d\right)  $ will mean a
subset of $E$ that is an $\alpha$-clique for some $\alpha\in\mathbb{V}$.
\end{definition}

Thus, any $1$-element subset of $E$ is a clique (and an $\alpha$-clique for
every $\alpha\in\mathbb{V}$). The same holds for the empty subset. Any
$2$-element subset $\left\{  a,b\right\}  $ of $E$ is a clique and, in fact, a
$d\left(  a,b\right)  $-clique.

Note that the notion of a clique (and of an $\alpha$-clique) depends only on
$E$ and $d$, not on $w$.

\begin{example}
For this example, let $m$, $\mathbb{V}$, $E$, $w$ and $d$ be as in Example
\ref{exa.Vultra.a=bmodm}. Then:

\begin{enumerate}
\item[\textbf{(a)}] The $0$-cliques of $E$ are the subsets of $E$ whose
elements are all mutually congruent modulo $m$.

\item[\textbf{(b)}] The $1$-cliques of $E$ are the subsets of $E$ that have no
two distinct elements congruent to each other modulo $m$. Thus, any $1$-clique
has size $\leq m$ if $m$ is positive.

\item[\textbf{(c)}] If $\alpha\in\mathbb{V}$ is distinct from $0$ and $1$,
then the $\alpha$-cliques of $E$ are the subsets of $E$ having size $\leq1$.
\end{enumerate}
\end{example}

Using the notion of cliques, we can assign a number $\operatorname*{mcs}%
\left(  E,w,d\right)  $ to our $\mathbb{V}$-ultra triple $\left(
E,w,d\right)  $:

\begin{definition}
\label{def.mcs}Let $\operatorname*{mcs}\left(  E,w,d\right)  $ denote the
maximum size of a clique of $\left(  E,w,d\right)  $. (This is well-defined
whenever $E$ is finite, and sometimes even otherwise.)
\end{definition}

Clearly, $\operatorname*{mcs}\left(  E,w,d\right)  \leq\left\vert E\right\vert
$, since any clique of $\left(  E,w,d\right)  $ is a subset of $E$.

\begin{example}
\label{exa.bhargava.mcs.1}Let $\mathbb{V}$, $E$, $w$ and $d$ be as in Example
\ref{exa.bhargava.1}. Then, $\left\{  0,1,2\right\}  $ is a $3$-clique of
$\left(  E,w,d\right)  $ and has size $3$; no larger cliques exist in $\left(
E,w,d\right)  $. Thus, $\operatorname*{mcs}\left(  E,w,d\right)  =3$.
\end{example}

\begin{example}
For this example, let $m$, $\mathbb{V}$, $E$, $w$ and $d$ be as in Example
\ref{exa.Vultra.a=bmodm}. Then:

\begin{enumerate}
\item[\textbf{(a)}] If $m=2$ and $E=\left\{  1,2,3,4,5,6\right\}  $, then
$\operatorname*{mcs}\left(  E,w,d\right)  =3$, due to the $0$-clique $\left\{
1,3,5\right\}  $ having maximum size among all cliques.

\item[\textbf{(b)}] If $m=3$ and $E=\left\{  1,2,3,4,5,6\right\}  $, then
$\operatorname*{mcs}\left(  E,w,d\right)  =3$, due to the $1$-clique $\left\{
1,2,3\right\}  $ having maximum size among all cliques.
\end{enumerate}
\end{example}

We can now state a stronger version of Theorem \ref{thm.bh-geg}:

\begin{theorem}
\label{thm.bh-geg-mcs}Let $\left(  E,w,d\right)  $ be a $\mathbb{V}$-ultra
triple such that $E$ is finite. Let $\mathcal{F}$ be the Bhargava greedoid of
$\left(  E,w,d\right)  $. Let $\mathbb{K}$ be a field of size $\left\vert
\mathbb{K}\right\vert \geq\operatorname*{mcs}\left(  E,w,d\right)  $. Then,
$\mathcal{F}$ is the Gaussian elimination greedoid of a vector family over
$\mathbb{K}$.
\end{theorem}

Theorem \ref{thm.bh-geg-mcs} is stronger than Theorem \ref{thm.bh-geg} because
$\left\vert E\right\vert \geq\operatorname*{mcs}\left(  E,w,d\right)  $.

We shall prove Theorem \ref{thm.bh-geg-mcs} in Section \ref{sect.main-proof}.

\section{The converse direction}

Before that, let us explore the question whether the bound $\left\vert
\mathbb{K}\right\vert \geq\operatorname*{mcs}\left(  E,w,d\right)  $ can be
improved. In an important particular case -- namely, when the map $w$ is
constant\footnote{A map $f:X\rightarrow Y$ between two sets $X$ and $Y$ is
said to be \emph{constant} if all values of $f$ are equal (i.e., if every
$x_{1},x_{2}\in X$ satisfy $f\left(  x_{1}\right)  =f\left(  x_{2}\right)  $).
In particular, if $\left\vert X\right\vert \leq1$, then $f:X\rightarrow Y$ is
automatically constant.} --, it cannot, as the following theorem shows:

\begin{theorem}
\label{thm.converse1}Let $\left(  E,w,d\right)  $ be a $\mathbb{V}$-ultra
triple such that $E$ is finite. Assume that the map $w$ is constant. Let
$\mathcal{F}$ be the Bhargava greedoid of $\left(  E,w,d\right)  $. Let
$\mathbb{K}$ be a field such that $\mathcal{F}$ is the Gaussian elimination
greedoid of a vector family over $\mathbb{K}$. Then, $\left\vert
\mathbb{K}\right\vert \geq\operatorname*{mcs}\left(  E,w,d\right)  $.
\end{theorem}

We shall prove Theorem \ref{thm.converse1} in Section \ref{sect.converse}.

When the map $w$ in a $\mathbb{V}$-ultra triple $\left(  E,w,d\right)  $ is
constant, Theorems \ref{thm.bh-geg-mcs} and \ref{thm.converse1} combined yield
an exact characterization of those fields $\mathbb{K}$ for which the Bhargava
greedoid of $\left(  E,w,d\right)  $ can be represented as the Gaussian
elimination greedoid of a vector family over $\mathbb{K}$: Namely, those
fields are precisely the fields $\mathbb{K}$ of size $\left\vert
\mathbb{K}\right\vert \geq\operatorname*{mcs}\left(  E,w,d\right)  $. When $w$
is not constant, Theorem \ref{thm.bh-geg-mcs} gives a sufficient condition; we
don't know a necessary condition. Here are two examples:

\begin{example}
Let $\mathbb{V}$, $E$, $w$, $d$ and $\mathcal{F}$ be as in Example
\ref{exa.bhargava.1}. Then, $\operatorname*{mcs}\left(  E,w,d\right)  =3$ (as
we saw in Example \ref{exa.bhargava.mcs.1}). Hence, Theorem
\ref{thm.bh-geg-mcs} shows that $\mathcal{F}$ can be represented as the
Gaussian elimination greedoid of a vector family over any field $\mathbb{K}$
of size $\left\vert \mathbb{K}\right\vert \geq3$. This bound on $\left\vert
\mathbb{K}\right\vert $ is optimal, since the Bhargava greedoid $\mathcal{F}$
is not the Gaussian elimination greedoid of any vector family over the
$2$-element field $\mathbb{F}_{2}$. (But this does not follow from Theorem
\ref{thm.converse1}, because $w$ is not constant.)
\end{example}

\begin{example}
Let $\mathbb{V}$, $E$, $w$, $d$ and $\mathcal{F}$ be as in Example
\ref{exa.bhargava.0}. Then, $\operatorname*{mcs}\left(  E,w,d\right)  =3$,
since $E$ itself is a clique. Hence, Theorem \ref{thm.bh-geg-mcs} shows that
$\mathcal{F}$ can be represented as the Gaussian elimination greedoid of a
vector family over any field $\mathbb{K}$ of size $\left\vert \mathbb{K}%
\right\vert \geq3$. However, this bound on $\left\vert \mathbb{K}\right\vert $
is not optimal. Indeed, the Bhargava greedoid $\mathcal{F}$ is the Gaussian
elimination greedoid of the vector family $\left(  v_{e}\right)  _{e\in
E}=\left(  v_{1},v_{2},v_{3}\right)  $ over the field $\mathbb{F}_{2}$, where
$v_{1}=\left(
\begin{array}
[c]{c}%
0\\
0\\
1
\end{array}
\right)  $, $v_{2}=\left(
\begin{array}
[c]{c}%
0\\
1\\
1
\end{array}
\right)  $ and $v_{3}=\left(
\begin{array}
[c]{c}%
1\\
1\\
1
\end{array}
\right)  $.
\end{example}

\begin{question}
Let $\left(  E,w,d\right)  $ be a $\mathbb{V}$-ultra triple such that $E$ is
finite. How to characterize the fields $\mathbb{K}$ for which the Bhargava
greedoid of $\left(  E,w,d\right)  $ is the Gaussian elimination greedoid of a
vector family over $\mathbb{K}$ ? Is there a constant $c\left(  E,w,d\right)
$ such that these fields are precisely the fields of size $\geq c\left(
E,w,d\right)  $ ?
\end{question}

\begin{remark}
Let $E$, $w$, $d$ and $\mathcal{F}$ be as in Theorem \ref{thm.bh-geg}. Let
$\mathbb{K}$ be any field. For each $k\in\mathbb{N}$, let $\mathcal{F}_{k}$ be
the set of all $k$-element sets in $\mathcal{F}$.

If $\mathcal{F}$ is the Gaussian elimination greedoid of a vector family over
$\mathbb{K}$, then each $\mathcal{F}_{k}$ with $k\in\left\{  0,1,\ldots
,\left\vert E\right\vert \right\}  $ is the collection of bases of a
representable matroid on the ground set $E$. (Indeed, this follows from
Proposition \ref{prop.geg.matroid-rep}, since $\mathcal{F}_{k}$ is nonempty.)
But the converse is not true: It can happen that each $\mathcal{F}_{k}$ with
$k\in\left\{  0,1,\ldots,\left\vert E\right\vert \right\}  $ is the collection
of bases of a representable matroid on the ground set $E$, yet $\mathcal{F}$
is not the Gaussian elimination greedoid of a vector family over $\mathbb{K}$.
For example, this happens if $E=\left\{  1,2,3\right\}  $ and both maps $w$
and $d$ are constant (so that $\mathcal{F}=2^{E}$), and $\mathbb{K}%
=\mathbb{F}_{2}$.
\end{remark}

\section{\label{sect.p-adic}Valadic $\mathbb{V}$-ultra triples}

As a first step towards the proof of Theorem \ref{thm.bh-geg-mcs}, we will
next introduce a special kind of $\mathbb{V}$-ultra triples which, in a way,
are similar to Bhargava's for integers (see \cite[Example 2.5 and
\S 9]{GriPet19}). We will call them \textit{valadic}\footnote{The name is a
homage to the notion of a valuation ring, which is latent in the argument that
follows (although never used explicitly). Indeed, if we define the notion of a
valuation ring as in \cite[Exercise 11.1]{Eisenb95}, then the $\mathbb{K}%
$-algebra $\mathbb{L}_{+}$ constructed below is an instance of a valuation
ring (with $\mathbb{L}$ being its fraction field, and $\operatorname*{ord}%
:\mathbb{L}\setminus\left\{  0\right\}  \rightarrow\mathbb{V}$ being its
valuation), and many of its properties that will be used below are instances
of general properties of valuation rings. If we extended our argument to the
more general setting of valuation rings, we would also recover Bhargava's
original ultra triples based on integer divisibility (see \cite[Example 2.5
and \S 9]{GriPet19}). However, we have no need for this generality (as we only
need the construction as a stepping stone towards our proof of Theorem
\ref{thm.bh-geg-mcs}), and prefer to remain elementary and self-contained.},
and we will see (in Theorem \ref{thm.poly-ut-greed}) that they satisfy Theorem
\ref{thm.bh-geg}. Afterwards (in Theorem \ref{thm.approx-p}), we will prove
that any $\mathbb{V}$-ultra triple with finite ground set is isomorphic (in an
appropriate sense) to a valadic one over a sufficiently large field. Combining
these two facts, we will then readily obtain Theorem \ref{thm.bh-geg-mcs}.

Recall that $\mathbb{V}$ is a totally ordered abelian group (see Convention
\ref{conv.V} for details). Let us introduce some further notations that will
be used throughout Section \ref{sect.p-adic}.

\begin{definition}
We fix a field $\mathbb{K}$. Let $\mathbb{K}\left[  \mathbb{V}\right]  $
denote the group algebra of the group $\mathbb{V}$ over $\mathbb{K}$. This is
a free $\mathbb{K}$-module with basis $\left(  t_{\alpha}\right)  _{\alpha
\in\mathbb{V}}$; it becomes a $\mathbb{K}$-algebra with unity $t_{0}$ and with
multiplication determined by%
\[
t_{\alpha}t_{\beta}=t_{\alpha+\beta}\ \ \ \ \ \ \ \ \ \ \text{for all }%
\alpha,\beta\in\mathbb{V}.
\]
This group algebra $\mathbb{K}\left[  \mathbb{V}\right]  $ is commutative,
since the group $\mathbb{V}$ is abelian.

Let $\mathbb{V}_{\geq0}$ be the set of all $\alpha\in\mathbb{V}$ satisfying
$\alpha\geq0$; this is a submonoid of the group $\mathbb{V}$. Let
$\mathbb{K}\left[  \mathbb{V}_{\geq0}\right]  $ be the monoid algebra of this
monoid $\mathbb{V}_{\geq0}$ over $\mathbb{K}$. This is a $\mathbb{K}$-algebra
defined in the same way as $\mathbb{K}\left[  \mathbb{V}\right]  $, but using
$\mathbb{V}_{\geq0}$ instead of $\mathbb{V}$. It is clear that $\mathbb{K}%
\left[  \mathbb{V}_{\geq0}\right]  $ is the $\mathbb{K}$-subalgebra of
$\mathbb{K}\left[  \mathbb{V}\right]  $ spanned by the basis elements
$t_{\alpha}$ with $\alpha\in\mathbb{V}_{\geq0}$.
\end{definition}

\begin{example}
\label{exa.ord.V=Z}If $\mathbb{V}=\mathbb{Z}$ (with the usual addition and
total order), then $\mathbb{V}_{\geq0}=\mathbb{N}$. In this case, the group
algebra $\mathbb{K}\left[  \mathbb{V}\right]  $ is the Laurent polynomial ring
$\mathbb{K}\left[  X,X^{-1}\right]  $ in a single indeterminate $X$ over
$\mathbb{K}$ (indeed, $t_{1}$ plays the role of $X$, and more generally, each
$t_{\alpha}$ plays the role of $X^{\alpha}$), and its subalgebra
$\mathbb{K}\left[  \mathbb{V}_{\geq0}\right]  $ is the polynomial ring
$\mathbb{K}\left[  X\right]  $.
\end{example}

This example should be regarded as a guide; even in the general case (where
$\mathbb{V}$ does not have to be $\mathbb{Z}$), the reader cannot go wrong
thinking of $\mathbb{K}\left[  \mathbb{V}\right]  $ as a generalized Laurent
polynomial ring and of $\mathbb{K}\left[  \mathbb{V}_{\geq0}\right]  $ as a
generalized polynomial ring (in a single indeterminate) and of $t_{\alpha}$ as
a generalized monomial $X^{\alpha}$. This analogy shall clarify much of what follows.

\begin{definition}
\ 

\begin{enumerate}
\item[\textbf{(a)}] Let $\mathbb{L}$ be the commutative $\mathbb{K}$-algebra
$\mathbb{K}\left[  \mathbb{V}\right]  $, and let $\mathbb{L}_{+}$ be its
$\mathbb{K}$-subalgebra $\mathbb{K}\left[  \mathbb{V}_{\geq0}\right]  $. Thus,
the $\mathbb{K}$-module $\mathbb{L}$ has basis $\left(  t_{\alpha}\right)
_{\alpha\in\mathbb{V}}$, while its $\mathbb{K}$-submodule $\mathbb{L}_{+}$ has
basis $\left(  t_{\alpha}\right)  _{\alpha\in\mathbb{V}_{\geq0}}$.

\item[\textbf{(b)}] If $a\in\mathbb{L}$ and $\beta\in\mathbb{V}$, then
$\left[  t_{\beta}\right]  a$ shall denote the coefficient of $t_{\beta}$ in
$a$ (when $a$ is expanded in the basis $\left(  t_{\alpha}\right)  _{\alpha
\in\mathbb{V}}$ of $\mathbb{L}$). This is an element of $\mathbb{K}$. For
example, $\left[  t_{3}\right]  \left(  t_{2}-t_{3}+5t_{6}\right)  =-1$ (if
$\mathbb{V}=\mathbb{Z}$).

\item[\textbf{(c)}] If $a\in\mathbb{L}$ is nonzero, then the \textit{order} of
$a$ is defined to be the smallest $\beta\in\mathbb{V}$ such that $\left[
t_{\beta}\right]  a\neq0$. This order is an element of $\mathbb{V}$, and is
denoted by $\operatorname*{ord}a$. For example, $\operatorname*{ord}\left(
t_{2}-t_{3}+5t_{6}\right)  =2$ (if $\mathbb{V}=\mathbb{Z}$). Note that
$\operatorname*{ord}\left(  t_{\alpha}\right)  =\alpha$ for each $\alpha
\in\mathbb{V}$.
\end{enumerate}
\end{definition}

The notations we just defined generalize standard features of Laurent
polynomials: If $\mathbb{V}=\mathbb{Z}$ as in Example \ref{exa.ord.V=Z}, then
the coefficient $\left[  t_{\beta}\right]  a$ of an element $a\in
\mathbb{L}=\mathbb{K}\left[  X,X^{-1}\right]  $ is simply the coefficient of
$X^{\beta}$ in the Laurent polynomial $a$, and the order $\operatorname*{ord}%
a$ of a nonzero Laurent polynomial $a\in\mathbb{L}$ is the order of $a$ in the
usual sense (i.e., the smallest exponent of a monomial appearing in $a$). If
we substitute $X^{-1}$ for $X$ (thus replacing each monomial $X^{\beta}$ by
$X^{-\beta}$), then the order of a Laurent polynomial becomes its degree (with
a negative sign). In light of this, the following properties of orders should
not be surprising:

\begin{lemma}
\label{lem.ord.triangle}\ 

\begin{enumerate}
\item[\textbf{(a)}] A nonzero element $a\in\mathbb{L}$ belongs to
$\mathbb{L}_{+}$ if and only if its order $\operatorname*{ord}a$ is
nonnegative (i.e., we have $\operatorname*{ord}a\geq0$).

\item[\textbf{(b)}] We have $\operatorname*{ord}\left(  -a\right)
=\operatorname*{ord}a$ for any nonzero $a\in\mathbb{L}$.

\item[\textbf{(c)}] Let $a$ and $b$ be two nonzero elements of $\mathbb{L}$.
Then, $ab$ is nonzero and satisfies $\operatorname*{ord}\left(  ab\right)
=\operatorname*{ord}a+\operatorname*{ord}b$.

\item[\textbf{(d)}] Let $a$ and $b$ be two nonzero elements of $\mathbb{L}$
such that $a+b$ is nonzero. Then, $\operatorname*{ord}\left(  a+b\right)
\geq\min\left\{  \operatorname*{ord}a,\operatorname*{ord}b\right\}  $.
\end{enumerate}
\end{lemma}

See Section \ref{sect.ord-proofs} for the (straightforward) proof of this
lemma. (This proof is entirely analogous to the proof of the corresponding
properties of usual polynomials.)

\begin{corollary}
\label{cor.ord.intdom}The ring $\mathbb{L}$ is an integral domain.
\end{corollary}

\begin{proof}
[Proof of Corollary \ref{cor.ord.intdom}.]This follows from Lemma
\ref{lem.ord.triangle} \textbf{(c)}.
\end{proof}

Applying Lemma \ref{lem.ord.triangle} \textbf{(c)} many times, we also obtain
the following:

\begin{corollary}
\label{cor.ord.prod-to-sum}The map $\operatorname*{ord}:\mathbb{L}%
\setminus\left\{  0\right\}  \rightarrow\mathbb{V}$ transforms (finite)
products into sums. In more detail: If $\left(  a_{i}\right)  _{i\in I}$ is
any finite family of nonzero elements of $\mathbb{L}$, then the product
$\prod_{i\in I}a_{i}$ is nonzero and satisfies%
\[
\operatorname*{ord}\left(  \prod_{i\in I}a_{i}\right)  =\sum_{i\in
I}\operatorname*{ord}\left(  a_{i}\right)  .
\]

\end{corollary}

\begin{proof}
Induction on $\left\vert I\right\vert $. The induction step uses Lemma
\ref{lem.ord.triangle} \textbf{(c)}; the straightforward details are left to
the reader.
\end{proof}

We can now assign a $\mathbb{V}$-ultra triple to each subset of $\mathbb{L}$:

\begin{definition}
\label{def.padic-ut}Let $E$ be a subset of $\mathbb{L}$. Define a distance
function $d:E\timesu E\rightarrow\mathbb{V}$ by setting%
\[
d\left(  a,b\right)  =-\operatorname*{ord}\left(  a-b\right)
\ \ \ \ \ \ \ \ \ \ \text{for all }\left(  a,b\right)  \in E\timesu E.
\]
(Recall that $E\timesu E$ means the set $\left\{  \left(  a,b\right)  \in
E\times E\ \mid\ a\neq b\right\}  $.)

Then, $\left(  E,w,d\right)  $ is a $\mathbb{V}$-ultra triple whenever
$w:E\rightarrow\mathbb{V}$ is a function (by Lemma \ref{lem.valadic-ut.wd}
below). Such a $\mathbb{V}$-ultra triple $\left(  E,w,d\right)  $ will be
called \textit{valadic}.
\end{definition}

\begin{lemma}
\label{lem.valadic-ut.wd}In Definition \ref{def.padic-ut}, the triple $\left(
E,w,d\right)  $ is indeed a $\mathbb{V}$-ultra triple.
\end{lemma}

Lemma \ref{lem.valadic-ut.wd} follows easily from Lemma \ref{lem.ord.triangle}%
. (See Section \ref{sect.ord-proofs} for the details of the proof.)

Now, we claim that the Bhargava greedoid of a valadic $\mathbb{V}$-ultra
triple $\left(  E,w,d\right)  $ with finite $E$ is the Gaussian elimination
greedoid of a vector family over $\mathbb{K}$:

\begin{theorem}
\label{thm.poly-ut-greed}Let $E$ be a finite subset of $\mathbb{L}$. Define
$d$ as in Definition \ref{def.padic-ut}. Let $w:E\rightarrow\mathbb{V}$ be a
function. Then, the Bhargava greedoid of the $\mathbb{V}$-ultra triple
$\left(  E,w,d\right)  $ is the Gaussian elimination greedoid of a vector
family over $\mathbb{K}$.
\end{theorem}

In order to prove this theorem, we will need a determinantal identity:

\begin{lemma}
\label{lem.vandermonde.monic}Let $R$ be a commutative ring. Consider the
polynomial ring $R\left[  X\right]  $. Let $m\in\mathbb{N}$. Let $f_{1}%
,f_{2},\ldots,f_{m}$ be $m$ polynomials in $R\left[  X\right]  $. Assume that
$f_{j}$ is a monic polynomial of degree $j-1$ for each $j\in\left\{
1,2,\ldots,m\right\}  $. Let $u_{1},u_{2},\ldots,u_{m}$ be $m$ elements of
$R$. Then,%
\[
\det\left(  \left(  f_{j}\left(  u_{i}\right)  \right)  _{1\leq i\leq
m,\ 1\leq j\leq m}\right)  =\prod_{\substack{\left(  i,j\right)  \in\left\{
1,2,\ldots,m\right\}  ^{2};\\i>j}}\left(  u_{i}-u_{j}\right)  .
\]

\end{lemma}

\noindent Here, we are using the notation $\left(  b_{i,j}\right)  _{1\leq
i\leq p,\ 1\leq j\leq q}$ for the $p\times q$-matrix whose $\left(
i,j\right)  $-th entry is $b_{i,j}$ for all $i\in\left\{  1,2,\ldots
,p\right\}  $ and all $j\in\left\{  1,2,\ldots,q\right\}  $.

Lemma \ref{lem.vandermonde.monic} is a classical generalization of the famous
Vandermonde determinant. In this form, it is a particular case of
\cite[Theorem 2]{hyperfact} (applied to $P_{j}=f_{j}$ and $a_{i}=u_{i}$),
because the coefficient of $X^{j-1}$ in a monic polynomial of degree $j-1$ is
$1$. It also appears in \cite[Exercise 267]{FadSom72} (where it is stated for
the transpose of the matrix we are considering here), in \cite[Chapter
XI,\ Exercise 2 in Set XVIII]{Muir60} (where it, too, is stated for the
transpose of the matrix), in \cite[Proposition 1]{Kratte99}, and in
\cite[Exercise 6.62]{detnotes}.

We need two more simple lemmas for our proof of Theorem
\ref{thm.poly-ut-greed}:

\begin{lemma}
\label{lem.piLK}The map%
\begin{align*}
\pi:\mathbb{L}_{+}  &  \rightarrow\mathbb{K},\\
x  &  \mapsto\left[  t_{0}\right]  x
\end{align*}
is a $\mathbb{K}$-algebra homomorphism.
\end{lemma}

\begin{lemma}
\label{lem.pinot0}Consider the map $\pi:\mathbb{L}_{+}\rightarrow\mathbb{K}$
from Lemma \ref{lem.piLK}. Let $a\in\mathbb{L}_{+}$ be nonzero. Then,
$\pi\left(  a\right)  \neq0$ holds if and only if $\operatorname*{ord}a=0$.
\end{lemma}

See Section \ref{sect.ord-proofs} for the (easy) proofs of these two lemmas.

\begin{proof}
[Proof of Theorem \ref{thm.poly-ut-greed}.]Let $m=\left\vert E\right\vert $.
Consider the $\mathbb{V}$-ultra triple $\left(  E,w,d\right)  $; all
perimeters discussed in this proof are defined with respect to this
$\mathbb{V}$-ultra triple.

We construct a list $\left(  c_{1},c_{2},\ldots,c_{m}\right)  $ of elements of
$E$ by the following recursive procedure:

\begin{itemize}
\item For each $i\in\left\{  1,2,\ldots,m\right\}  $, we choose $c_{i}$
(assuming that all the preceding entries $c_{1},c_{2},\ldots,c_{i-1}$ of our
list are already constructed) to be an element of $E\setminus\left\{
c_{1},c_{2},\ldots,c_{i-1}\right\}  $ that maximizes the perimeter
$\operatorname*{PER}\left\{  c_{1},c_{2},\ldots,c_{i}\right\}  $.
\end{itemize}

\noindent This procedure can indeed be carried out, since at each step we can
find an element $c_{i}\in E\setminus\left\{  c_{1},c_{2},\ldots,c_{i-1}%
\right\}  $ that maximizes the perimeter $\operatorname{PER}\left\{
c_{1},c_{2},\ldots,c_{i}\right\}  $.\ \ \ \ \footnote{Indeed, the set
$E\setminus\left\{  c_{1},c_{2},\ldots,c_{i-1}\right\}  $ is nonempty (since
$\left\vert \left\{  c_{1},c_{2},\ldots,c_{i-1}\right\}  \right\vert \leq
i-1<i\leq m=\left\vert E\right\vert $ and thus $\left\{  c_{1},c_{2}%
,\ldots,c_{i-1}\right\}  \not \supseteq E$) and finite (since $E$ is finite),
and thus at least one of its elements will maximize the perimeter in
question.} Clearly, this procedure constructs an $m$-tuple $\left(
c_{1},c_{2},\ldots,c_{m}\right)  $ of elements of $E$. The $m$ entries
$c_{1},c_{2},\ldots,c_{m}$ of this $m$-tuple are distinct\footnote{since each
$c_{i}$ is chosen to be an element of $E\setminus\left\{  c_{1},c_{2}%
,\ldots,c_{i-1}\right\}  $, and thus is distinct from all the preceding
entries $c_{1},c_{2},\ldots,c_{i-1}$}, and thus are $m$ distinct elements of
$E$; but $E$ has only $m$ elements altogether (since $m=\left\vert
E\right\vert $). Hence, the $m$ entries $c_{1},c_{2},\ldots,c_{m}$ must cover
the whole set $E$. In other words, $E=\left\{  c_{1},c_{2},\ldots
,c_{m}\right\}  $.

Furthermore, for each $i\in\left\{  1,2,\ldots,m\right\}  $ and each $x\in
E\setminus\left\{  c_{1},c_{2},\ldots,c_{i-1}\right\}  $, we have
\begin{equation}
\operatorname*{PER}\left\{  c_{1},c_{2},\ldots,c_{i}\right\}  \geq
\operatorname*{PER}\left\{  c_{1},c_{2},\ldots,c_{i-1},x\right\}
\label{pf.thm.poly-ut-greed.greedy}%
\end{equation}
(due to how $c_{i}$ is chosen). Thus, in the parlance of \cite[\S 3.2]%
{GriPet19}, the $m$-tuple $\left(  c_{1},c_{2},\ldots,c_{m}\right)  $ is a
greedy $m$-permutation of $E$.

For each $j\in\left\{  1,2,\ldots,m\right\}  $, define a $\rho_{j}%
\in\mathbb{V}$ by%
\begin{equation}
\rho_{j}=w\left(  c_{j}\right)  +\sum_{i=1}^{j-1}d\left(  c_{i},c_{j}\right)
. \label{pf.thm.poly-ut-greed.rhoj=}%
\end{equation}
(This is precisely what is called $\nu_{j}^{\circ}\left(  C\right)  $ in
\cite{GriPet19}, where $C=E$.)

Consider the polynomial ring $\mathbb{L}\left[  X\right]  $. For each
$j\in\left\{  1,2,\ldots,m\right\}  $, define a polynomial $f_{j}\in
\mathbb{L}\left[  X\right]  $ by
\[
f_{j}=\left(  X-c_{1}\right)  \left(  X-c_{2}\right)  \cdots\left(
X-c_{j-1}\right)  =\prod_{i=1}^{j-1}\left(  X-c_{i}\right)  .
\]
This is a monic polynomial of degree $j-1$.

Next we claim the following:

\begin{statement}
\textit{Claim 1:} Let $e\in E$ and $j\in\left\{  1,2,\ldots,m\right\}  $.
Then, $t_{\rho_{j}-w\left(  e\right)  }f_{j}\left(  e\right)  \in
\mathbb{L}_{+}$.
\end{statement}

[\textit{Proof of Claim 1:} We have $f_{j}=\prod_{i=1}^{j-1}\left(
X-c_{i}\right)  $ and thus $f_{j}\left(  e\right)  =\prod_{i=1}^{j-1}\left(
e-c_{i}\right)  $. Hence, if $e\in\left\{  c_{1},c_{2},\ldots,c_{j-1}\right\}
$, then $f_{j}\left(  e\right)  =0$ and thus our claim $t_{\rho_{j}-w\left(
e\right)  }f_{j}\left(  e\right)  \in\mathbb{L}_{+}$ is obvious. Thus, we WLOG
assume that $e\notin\left\{  c_{1},c_{2},\ldots,c_{j-1}\right\}  $. Thus, each
$i\in\left\{  1,2,\ldots,j-1\right\}  $ satisfies $e\neq c_{i}$ and thus
$e-c_{i}\neq0$. Hence, $\prod_{i=1}^{j-1}\left(  e-c_{i}\right)  $ is a
product of nonzero elements of $\mathbb{L}$, and thus is itself nonzero (since
Corollary \ref{cor.ord.intdom} says that $\mathbb{L}$ is an integral domain).
In other words, $f_{j}\left(  e\right)  $ is nonzero (since $f_{j}\left(
e\right)  =\prod_{i=1}^{j-1}\left(  e-c_{i}\right)  $). Hence, $t_{\rho
_{j}-w\left(  e\right)  }f_{j}\left(  e\right)  $ is nonzero as well (since
$t_{\rho_{j}-w\left(  e\right)  }$ is nonzero, and since $\mathbb{L}$ is an
integral domain).

Moreover, from $f_{j}\left(  e\right)  =\prod_{i=1}^{j-1}\left(
e-c_{i}\right)  $, we obtain%
\begin{equation}
\operatorname*{ord}\left(  f_{j}\left(  e\right)  \right)
=\operatorname*{ord}\left(  \prod_{i=1}^{j-1}\left(  e-c_{i}\right)  \right)
=\sum_{i=1}^{j-1}\operatorname*{ord}\left(  e-c_{i}\right)
\label{pf.thm.poly-ut-greed.c1.pf.ordfj}%
\end{equation}
(by Corollary \ref{cor.ord.prod-to-sum}).

From $e\in E$ and $e\notin\left\{  c_{1},c_{2},\ldots,c_{j-1}\right\}  $, we
obtain $e\in E\setminus\left\{  c_{1},c_{2},\ldots,c_{j-1}\right\}  $. Hence,
(\ref{pf.thm.poly-ut-greed.greedy}) (applied to $i=j$ and $x=e$) yields%
\begin{equation}
\operatorname*{PER}\left\{  c_{1},c_{2},\ldots,c_{j}\right\}  \geq
\operatorname*{PER}\left\{  c_{1},c_{2},\ldots,c_{j-1},e\right\}  .
\label{pf.thm.poly-ut-greed.c1.pf.1}%
\end{equation}
But $c_{1},c_{2},\ldots,c_{j}$ are distinct\footnote{This is because
$c_{1},c_{2},\ldots,c_{m}$ are distinct.}. Hence, the definition of the
perimeter yields%
\begin{align*}
\operatorname*{PER}\left\{  c_{1},c_{2},\ldots,c_{j}\right\}   &
=\underbrace{\sum_{i=1}^{j}w\left(  c_{i}\right)  }_{=\sum_{i=1}^{j-1}w\left(
c_{i}\right)  +w\left(  c_{j}\right)  }+\underbrace{\sum_{1\leq i<p\leq
j}d\left(  c_{i},c_{p}\right)  }_{=\sum_{1\leq i<p\leq j-1}d\left(
c_{i},c_{p}\right)  +\sum_{i=1}^{j-1}d\left(  c_{i},c_{j}\right)  }\\
&  =\sum_{i=1}^{j-1}w\left(  c_{i}\right)  +w\left(  c_{j}\right)
+\sum_{1\leq i<p\leq j-1}d\left(  c_{i},c_{p}\right)  +\sum_{i=1}%
^{j-1}d\left(  c_{i},c_{j}\right) \\
&  =\underbrace{w\left(  c_{j}\right)  +\sum_{i=1}^{j-1}d\left(  c_{i}%
,c_{j}\right)  }_{\substack{=\rho_{j}\\\text{(by
(\ref{pf.thm.poly-ut-greed.rhoj=}))}}}+\sum_{i=1}^{j-1}w\left(  c_{i}\right)
+\sum_{1\leq i<p\leq j-1}d\left(  c_{i},c_{p}\right) \\
&  =\rho_{j}+\sum_{i=1}^{j-1}w\left(  c_{i}\right)  +\sum_{1\leq i<p\leq
j-1}d\left(  c_{i},c_{p}\right)
\end{align*}
and%
\[
\operatorname*{PER}\left\{  c_{1},c_{2},\ldots,c_{j-1},e\right\}  =\sum
_{i=1}^{j-1}w\left(  c_{i}\right)  +w\left(  e\right)  +\sum_{1\leq i<p\leq
j-1}d\left(  c_{i},c_{p}\right)  +\sum_{i=1}^{j-1}d\left(  c_{i},e\right)
\]
(since $c_{1},c_{2},\ldots,c_{j-1},e$ are distinct\footnote{This is because
$c_{1},c_{2},\ldots,c_{m}$ are distinct and $e\notin\left\{  c_{1}%
,c_{2},\ldots,c_{j-1}\right\}  $.}). Hence,
(\ref{pf.thm.poly-ut-greed.c1.pf.1}) rewrites as%
\begin{align*}
&  \rho_{j}+\sum_{i=1}^{j-1}w\left(  c_{i}\right)  +\sum_{1\leq i<p\leq
j-1}d\left(  c_{i},c_{p}\right) \\
&  \geq\sum_{i=1}^{j-1}w\left(  c_{i}\right)  +w\left(  e\right)  +\sum_{1\leq
i<p\leq j-1}d\left(  c_{i},c_{p}\right)  +\sum_{i=1}^{j-1}d\left(
c_{i},e\right)  .
\end{align*}
After cancelling equal terms, this inequality transforms into%
\[
\rho_{j}\geq w\left(  e\right)  +\sum_{i=1}^{j-1}d\left(  c_{i},e\right)  .
\]
In view of%
\[
\sum_{i=1}^{j-1}\underbrace{d\left(  c_{i},e\right)  }_{\substack{=d\left(
e,c_{i}\right)  \\\text{(by the \textquotedblleft Symmetry\textquotedblright%
}\\\text{axiom in the definition}\\\text{of a }\mathbb{V}\text{-ultra
triple)}}}=\sum_{i=1}^{j-1}\underbrace{d\left(  e,c_{i}\right)  }%
_{\substack{=-\operatorname*{ord}\left(  e-c_{i}\right)  \\\text{(by the
definition of }d\text{)}}}=-\underbrace{\sum_{i=1}^{j-1}\operatorname*{ord}%
\left(  e-c_{i}\right)  }_{\substack{=\operatorname*{ord}\left(  f_{j}\left(
e\right)  \right)  \\\text{(by (\ref{pf.thm.poly-ut-greed.c1.pf.ordfj}))}%
}}=-\operatorname*{ord}\left(  f_{j}\left(  e\right)  \right)  ,
\]
this rewrites as%
\[
\rho_{j}\geq w\left(  e\right)  -\operatorname*{ord}\left(  f_{j}\left(
e\right)  \right)  .
\]
In other words, $\operatorname*{ord}\left(  f_{j}\left(  e\right)  \right)
\geq w\left(  e\right)  -\rho_{j}$. Now, Lemma \ref{lem.ord.triangle}
\textbf{(c)} (applied to $a=t_{\rho_{j}-w\left(  e\right)  }$ and
$b=f_{j}\left(  e\right)  $) yields%
\begin{align*}
\operatorname*{ord}\left(  t_{\rho_{j}-w\left(  e\right)  }f_{j}\left(
e\right)  \right)   &  =\underbrace{\operatorname*{ord}\left(  t_{\rho
_{j}-w\left(  e\right)  }\right)  }_{=\rho_{j}-w\left(  e\right)
}+\underbrace{\operatorname*{ord}\left(  f_{j}\left(  e\right)  \right)
}_{\geq w\left(  e\right)  -\rho_{j}}\\
&  \geq\rho_{j}-w\left(  e\right)  +w\left(  e\right)  -\rho_{j}=0.
\end{align*}
Hence, Lemma \ref{lem.ord.triangle} \textbf{(a)} (applied to $a=t_{\rho
_{j}-w\left(  e\right)  }f_{j}\left(  e\right)  $) shows that $t_{\rho
_{j}-w\left(  e\right)  }f_{j}\left(  e\right)  $ belongs to $\mathbb{L}_{+}$.
Thus, $t_{\rho_{j}-w\left(  e\right)  }f_{j}\left(  e\right)  \in
\mathbb{L}_{+}$. This proves Claim 1.] \medskip

For each $e\in E$ and $j\in\left\{  1,2,\ldots,m\right\}  $, we define an
$a\left(  e,j\right)  \in\mathbb{L}_{+}$ by%
\begin{equation}
a\left(  e,j\right)  =t_{\rho_{j}-w\left(  e\right)  }f_{j}\left(  e\right)  .
\label{pf.thm.poly-ut-greed.alpha=}%
\end{equation}
(This is well-defined, due to Claim 1.)

We now claim the following:

\begin{statement}
\textit{Claim 2:} Let $k\in\mathbb{N}$. Let $u_{1},u_{2},\ldots,u_{k}$ be any
$k$ distinct elements of $E$. Let $U=\left\{  u_{1},u_{2},\ldots
,u_{k}\right\}  $. Then,
\begin{equation}
\det\left(  \left(  a\left(  u_{i},j\right)  \right)  _{1\leq j\leq k,\ 1\leq
i\leq k}\right)  \text{ is a nonzero element of }\mathbb{L}_{+}
\label{pf.thm.poly-ut-greed.c2.1}%
\end{equation}
and%
\begin{equation}
\operatorname*{ord}\left(  \det\left(  \left(  a\left(  u_{i},j\right)
\right)  _{1\leq j\leq k,\ 1\leq i\leq k}\right)  \right)  =\sum_{j=1}^{k}%
\rho_{j}-\operatorname*{PER}\left(  U\right)  .
\label{pf.thm.poly-ut-greed.c2.2}%
\end{equation}

\end{statement}

[\textit{Proof of Claim 2:} The set $E$ has at least $k$ many elements (since
$u_{1},u_{2},\ldots,u_{k}$ are $k$ distinct elements of $E$). In other words,
$\left\vert E\right\vert \geq k$. Hence, $k\leq\left\vert E\right\vert =m$.
Therefore, $\left\{  1,2,\ldots,k\right\}  \subseteq\left\{  1,2,\ldots
,m\right\}  $. In other words, for each $j\in\left\{  1,2,\ldots,k\right\}  $,
we have $j\in\left\{  1,2,\ldots,m\right\}  $.

Hence, $a\left(  u_{i},j\right)  \in\mathbb{L}_{+}$ for any $i,j\in\left\{
1,2,\ldots,k\right\}  $ (since we defined $a\left(  e,j\right)  $ to satisfy
$a\left(  e,j\right)  \in\mathbb{L}_{+}$ for any $e\in E$ and $j\in\left\{
1,2,\ldots,m\right\}  $). In other words, all entries of the matrix $\left(
a\left(  u_{i},j\right)  \right)  _{1\leq j\leq k,\ 1\leq i\leq k}$ belong to
$\mathbb{L}_{+}$. Hence, its determinant $\det\left(  \left(  a\left(
u_{i},j\right)  \right)  _{1\leq j\leq k,\ 1\leq i\leq k}\right)  $ belongs to
$\mathbb{L}_{+}$ as well (since $\mathbb{L}_{+}$ is a ring).

Lemma \ref{lem.vandermonde.monic} (applied to $\mathbb{L}$ and $k$ instead of
$R$ and $m$) yields%
\begin{equation}
\det\left(  \left(  f_{j}\left(  u_{i}\right)  \right)  _{1\leq i\leq
k,\ 1\leq j\leq k}\right)  =\prod_{\substack{\left(  i,j\right)  \in\left\{
1,2,\ldots,k\right\}  ^{2};\\i>j}}\left(  u_{i}-u_{j}\right)  .
\label{pf.thm.poly-ut-greed.vand}%
\end{equation}

It is known that the determinant of a matrix equals the determinant of its
transpose. Thus,%
\[
\det\left(  \left(  f_{j}\left(  u_{i}\right)  \right)  _{1\leq j\leq
k,\ 1\leq i\leq k}\right)  =\det\left(  \left(  f_{j}\left(  u_{i}\right)
\right)  _{1\leq i\leq k,\ 1\leq j\leq k}\right)
\]
(since the matrix $\left(  f_{j}\left(  u_{i}\right)  \right)  _{1\leq j\leq
k,\ 1\leq i\leq k}$ is the transpose of the matrix $\left(  f_{j}\left(
u_{i}\right)  \right)  _{1\leq i\leq k,\ 1\leq j\leq k}$).

But when we scale a column of a matrix by a scalar $\lambda$, then its
determinant also gets multiplied by $\lambda$. Hence,%
\begin{align*}
\det\left(  \left(  t_{-w\left(  u_{i}\right)  }f_{j}\left(  u_{i}\right)
\right)  _{1\leq j\leq k,\ 1\leq i\leq k}\right)   &  =\left(  \prod_{i=1}%
^{k}t_{-w\left(  u_{i}\right)  }\right)  \cdot\underbrace{\det\left(  \left(
f_{j}\left(  u_{i}\right)  \right)  _{1\leq j\leq k,\ 1\leq i\leq k}\right)
}_{\substack{=\det\left(  \left(  f_{j}\left(  u_{i}\right)  \right)  _{1\leq
i\leq k,\ 1\leq j\leq k}\right)  \\=\prod_{\substack{\left(  i,j\right)
\in\left\{  1,2,\ldots,k\right\}  ^{2};\\i>j}}\left(  u_{i}-u_{j}\right)
\\\text{(by (\ref{pf.thm.poly-ut-greed.vand}))}}}\\
&  =\left(  \prod_{i=1}^{k}t_{-w\left(  u_{i}\right)  }\right)  \cdot
\prod_{\substack{\left(  i,j\right)  \in\left\{  1,2,\ldots,k\right\}
^{2};\\i>j}}\left(  u_{i}-u_{j}\right)  .
\end{align*}
Furthermore, when we scale a row of a matrix by a scalar $\lambda$, then its
determinant also gets multiplied by $\lambda$. Hence,%
\begin{align*}
&  \det\left(  \left(  t_{\rho_{j}}t_{-w\left(  u_{i}\right)  }f_{j}\left(
u_{i}\right)  \right)  _{1\leq j\leq k,\ 1\leq i\leq k}\right) \\
&  =\left(  \prod_{j=1}^{k}t_{\rho_{j}}\right)  \cdot\underbrace{\det\left(
\left(  t_{-w\left(  u_{i}\right)  }f_{j}\left(  u_{i}\right)  \right)
_{1\leq j\leq k,\ 1\leq i\leq k}\right)  }_{=\left(  \prod_{i=1}%
^{k}t_{-w\left(  u_{i}\right)  }\right)  \cdot\prod_{\substack{\left(
i,j\right)  \in\left\{  1,2,\ldots,k\right\}  ^{2};\\i>j}}\left(  u_{i}%
-u_{j}\right)  }\\
&  =\left(  \prod_{j=1}^{k}t_{\rho_{j}}\right)  \left(  \prod_{i=1}%
^{k}t_{-w\left(  u_{i}\right)  }\right)  \cdot\prod_{\substack{\left(
i,j\right)  \in\left\{  1,2,\ldots,k\right\}  ^{2};\\i>j}}\left(  u_{i}%
-u_{j}\right)  .
\end{align*}

However, for every $i\in\left\{  1,2,\ldots,k\right\}  $ and $j\in\left\{
1,2,\ldots,k\right\}  $, we have%
\begin{align*}
a\left(  u_{i},j\right)   &  =\underbrace{t_{\rho_{j}-w\left(  u_{i}\right)
}}_{=t_{\rho_{j}}t_{-w\left(  u_{i}\right)  }}f_{j}\left(  u_{i}\right)
\ \ \ \ \ \ \ \ \ \ \left(  \text{by the definition of }a\left(
u_{i},j\right)  \right) \\
&  =t_{\rho_{j}}t_{-w\left(  u_{i}\right)  }f_{j}\left(  u_{i}\right)  .
\end{align*}
Hence,%
\begin{align}
&  \det\left(  \left(  \underbrace{a\left(  u_{i},j\right)  }_{=t_{\rho_{j}%
}t_{-w\left(  u_{i}\right)  }f_{j}\left(  u_{i}\right)  }\right)  _{1\leq
j\leq k,\ 1\leq i\leq k}\right) \nonumber\\
&  =\det\left(  \left(  t_{\rho_{j}}t_{-w\left(  u_{i}\right)  }f_{j}\left(
u_{i}\right)  \right)  _{1\leq j\leq k,\ 1\leq i\leq k}\right) \nonumber\\
&  =\left(  \prod_{j=1}^{k}t_{\rho_{j}}\right)  \left(  \prod_{i=1}%
^{k}t_{-w\left(  u_{i}\right)  }\right)  \cdot\prod_{\substack{\left(
i,j\right)  \in\left\{  1,2,\ldots,k\right\}  ^{2};\\i>j}}\left(  u_{i}%
-u_{j}\right)  . \label{pf.thm.poly-ut-greed.c2.pf.1}%
\end{align}
The right hand side of this equality is a product of nonzero elements of
$\mathbb{L}$ (since $u_{1},u_{2},\ldots,u_{k}$ are distinct), and thus is
nonzero (by Corollary \ref{cor.ord.intdom}). Hence, the left hand side is
nonzero. In other words, $\det\left(  \left(  a\left(  u_{i},j\right)
\right)  _{1\leq j\leq k,\ 1\leq i\leq k}\right)  $ is nonzero. This proves
(\ref{pf.thm.poly-ut-greed.c2.1}) (since we already know that $\det\left(
\left(  a\left(  u_{i},j\right)  \right)  _{1\leq j\leq k,\ 1\leq i\leq
k}\right)  $ belongs to $\mathbb{L}_{+}$).

Moreover, (\ref{pf.thm.poly-ut-greed.c2.pf.1}) yields%
\begin{align}
&  \operatorname*{ord}\left(  \det\left(  \left(  a\left(  u_{i},j\right)
\right)  _{1\leq j\leq k,\ 1\leq i\leq k}\right)  \right) \nonumber\\
&  =\operatorname*{ord}\left(  \left(  \prod_{j=1}^{k}t_{\rho_{j}}\right)
\left(  \prod_{i=1}^{k}t_{-w\left(  u_{i}\right)  }\right)  \cdot
\prod_{\substack{\left(  i,j\right)  \in\left\{  1,2,\ldots,k\right\}
^{2};\\i>j}}\left(  u_{i}-u_{j}\right)  \right) \nonumber\\
&  =\sum_{j=1}^{k}\underbrace{\operatorname*{ord}\left(  t_{\rho_{j}}\right)
}_{=\rho_{j}}+\sum_{i=1}^{k}\underbrace{\operatorname*{ord}\left(
t_{-w\left(  u_{i}\right)  }\right)  }_{=-w\left(  u_{i}\right)  }%
+\sum_{\substack{\left(  i,j\right)  \in\left\{  1,2,\ldots,k\right\}
^{2};\\i>j}}\operatorname*{ord}\left(  u_{i}-u_{j}\right) \nonumber\\
&  \ \ \ \ \ \ \ \ \ \ \left(  \text{by Lemma \ref{lem.ord.triangle}
\textbf{(c)} and Corollary \ref{cor.ord.prod-to-sum}}\right) \nonumber\\
&  =\sum_{j=1}^{k}\rho_{j}-\sum_{i=1}^{k}w\left(  u_{i}\right)  +\sum
_{\substack{\left(  i,j\right)  \in\left\{  1,2,\ldots,k\right\}  ^{2}%
;\\i>j}}\operatorname*{ord}\left(  u_{i}-u_{j}\right) \nonumber\\
&  =\sum_{j=1}^{k}\rho_{j}-\left(  \sum_{i=1}^{k}w\left(  u_{i}\right)
-\sum_{\substack{\left(  i,j\right)  \in\left\{  1,2,\ldots,k\right\}
^{2};\\i>j}}\operatorname*{ord}\left(  u_{i}-u_{j}\right)  \right)  .
\label{pf.thm.poly-ut-greed.vp1}%
\end{align}
But recall that $U=\left\{  u_{1},u_{2},\ldots,u_{k}\right\}  $ with
$u_{1},u_{2},\ldots,u_{k}$ distinct. The definition of perimeter thus yields%
\begin{align}
\operatorname*{PER}\left(  U\right)   &  =\sum_{i=1}^{k}w\left(  u_{i}\right)
+\sum_{1\leq i<j\leq k}d\left(  u_{i},u_{j}\right) \nonumber\\
&  =\sum_{i=1}^{k}w\left(  u_{i}\right)  +\underbrace{\sum_{1\leq j<i\leq k}%
}_{=\sum_{\substack{\left(  i,j\right)  \in\left\{  1,2,\ldots,k\right\}
^{2};\\i>j}}}\ \ \underbrace{d\left(  u_{j},u_{i}\right)  }%
_{\substack{=d\left(  u_{i},u_{j}\right)  \\\text{(by the \textquotedblleft
Symmetry\textquotedblright}\\\text{axiom in the definition}\\\text{of a
}\mathbb{V}\text{-ultra triple)}}}\nonumber\\
&  \ \ \ \ \ \ \ \ \ \ \left(
\begin{array}
[c]{c}%
\text{here, we have renamed}\\
\text{the index }\left(  i,j\right)  \text{ as }\left(  j,i\right)  \text{ in
the second sum}%
\end{array}
\right) \nonumber\\
&  =\sum_{i=1}^{k}w\left(  u_{i}\right)  +\sum_{\substack{\left(  i,j\right)
\in\left\{  1,2,\ldots,k\right\}  ^{2};\\i>j}}\ \ \underbrace{d\left(
u_{i},u_{j}\right)  }_{\substack{=-\operatorname*{ord}\left(  u_{i}%
-u_{j}\right)  \\\text{(by the definition of }d\text{)}}}\nonumber\\
&  =\sum_{i=1}^{k}w\left(  u_{i}\right)  -\sum_{\substack{\left(  i,j\right)
\in\left\{  1,2,\ldots,k\right\}  ^{2};\\i>j}}\operatorname*{ord}\left(
u_{i}-u_{j}\right)  . \label{pf.thm.poly-ut-greed.vp2}%
\end{align}
Hence, (\ref{pf.thm.poly-ut-greed.vp1}) becomes%
\begin{align*}
&  \operatorname*{ord}\left(  \det\left(  \left(  a\left(  u_{i},j\right)
\right)  _{1\leq j\leq k,\ 1\leq i\leq k}\right)  \right) \\
&  =\sum_{j=1}^{k}\rho_{j}-\underbrace{\left(  \sum_{i=1}^{k}w\left(
u_{i}\right)  -\sum_{\substack{\left(  i,j\right)  \in\left\{  1,2,\ldots
,k\right\}  ^{2};\\i>j}}\operatorname*{ord}\left(  u_{i}-u_{j}\right)
\right)  }_{\substack{=\operatorname*{PER}\left(  U\right)  \\\text{(by
(\ref{pf.thm.poly-ut-greed.vp2}))}}}\\
&  =\sum_{j=1}^{k}\rho_{j}-\operatorname*{PER}\left(  U\right)  .
\end{align*}
Hence, (\ref{pf.thm.poly-ut-greed.c2.2}) is proved. This proves Claim 2.]
\medskip

As a consequence of Claim 2, we obtain the following:

\begin{statement}
\textit{Claim 3:} Let $k\in\left\{  0,1,\ldots,m\right\}  $. Then, $\sum
_{j=1}^{k}\rho_{j}$ is the maximum perimeter of a $k$-subset of $E$.
\end{statement}

[\textit{Proof of Claim 3:} The elements $c_{1},c_{2},\ldots,c_{m}$ are
distinct; thus, the elements $c_{1},c_{2},\ldots,c_{k}$ are distinct. Hence,
$\left\{  c_{1},c_{2},\ldots,c_{k}\right\}  $ is a $k$-subset of $E$.

Adding up the equalities (\ref{pf.thm.poly-ut-greed.rhoj=}) for all
$j\in\left\{  1,2,\ldots,k\right\}  $, we obtain%
\begin{align*}
\sum_{j=1}^{k}\rho_{j}  &  =\sum_{j=1}^{k}\left(  w\left(  c_{j}\right)
+\sum_{i=1}^{j-1}d\left(  c_{i},c_{j}\right)  \right) \\
&  =\sum_{j=1}^{k}w\left(  c_{j}\right)  +\sum_{1\leq i<j\leq k}d\left(
c_{i},c_{j}\right)  =\operatorname*{PER}\left\{  c_{1},c_{2},\ldots
,c_{k}\right\}
\end{align*}
(since $c_{1},c_{2},\ldots,c_{k}$ are distinct). Since $\left\{  c_{1}%
,c_{2},\ldots,c_{k}\right\}  $ is a $k$-subset of $E$, we thus conclude that
$\sum_{j=1}^{k}\rho_{j}$ is the perimeter of some $k$-subset of $E$. Thus, in
order to prove Claim 3, we need only to show that $\sum_{j=1}^{k}\rho_{j}%
\geq\operatorname*{PER}\left(  U\right)  $ for every $k$-subset $U$ of $E$.

So let $U$ be a $k$-subset of $E$. We must prove $\sum_{j=1}^{k}\rho_{j}%
\geq\operatorname*{PER}\left(  U\right)  $.

Write the $k$-subset $U$ in the form $U=\left\{  u_{1},u_{2},\ldots
,u_{k}\right\}  $ for $k$ distinct elements $u_{1},u_{2},\ldots,u_{k}$ of $E$.
Claim 2 thus yields that%
\[
\det\left(  \left(  a\left(  u_{i},j\right)  \right)  _{1\leq j\leq k,\ 1\leq
i\leq k}\right)  \text{ is a nonzero element of }\mathbb{L}_{+}%
\]
and%
\begin{equation}
\operatorname*{ord}\left(  \det\left(  \left(  a\left(  u_{i},j\right)
\right)  _{1\leq j\leq k,\ 1\leq i\leq k}\right)  \right)  =\sum_{j=1}^{k}%
\rho_{j}-\operatorname*{PER}\left(  U\right)  .
\label{pf.thm.poly-ut-greed.c3.pf.2}%
\end{equation}
Thus, in particular, $\det\left(  \left(  a\left(  u_{i},j\right)  \right)
_{1\leq j\leq k,\ 1\leq i\leq k}\right)  $ belongs to $\mathbb{L}_{+}$.
Hence,
\[
\operatorname*{ord}\left(  \det\left(  \left(  a\left(  u_{i},j\right)
\right)  _{1\leq j\leq k,\ 1\leq i\leq k}\right)  \right)  \geq0
\]
(by Lemma \ref{lem.ord.triangle} \textbf{(a)}, applied to $\det\left(  \left(
a\left(  u_{i},j\right)  \right)  _{1\leq j\leq k,\ 1\leq i\leq k}\right)  $
instead of $a$). In view of (\ref{pf.thm.poly-ut-greed.c3.pf.2}), this
rewrites as $\sum_{j=1}^{k}\rho_{j}-\operatorname*{PER}\left(  U\right)
\geq0$. In other words, $\sum_{j=1}^{k}\rho_{j}\geq\operatorname*{PER}\left(
U\right)  $. This completes the proof of Claim 3.] \medskip

Now, recall the $\mathbb{K}$-algebra homomorphism $\pi:\mathbb{L}%
_{+}\rightarrow\mathbb{K}$ from Lemma \ref{lem.piLK}. For each $e\in E$,
define a column vector $v_{e}\in\mathbb{K}^{m}$ by%
\[
v_{e}=\left(
\begin{array}
[c]{c}%
\pi\left(  a\left(  e,1\right)  \right) \\
\pi\left(  a\left(  e,2\right)  \right) \\
\vdots\\
\pi\left(  a\left(  e,m\right)  \right)
\end{array}
\right)  =\left(  \pi\left(  a\left(  e,j\right)  \right)  \right)  _{1\leq
j\leq m}.
\]
We thus have a vector family $\left(  v_{e}\right)  _{e\in E}$ over
$\mathbb{K}$. Let $\mathcal{G}$ be the Gaussian elimination greedoid of this
family. Let $\mathcal{F}$ be the Bhargava greedoid of $\left(  E,w,d\right)
$. Our goal is to prove that $\mathcal{F}=\mathcal{G}$ (since this will yield
Theorem \ref{thm.poly-ut-greed}).

In order to do so, it suffices to show that if $U$ is any subset of $E$, then
we have the logical equivalence
\begin{equation}
\left(  U\in\mathcal{F}\right)  \ \Longleftrightarrow\ \left(  U\in
\mathcal{G}\right)  . \label{pf.thm.poly-ut-greed.equiv}%
\end{equation}

So let us do this. Let $U$ be a subset of $E$. We must prove the equivalence
(\ref{pf.thm.poly-ut-greed.equiv}).

Write the subset $U$ in the form $U=\left\{  u_{1},u_{2},\ldots,u_{k}\right\}
$ with $u_{1},u_{2},\ldots,u_{k}$ distinct. Thus, $\left\vert U\right\vert
=k$. In other words, $U$ is a $k$-subset of $E$. Claim 2 thus yields that%
\begin{equation}
\det\left(  \left(  a\left(  u_{i},j\right)  \right)  _{1\leq j\leq k,\ 1\leq
i\leq k}\right)  \text{ is a nonzero element of }\mathbb{L}_{+}
\label{pf.thm.poly-ut-greed.equiv.pf.1}%
\end{equation}
and%
\begin{equation}
\operatorname*{ord}\left(  \det\left(  \left(  a\left(  u_{i},j\right)
\right)  _{1\leq j\leq k,\ 1\leq i\leq k}\right)  \right)  =\sum_{j=1}^{k}%
\rho_{j}-\operatorname*{PER}\left(  U\right)  .
\label{pf.thm.poly-ut-greed.equiv.pf.2}%
\end{equation}

Also, recall that $\pi$ is a $\mathbb{K}$-algebra homomorphism (by Lemma
\ref{lem.piLK}), thus a ring homomorphism. Hence,%
\begin{equation}
\det\left(  \left(  \pi\left(  a\left(  u_{i},j\right)  \right)  \right)
_{1\leq j\leq k,\ 1\leq i\leq k}\right)  =\pi\left(  \det\left(  \left(
a\left(  u_{i},j\right)  \right)  _{1\leq j\leq k,\ 1\leq i\leq k}\right)
\right)  \label{pf.thm.poly-ut-greed.equiv.pf.3}%
\end{equation}
(since ring homomorphisms respect determinants).

Also, $U$ is a subset of $E$; hence, $\left\vert U\right\vert \leq\left\vert
E\right\vert =m$. Thus, $k=\left\vert U\right\vert \leq m$. Hence,
$k\in\left\{  0,1,\ldots,m\right\}  $. Therefore, Claim 3 shows that
$\sum_{j=1}^{k}\rho_{j}$ is the maximum perimeter of a $k$-subset of $E$. In
other words, $\sum_{j=1}^{k}\rho_{j}$ is the maximum perimeter of a
$\left\vert U\right\vert $-subset of $E$ (since $\left\vert U\right\vert =k$).

The definition of the Gaussian elimination greedoid $\mathcal{G}$ shows that
we have the following equivalence:\footnote{The words \textquotedblleft
linearly independent\textquotedblright\ should always be understood to mean
\textquotedblleft$\mathbb{K}$-linearly independent\textquotedblright\ here.}%
\begin{align*}
&  \ \left(  U\in\mathcal{G}\right) \\
&  \Longleftrightarrow\ \left(  \text{the family }\left(  \pi_{\left\vert
U\right\vert }\left(  v_{e}\right)  \right)  _{e\in U}\in\left(
\mathbb{K}^{\left\vert U\right\vert }\right)  ^{U}\text{ is linearly
independent}\right) \\
&  \Longleftrightarrow\ \left(  \text{the family }\left(  \pi_{k}\left(
v_{e}\right)  \right)  _{e\in U}\in\left(  \mathbb{K}^{k}\right)  ^{U}\text{
is linearly independent}\right) \\
&  \ \ \ \ \ \ \ \ \ \ \left(  \text{since }\left\vert U\right\vert =k\right)
\\
&  \Longleftrightarrow\ \left(  \text{the vectors }\pi_{k}\left(  v_{u_{1}%
}\right)  ,\pi_{k}\left(  v_{u_{2}}\right)  ,\ldots,\pi_{k}\left(  v_{u_{k}%
}\right)  \text{ are linearly independent}\right) \\
&  \ \ \ \ \ \ \ \ \ \ \left(  \text{since }U=\left\{  u_{1},u_{2}%
,\ldots,u_{k}\right\}  \text{ with }u_{1},u_{2},\ldots,u_{k}\text{
distinct}\right) \\
&  \Longleftrightarrow\ \left(  \text{the columns of the matrix }\left(
\pi\left(  a\left(  u_{i},j\right)  \right)  \right)  _{1\leq j\leq k,\ 1\leq
i\leq k}\text{ are linearly independent}\right) \\
&  \ \ \ \ \ \ \ \ \ \ \ \ \ \ \ \ \ \ \ \ \left(
\begin{array}
[c]{c}%
\text{since the vectors }\pi_{k}\left(  v_{u_{1}}\right)  ,\pi_{k}\left(
v_{u_{2}}\right)  ,\ldots,\pi_{k}\left(  v_{u_{k}}\right) \\
\text{are the columns of the matrix }\left(  \pi\left(  a\left(
u_{i},j\right)  \right)  \right)  _{1\leq j\leq k,\ 1\leq i\leq k}%
\end{array}
\right) \\
&  \Longleftrightarrow\ \left(  \text{the matrix }\left(  \pi\left(  a\left(
u_{i},j\right)  \right)  \right)  _{1\leq j\leq k,\ 1\leq i\leq k}\text{ is
invertible}\right) \\
&  \ \ \ \ \ \ \ \ \ \ \ \ \ \ \ \ \ \ \ \ \left(
\begin{array}
[c]{c}%
\text{since a square matrix over the field }\mathbb{K}\text{ is invertible}\\
\text{if and only if its columns are linearly independent}%
\end{array}
\right) \\
&  \Longleftrightarrow\ \left(  \det\left(  \left(  \pi\left(  a\left(
u_{i},j\right)  \right)  \right)  _{1\leq j\leq k,\ 1\leq i\leq k}\right)
\neq0\text{ in }\mathbb{K}\right) \\
&  \Longleftrightarrow\ \left(  \pi\left(  \det\left(  \left(  a\left(
u_{i},j\right)  \right)  _{1\leq j\leq k,\ 1\leq i\leq k}\right)  \right)
\neq0\text{ in }\mathbb{K}\right)  \ \ \ \ \ \ \ \ \ \ \left(  \text{by
(\ref{pf.thm.poly-ut-greed.equiv.pf.3})}\right) \\
&  \Longleftrightarrow\ \left(  \operatorname*{ord}\left(  \det\left(  \left(
a\left(  u_{i},j\right)  \right)  _{1\leq j\leq k,\ 1\leq i\leq k}\right)
\right)  =0\right) \\
&  \ \ \ \ \ \ \ \ \ \ \ \ \ \ \ \ \ \ \ \ \left(  \text{by Lemma
\ref{lem.pinot0}, applied to }\det\left(  \left(  a\left(  u_{i},j\right)
\right)  _{1\leq j\leq k,\ 1\leq i\leq k}\right)  \text{ instead of }a\right)
\\
&  \Longleftrightarrow\ \left(  \sum_{j=1}^{k}\rho_{j}-\operatorname*{PER}%
\left(  U\right)  =0\right)  \ \ \ \ \ \ \ \ \ \ \left(  \text{by
(\ref{pf.thm.poly-ut-greed.equiv.pf.2})}\right) \\
&  \Longleftrightarrow\ \left(  \operatorname*{PER}\left(  U\right)
=\sum_{j=1}^{k}\rho_{j}\right) \\
&  \Longleftrightarrow\ \left(  \operatorname*{PER}\left(  U\right)  \text{ is
the maximum perimeter of a }\left\vert U\right\vert \text{-subset of }E\right)
\\
&  \ \ \ \ \ \ \ \ \ \ \ \ \ \ \ \ \ \ \ \ \left(  \text{since }\sum_{j=1}%
^{k}\rho_{j}\text{ is the maximum perimeter of a }\left\vert U\right\vert
\text{-subset of }E\right) \\
&  \Longleftrightarrow\ \left(  U\text{ has maximum perimeter among all
}\left\vert U\right\vert \text{-subsets of }E\right) \\
&  \Longleftrightarrow\ \left(  U\in\mathcal{F}\right)
\end{align*}
(by the definition of the Bhargava greedoid $\mathcal{F}$). Thus, the
equivalence (\ref{pf.thm.poly-ut-greed.equiv}) is proven. This concludes the
proof of Theorem \ref{thm.poly-ut-greed}.
\end{proof}

\section{Isomorphism}

Next, we introduce the notion of a set system. This elementary notion will
play a purely technical role in what follows.

\begin{definition}
\label{def.setsys}Let $E$ be a set.

\begin{enumerate}
\item[\textbf{(a)}] We let $2^{E}$ denote the powerset of $E$.

\item[\textbf{(b)}] A \textit{set system} on ground set $E$ shall mean a
subset of $2^{E}$.
\end{enumerate}
\end{definition}

Thus:

\begin{itemize}
\item The Gaussian elimination greedoid of a vector family $\left(
v_{e}\right)  _{e\in E}$ (over any field $\mathbb{K}$) is a set system on
ground set $E$.

\item The Bhargava greedoid of any $\mathbb{V}$-ultra triple $\left(
E,w,d\right)  $ is a set system on ground set $E$.
\end{itemize}

(More generally, any greedoid is a set system, but we shall not need this.)

We shall use the following two simple concepts of isomorphism:

\begin{definition}
Let $\left(  E,w,d\right)  $ and $\left(  F,v,c\right)  $ be two $\mathbb{V}%
$-ultra triples.

\begin{enumerate}
\item[\textbf{(a)}] A bijective map $f:E\rightarrow F$ is said to be an
\textit{isomorphism of }$\mathbb{V}$\textit{-ultra triples from }$\left(
E,w,d\right)  $ \textit{to }$\left(  F,v,c\right)  $ if it satisfies $v\circ
f=w$ and
\[
c\left(  f\left(  a\right)  ,f\left(  b\right)  \right)  =d\left(  a,b\right)
\ \ \ \ \ \ \ \ \ \ \text{for all }\left(  a,b\right)  \in E\timesu E.
\]

\item[\textbf{(b)}] The $\mathbb{V}$-ultra triples $\left(  E,w,d\right)  $
and $\left(  F,v,c\right)  $ are said to be \textit{isomorphic} if there
exists an isomorphism $f:E\rightarrow F$ of $\mathbb{V}$-ultra triples from
$\left(  E,w,d\right)  $ to $\left(  F,v,c\right)  $. (Note that being
isomorphic is clearly a symmetric relation, because if $f:E\rightarrow F$ is
an isomorphism of $\mathbb{V}$-ultra triples from $\left(  E,w,d\right)  $ to
$\left(  F,v,c\right)  $, then $f^{-1}:F\rightarrow E$ is an isomorphism of
$\mathbb{V}$-ultra triples from $\left(  F,v,c\right)  $ to $\left(
E,w,d\right)  $.)
\end{enumerate}
\end{definition}

\begin{definition}
Let $\mathcal{E}$ and $\mathcal{F}$ be two set systems on ground sets $E$ and
$F$, respectively.

\begin{enumerate}
\item[\textbf{(a)}] A bijective map $f:E\rightarrow F$ is said to be an
\textit{isomorphism of set systems from }$\mathcal{E}$ \textit{to
}$\mathcal{F}$ if the bijection $2^{f}:2^{E}\rightarrow2^{F}$ induced by it
(i.e., the bijection that sends each $S\in2^{E}$ to $f\left(  S\right)
\in2^{F}$) satisfies $2^{f}\left(  \mathcal{E}\right)  =\mathcal{F}$.

\item[\textbf{(b)}] The set systems $\mathcal{E}$ and $\mathcal{F}$ are said
to be \textit{isomorphic} if there exists an isomorphism $f:E\rightarrow F$ of
set systems from $\mathcal{E}$ to $\mathcal{F}$. (Note that being isomorphic
is clearly a symmetric relation, because if $f:E\rightarrow F$ is an
isomorphism of set systems from $\mathcal{E}$ to $\mathcal{F}$, then
$f^{-1}:F\rightarrow E$ is an isomorphism of set systems from $\mathcal{F}$ to
$\mathcal{E}$.)
\end{enumerate}
\end{definition}

The intuitive meaning of both of these two definitions is simple: Two
$\mathbb{V}$-ultra triples are isomorphic if and only if one can be obtained
from the other by relabeling the elements of the ground set. The same holds
for two set systems.

The following two propositions are obvious:

\begin{proposition}
\label{prop.iso.bh-gr}Let $\left(  E,w,d\right)  $ and $\left(  F,v,c\right)
$ be two isomorphic $\mathbb{V}$-ultra triples such that $E$ and $F$ are
finite. Then, the Bhargava greedoids of $\left(  E,w,d\right)  $ and $\left(
F,v,c\right)  $ are isomorphic as set systems.
\end{proposition}

\begin{proposition}
\label{prop.iso.gauss}Let $\mathbb{K}$ be a field. Let $\mathcal{E}$ and
$\mathcal{F}$ be two isomorphic set systems. If $\mathcal{E}$ is the Gaussian
elimination greedoid of a vector family over $\mathbb{K}$, then so is
$\mathcal{F}$.
\end{proposition}

\section{\label{sect.cliques}Decomposing a $\mathbb{V}$-ultra triple}

For the rest of Section \ref{sect.cliques}, we fix a $\mathbb{V}$-ultra triple
$\left(  E,w,d\right)  $.

We are going to study the structure of this $\mathbb{V}$-ultra triple. We
recall the notions of $\alpha$-cliques and cliques (defined in Definition
\ref{def.alpha-clique} and Definition \ref{def.clique}, respectively). We
shall next define another kind of subsets of $E$: the \textit{open balls}.

\begin{definition}
\label{def.open-ball}Let $\alpha\in\mathbb{V}$ and $e\in E$. The \textit{open
ball }$B_{\alpha}^{\circ}\left(  e\right)  $ is defined to be the subset%
\[
\left\{  f\in E\ \mid\ f=e\text{ or else }d\left(  f,e\right)  <\alpha
\right\}
\]
of $E$.
\end{definition}

Clearly, for each $\alpha\in\mathbb{V}$ and each $e\in E$, we have $e\in
B_{\alpha}^{\circ}\left(  e\right)  $, so that the open ball $B_{\alpha
}^{\circ}\left(  e\right)  $ contains at least the element $e$.

\begin{proposition}
\label{prop.Bo.self-rad}Let $\alpha\in\mathbb{V}$ and $e,f\in E$ be such that
$e\neq f$ and $d\left(  e,f\right)  <\alpha$. Then, $B_{\alpha}^{\circ}\left(
e\right)  =B_{\alpha}^{\circ}\left(  f\right)  $.
\end{proposition}

\begin{proof}
We have $e\neq f$, so that $f\neq e$. Hence, the \textquotedblleft
symmetry\textquotedblright\ axiom in Definition \ref{def.Vultra} yields that
$d\left(  f,e\right)  =d\left(  e,f\right)  <\alpha$.

From $e\neq f$ and $d\left(  e,f\right)  <\alpha$, we obtain $e\in B_{\alpha
}^{\circ}\left(  f\right)  $ (by the definition of $B_{\alpha}^{\circ}\left(
f\right)  $). Also, $f\in B_{\alpha}^{\circ}\left(  f\right)  $ (by the
definition of $B_{\alpha}^{\circ}\left(  f\right)  $).

Let $x\in B_{\alpha}^{\circ}\left(  e\right)  $. We shall show that $x\in
B_{\alpha}^{\circ}\left(  f\right)  $.

If $x=e$, then this follows immediately from $e\in B_{\alpha}^{\circ}\left(
f\right)  $. Hence, for the rest of the proof of $x\in B_{\alpha}^{\circ
}\left(  f\right)  $, we WLOG assume that $x\neq e$. Thus, from $x\in
B_{\alpha}^{\circ}\left(  e\right)  $, we obtain $d\left(  x,e\right)
<\alpha$ (by the definition of $B_{\alpha}^{\circ}\left(  e\right)  $). If
$x=f$, then $x\in B_{\alpha}^{\circ}\left(  f\right)  $ follows immediately
from the fact that $f\in B_{\alpha}^{\circ}\left(  f\right)  $. Thus, for the
rest of the proof of $x\in B_{\alpha}^{\circ}\left(  f\right)  $, we WLOG
assume that $x\neq f$. Now, the points $e,f,x$ are distinct (since $x\neq e$,
$x\neq f$ and $e\neq f$). Hence, the ultrametric triangle inequality yields
$d\left(  x,f\right)  \leq\max\left\{  d\left(  x,e\right)  ,d\left(
f,e\right)  \right\}  <\alpha$ (since $d\left(  x,e\right)  <\alpha$ and
$d\left(  f,e\right)  <\alpha$). Thus, $x\in B_{\alpha}^{\circ}\left(
f\right)  $ (by the definition of $B_{\alpha}^{\circ}\left(  f\right)  $).

Now, forget that we fixed $x$. We thus have shown that $x\in B_{\alpha}%
^{\circ}\left(  f\right)  $ for each $x\in B_{\alpha}^{\circ}\left(  e\right)
$. In other words, $B_{\alpha}^{\circ}\left(  e\right)  \subseteq B_{\alpha
}^{\circ}\left(  f\right)  $.

But our situation is symmetric in $e$ and $f$ (since $f\neq e$ and $d\left(
f,e\right)  <\alpha$). Hence, the same argument that let us prove $B_{\alpha
}^{\circ}\left(  e\right)  \subseteq B_{\alpha}^{\circ}\left(  f\right)  $ can
be applied with the roles of $e$ and $f$ interchanged; thus we obtain
$B_{\alpha}^{\circ}\left(  f\right)  \subseteq B_{\alpha}^{\circ}\left(
e\right)  $. Combining this with $B_{\alpha}^{\circ}\left(  e\right)
\subseteq B_{\alpha}^{\circ}\left(  f\right)  $, we obtain $B_{\alpha}^{\circ
}\left(  e\right)  =B_{\alpha}^{\circ}\left(  f\right)  $. This proves
Proposition \ref{prop.Bo.self-rad}.
\end{proof}

\begin{corollary}
\label{cor.Bo.self-rad2}Let $\alpha\in\mathbb{V}$ and $e\in E$. Let $f\in
B_{\alpha}^{\circ}\left(  e\right)  $. Then, $B_{\alpha}^{\circ}\left(
e\right)  =B_{\alpha}^{\circ}\left(  f\right)  $.
\end{corollary}

\begin{proof}
If $e=f$, then this is obvious. Hence, WLOG assume that $e\neq f$. Thus,
$f\neq e$. Hence, $d\left(  f,e\right)  <\alpha$ (since $f\in B_{\alpha
}^{\circ}\left(  e\right)  $). Thus, the \textquotedblleft
symmetry\textquotedblright\ axiom in Definition \ref{def.Vultra} yields that
$d\left(  e,f\right)  =d\left(  f,e\right)  <\alpha$. Hence, Proposition
\ref{prop.Bo.self-rad} yields $B_{\alpha}^{\circ}\left(  e\right)  =B_{\alpha
}^{\circ}\left(  f\right)  $. Qed.
\end{proof}

\begin{proposition}
\label{prop.Bo.disjoint}Let $\alpha\in\mathbb{V}$ and $e,f\in E$ be such that
$e\neq f$ and $d\left(  e,f\right)  \geq\alpha$. Then:

\begin{enumerate}
\item[\textbf{(a)}] If $a\in B_{\alpha}^{\circ}\left(  e\right)  $ and $b\in
B_{\alpha}^{\circ}\left(  f\right)  $, then $a\neq b$ and $d\left(
a,b\right)  \geq\alpha$.

\item[\textbf{(b)}] The open balls $B_{\alpha}^{\circ}\left(  e\right)  $ and
$B_{\alpha}^{\circ}\left(  f\right)  $ are disjoint.
\end{enumerate}
\end{proposition}

\begin{proof}
\textbf{(a)} Let $a\in B_{\alpha}^{\circ}\left(  e\right)  $ and $b\in
B_{\alpha}^{\circ}\left(  f\right)  $. We must prove that $a\neq b$ and
$d\left(  a,b\right)  \geq\alpha$.

Assume the contrary. Thus, either $a=b$ or else $d\left(  a,b\right)  <\alpha
$. Hence, $a\in B_{\alpha}^{\circ}\left(  b\right)  $ (by the definition of
$B_{\alpha}^{\circ}\left(  b\right)  $). Thus, $B_{\alpha}^{\circ}\left(
b\right)  =B_{\alpha}^{\circ}\left(  a\right)  $ (by Corollary
\ref{cor.Bo.self-rad2}, applied to $b$ and $a$ instead of $e$ and $f$). Also,
from $a\in B_{\alpha}^{\circ}\left(  e\right)  $, we obtain $B_{\alpha}%
^{\circ}\left(  e\right)  =B_{\alpha}^{\circ}\left(  a\right)  $ (by Corollary
\ref{cor.Bo.self-rad2}, applied to $a$ instead of $f$). Furthermore, from
$b\in B_{\alpha}^{\circ}\left(  f\right)  $, we obtain $B_{\alpha}^{\circ
}\left(  f\right)  =B_{\alpha}^{\circ}\left(  b\right)  $ (by Corollary
\ref{cor.Bo.self-rad2}, applied to $f$ and $b$ instead of $e$ and $f$).
Finally, the definition of $B_{\alpha}^{\circ}\left(  e\right)  $ yields $e\in
B_{\alpha}^{\circ}\left(  e\right)  =B_{\alpha}^{\circ}\left(  a\right)
=B_{\alpha}^{\circ}\left(  b\right)  =B_{\alpha}^{\circ}\left(  f\right)  $.
Since $e\neq f$, this entails $d\left(  e,f\right)  <\alpha$ (by the
definition of $B_{\alpha}^{\circ}\left(  f\right)  $). This contradicts
$d\left(  e,f\right)  \geq\alpha$. This contradiction shows that our
assumption was false. Hence, Proposition \ref{prop.Bo.disjoint} \textbf{(a)}
follows. \medskip

\textbf{(b)} This is simply the \textquotedblleft$a\neq b$\textquotedblright%
\ part of Proposition \ref{prop.Bo.disjoint} \textbf{(a)}.
\end{proof}

The next proposition is trivial:

\begin{proposition}
\label{prop.cliques.sub}Let $F$ be a subset of $E$. Let $\left(  F,w^{\prime
},d^{\prime}\right)  $ be the $\mathbb{V}$-ultra triple $\left(  F,w\mid
_{F},d\mid_{F\timesu F}\right)  $. Then:

\begin{enumerate}
\item[\textbf{(a)}] If a subset of $F$ is a clique of the $\mathbb{V}$-ultra
triple $\left(  E,w,d\right)  $, then this subset is also a clique of the
$\mathbb{V}$-ultra triple $\left(  F,w^{\prime},d^{\prime}\right)  $.

\item[\textbf{(b)}] Any clique of the $\mathbb{V}$-ultra triple $\left(
F,w^{\prime},d^{\prime}\right)  $ is a clique of the $\mathbb{V}$-ultra triple
$\left(  E,w,d\right)  $.
\end{enumerate}
\end{proposition}

\begin{proof}
\textbf{(a)} Let $C$ be a subset of $F$ that is a clique of the $\mathbb{V}%
$-ultra triple $\left(  E,w,d\right)  $. We must show that $C$ is also a
clique of the $\mathbb{V}$-ultra triple $\left(  F,w^{\prime},d^{\prime
}\right)  $.

The definition of $\left(  F,w^{\prime},d^{\prime}\right)  $ shows that
$d^{\prime}=d\mid_{F\timesu F}$. Thus, any two distinct elements $a,b\in F$
satisfy%
\begin{equation}
d^{\prime}\left(  a,b\right)  =d\left(  a,b\right)  .
\label{pf.prop.cliques.sub.1}%
\end{equation}

We know that $C$ is a clique of $\left(  E,w,d\right)  $. In other words, $C$
is an $\alpha$-clique of $\left(  E,w,d\right)  $ for some $\alpha
\in\mathbb{V}$. Consider this $\alpha$. Thus, $C$ is an $\alpha$-clique of
$\left(  E,w,d\right)  $. In other words, any two distinct elements $a,b\in C$
satisfy
\begin{equation}
d\left(  a,b\right)  =\alpha\label{pf.prop.cliques.sub.2}%
\end{equation}
(by the definition of an \textquotedblleft$\alpha$-clique\textquotedblright).
Hence, for any two distinct elements $a,b\in C$, we have%
\begin{align*}
d^{\prime}\left(  a,b\right)   &  =d\left(  a,b\right)
\ \ \ \ \ \ \ \ \ \ \left(  \text{by (\ref{pf.prop.cliques.sub.1}), since
}a,b\in C\subseteq F\right) \\
&  =\alpha\ \ \ \ \ \ \ \ \ \ \left(  \text{by (\ref{pf.prop.cliques.sub.2}%
)}\right)  .
\end{align*}

In other words, any two distinct elements $a,b\in C$ satisfy $d^{\prime
}\left(  a,b\right)  =\alpha$. Since $C$ is a subset of $F$, we thus conclude
that $C$ is an $\alpha$-clique of the $\mathbb{V}$-ultra triple $\left(
F,w^{\prime},d^{\prime}\right)  $ (by the definition of an \textquotedblleft%
$\alpha$-clique\textquotedblright). Hence, $C$ is a clique of the $\mathbb{V}%
$-ultra triple $\left(  F,w^{\prime},d^{\prime}\right)  $. This proves
Proposition \ref{prop.cliques.sub} \textbf{(a)}. \medskip

\textbf{(b)} This is proved essentially by reading the above proof of
Proposition \ref{prop.cliques.sub} \textbf{(a)} backwards (as we now have to
derive $d\left(  a,b\right)  =\alpha$ from $d^{\prime}\left(  a,b\right)
=\alpha$).
\end{proof}

The next proposition shows how a finite $\mathbb{V}$-ultra triple (of size
$>1$) can be decomposed into several smaller $\mathbb{V}$-ultra triples; this
will later be used for recursive reasoning:\footnote{See Definition
\ref{def.mcs} for the meaning of $\operatorname*{mcs}\left(  E,w,d\right)  $.}

\begin{proposition}
\label{prop.cliques.divide}Assume that $E$ is finite and $\left\vert
E\right\vert >1$. Let $\alpha=\max\left(  d\left(  E\timesu E\right)  \right)
$.

Pick any maximum-size $\alpha$-clique, and write it in the form $\left\{
e_{1},e_{2},\ldots,e_{m}\right\}  $ for some distinct elements $e_{1}%
,e_{2},\ldots,e_{m}$ of $E$.

For each $i\in\left\{  1,2,\ldots,m\right\}  $, let $E_{i}$ be the open ball
$B_{\alpha}^{\circ}\left(  e_{i}\right)  $, and let $\left(  E_{i},w_{i}%
,d_{i}\right)  $ be the $\mathbb{V}$-ultra triple $\left(  E_{i},w\mid_{E_{i}%
},d\mid_{E_{i}\timesu E_{i}}\right)  $.

Then:

\begin{enumerate}
\item[\textbf{(a)}] We have $m>1$.

\item[\textbf{(b)}] The sets $E_{1},E_{2},\ldots,E_{m}$ form a set partition
of $E$. (This means that these sets $E_{1},E_{2},\ldots,E_{m}$ are disjoint
and nonempty and their union is $E$.)

\item[\textbf{(c)}] We have $\left\vert E_{i}\right\vert <\left\vert
E\right\vert $ for each $i\in\left\{  1,2,\ldots,m\right\}  $.

\item[\textbf{(d)}] If $i\in\left\{  1,2,\ldots,m\right\}  $, and if $a,b\in
E_{i}$ are distinct, then $d\left(  a,b\right)  <\alpha$.

\item[\textbf{(e)}] If $i$ and $j$ are two distinct elements of $\left\{
1,2,\ldots,m\right\}  $, and if $a\in E_{i}$ and $b\in E_{j}$, then $a\neq b$
and $d\left(  a,b\right)  =\alpha$.

\item[\textbf{(f)}] Let $n_{i}=\operatorname*{mcs}\left(  E_{i},w_{i}%
,d_{i}\right)  $ for each $i\in\left\{  1,2,\ldots,m\right\}  $. Then,%
\[
\operatorname*{mcs}\left(  E,w,d\right)  =\max\left\{  m,n_{1},n_{2}%
,\ldots,n_{m}\right\}  .
\]

\end{enumerate}
\end{proposition}

\begin{proof}
Let us first check that $\alpha$ is well-defined. Indeed, the set $E\timesu E$
is nonempty (since $\left\vert E\right\vert >1$) and finite (since $E$ is
finite). Hence, the set $d\left(  E\timesu E\right)  $ is nonempty and finite
as well. Thus, it has a well-defined largest element $\max\left(  d\left(
E\timesu E\right)  \right)  $. In other words, $\alpha$ is well-defined (since
$\alpha$ was defined by $\alpha=\max\left(  d\left(  E\timesu E\right)
\right)  $).

We have
\begin{equation}
d\left(  a,b\right)  \leq\alpha\ \ \ \ \ \ \ \ \ \ \text{for each }\left(
a,b\right)  \in E\timesu E \label{pf.prop.cliques.divide.dleqa}%
\end{equation}
(since $\alpha=\max\left(  d\left(  E\timesu E\right)  \right)  $).

The set $\left\{  e_{1},e_{2},\ldots,e_{m}\right\}  $ is a maximum-size
$\alpha$-clique (by its definition), but its size is $m$ (since $e_{1}%
,e_{2},\ldots,e_{m}$ are distinct). Hence, the maximum size of an $\alpha
$-clique is $m$. Thus, every $\alpha$-clique has size $\leq m$. \medskip

\textbf{(a)} From $\alpha=\max\left(  d\left(  E\timesu E\right)  \right)  \in
d\left(  E\timesu E\right)  $, we conclude that there exist two distinct
elements $u$ and $v$ of $E$ satisfying $d\left(  u,v\right)  =\alpha$.
Consider these $u$ and $v$. Then, $\left\{  u,v\right\}  $ is an $\alpha
$-clique. This $\alpha$-clique must have size $\leq m$ (since every $\alpha
$-clique has size $\leq m$). Hence, $\left\vert \left\{  u,v\right\}
\right\vert \leq m$. Thus, $m\geq\left\vert \left\{  u,v\right\}  \right\vert
=2$ (since $u$ and $v$ are distinct), so that $m\geq2>1$. This proves
Proposition \ref{prop.cliques.divide} \textbf{(a)}. \medskip

\textbf{(b)} If $i$ and $j$ are two distinct elements of $\left\{
1,2,\ldots,m\right\}  $, then $e_{i}\neq e_{j}$ (since $e_{1},e_{2}%
,\ldots,e_{m}$ are distinct) and thus $d\left(  e_{i},e_{j}\right)  =\alpha$
(since $\left\{  e_{1},e_{2},\ldots,e_{m}\right\}  $ is an $\alpha$-clique),
and therefore the open balls $B_{\alpha}^{\circ}\left(  e_{i}\right)  $ and
$B_{\alpha}^{\circ}\left(  e_{j}\right)  $ are disjoint (by Proposition
\ref{prop.Bo.disjoint} \textbf{(b)}, applied to $e=e_{i}$ and $f=e_{j}$). In
other words, if $i$ and $j$ are two distinct elements of $\left\{
1,2,\ldots,m\right\}  $, then the open balls $E_{i}$ and $E_{j}$ are disjoint
(since $E_{i}=B_{\alpha}^{\circ}\left(  e_{i}\right)  $ and $E_{j}=B_{\alpha
}^{\circ}\left(  e_{j}\right)  $).

Hence, the open balls $E_{1},E_{2},\ldots,E_{m}$ are disjoint. Furthermore,
these balls are nonempty (since each open ball $E_{i}=B_{\alpha}^{\circ
}\left(  e_{i}\right)  $ contains at least the element $e_{i}$).

Furthermore, the union of these balls $E_{1},E_{2},\ldots,E_{m}$ is the whole
set $E$.

[\textit{Proof:} Assume the contrary. Thus, the union of the balls
$E_{1},E_{2},\ldots,E_{m}$ must be a \textbf{proper} subset of $E$ (since it
is clearly a subset of $E$). In other words, there exists some $a\in E$ that
belongs to none of these balls $E_{1},E_{2},\ldots,E_{m}$. Consider this $a$.
Then, for each $i\in\left\{  1,2,\ldots,m\right\}  $, we have $a\notin
E_{i}=B_{\alpha}^{\circ}\left(  e_{i}\right)  $. By the definition of
$B_{\alpha}^{\circ}\left(  e_{i}\right)  $, this entails that $a\neq e_{i}$
and $d\left(  a,e_{i}\right)  \geq\alpha$, hence $d\left(  a,e_{i}\right)
=\alpha$ (since (\ref{pf.prop.cliques.divide.dleqa}) yields $d\left(
a,e_{i}\right)  \leq\alpha$). Thus, we have shown that $d\left(
a,e_{i}\right)  =\alpha$ for all $i\in\left\{  1,2,\ldots,m\right\}  $.
Therefore, $\left\{  a,e_{1},e_{2},\ldots,e_{m}\right\}  $ is an $\alpha
$-clique. This $\alpha$-clique has size $m+1$ (since $e_{1},e_{2},\ldots
,e_{m}$ are distinct, and since $a\neq e_{i}$ for each $i\in\left\{
1,2,\ldots,m\right\}  $). But this contradicts the fact that every $\alpha
$-clique has size $\leq m$. This contradiction shows that our assumption was
wrong. Hence, the union of the balls $E_{1},E_{2},\ldots,E_{m}$ is the whole
set $E$.]

We have now proved that the balls $E_{1},E_{2},\ldots,E_{m}$ are disjoint and
nonempty and their union is the whole set $E$. In other words, they form a set
partition of $E$. This proves Proposition \ref{prop.cliques.divide}
\textbf{(b)}. \medskip

\textbf{(c)} Proposition \ref{prop.cliques.divide} \textbf{(b)} shows that the
$m$ sets $E_{1},E_{2},\ldots,E_{m}$ form a set partition of $E$. Thus, these
$m$ sets are $m$ disjoint nonempty subsets of $E$; hence, each of them is a
\textbf{proper} subset of $E$ (since $m>1$). In other words, $E_{i}$ is a
proper subset of $E$ for each $i\in\left\{  1,2,\ldots,m\right\}  $. Hence,
$\left\vert E_{i}\right\vert <\left\vert E\right\vert $ for each $i\in\left\{
1,2,\ldots,m\right\}  $. This proves Proposition \ref{prop.cliques.divide}
\textbf{(c)}. \medskip

\textbf{(d)} Let $i\in\left\{  1,2,\ldots,m\right\}  $. Let $a,b\in E_{i}$ be
distinct. We must prove that $d\left(  a,b\right)  <\alpha$.

We have $b\in E_{i}=B_{\alpha}^{\circ}\left(  e_{i}\right)  $ (by the
definition of $E_{i}$) and thus $B_{\alpha}^{\circ}\left(  e_{i}\right)
=B_{\alpha}^{\circ}\left(  b\right)  $ (by Corollary \ref{cor.Bo.self-rad2},
applied to $e_{i}$ and $b$ instead of $e$ and $f$). But $a\in E_{i}=B_{\alpha
}^{\circ}\left(  e_{i}\right)  =B_{\alpha}^{\circ}\left(  b\right)  $. In
other words, $a=b$ or else $d\left(  a,b\right)  <\alpha$ (by the definition
of $B_{\alpha}^{\circ}\left(  b\right)  $). Hence, $d\left(  a,b\right)
<\alpha$ (since $a\neq b$). This proves Proposition \ref{prop.cliques.divide}
\textbf{(d)}. \medskip

\textbf{(e)} Let $i$ and $j$ be two distinct elements of $\left\{
1,2,\ldots,m\right\}  $. Let $a\in E_{i}$ and $b\in E_{j}$. We must prove that
$a\neq b$ and $d\left(  a,b\right)  =\alpha$.

We have $a\in E_{i}=B_{\alpha}^{\circ}\left(  e_{i}\right)  $ (by the
definition of $E_{i}$) and similarly $b\in B_{\alpha}^{\circ}\left(
e_{j}\right)  $. Furthermore, $e_{i}\neq e_{j}$ (since $i\neq j$ and since
$e_{1},e_{2},\ldots,e_{m}$ are distinct) and thus $d\left(  e_{i}%
,e_{j}\right)  =\alpha$ (since $\left\{  e_{1},e_{2},\ldots,e_{m}\right\}  $
is an $\alpha$-clique). Hence, Proposition \ref{prop.Bo.disjoint} \textbf{(a)}
(applied to $e_{i}$ and $e_{j}$ instead of $e$ and $f$) yields $a\neq b$ and
$d\left(  a,b\right)  \geq\alpha$. Combining $d\left(  a,b\right)  \geq\alpha$
with (\ref{pf.prop.cliques.divide.dleqa}), we obtain $d\left(  a,b\right)
=\alpha$. This proves Proposition \ref{prop.cliques.divide} \textbf{(e)}.
\medskip

\textbf{(f)} Let us first notice that the map $d_{i}$ (for each $i\in\left\{
1,2,\ldots,m\right\}  $) is defined to be a restriction of the map $d$. Thus,
for any $i\in\left\{  1,2,\ldots,m\right\}  $, we have
\begin{equation}
d_{i}\left(  e,f\right)  =d\left(  e,f\right)  \ \ \ \ \ \ \ \ \ \ \text{for
any two distinct }e,f\in E_{i}. \label{pf.prop.cliques.divide.e.dres}%
\end{equation}

The $\mathbb{V}$-ultra triple $\left(  E,w,d\right)  $ has a clique of size
$m$ (namely, $\left\{  e_{1},e_{2},\ldots,e_{m}\right\}  $), and a clique of
size $n_{i}$ for each $i\in\left\{  1,2,\ldots,m\right\}  $ (indeed,
$n_{i}=\operatorname*{mcs}\left(  E_{i},w_{i},d_{i}\right)  $ shows that the
$\mathbb{V}$-ultra triple $\left(  E_{i},w_{i},d_{i}\right)  $ has such a
clique; but this clique must of course be a clique of $\left(  E,w,d\right)  $
as well\footnote{since Proposition \ref{prop.cliques.sub} \textbf{(b)}
(applied to $\left(  F,w^{\prime},d^{\prime}\right)  =\left(  E_{i}%
,w_{i},d_{i}\right)  $) shows that any clique of the $\mathbb{V}$-ultra triple
$\left(  E_{i},w_{i},d_{i}\right)  $ is a clique of the $\mathbb{V}$-ultra
triple $\left(  E,w,d\right)  $}). Thus, the $\mathbb{V}$-ultra triple
$\left(  E,w,d\right)  $ has a clique of size $\max\left\{  m,n_{1}%
,n_{2},\ldots,n_{m}\right\}  $ (since $\max\left\{  m,n_{1},n_{2},\ldots
,n_{m}\right\}  $ is one of the numbers $m,n_{1},n_{2},\ldots,n_{m}$). Thus,%
\[
\operatorname*{mcs}\left(  E,w,d\right)  \geq\max\left\{  m,n_{1},n_{2}%
,\ldots,n_{m}\right\}  .
\]
It remains to prove the reverse inequality -- i.e., to prove that
$\operatorname*{mcs}\left(  E,w,d\right)  \leq\max\left\{  m,n_{1}%
,n_{2},\ldots,n_{m}\right\}  $.

Assume the contrary. Thus, $\operatorname*{mcs}\left(  E,w,d\right)
>\max\left\{  m,n_{1},n_{2},\ldots,n_{m}\right\}  $.

The $\mathbb{V}$-ultra triple $\left(  E,w,d\right)  $ has a clique $C$ of
size $\operatorname*{mcs}\left(  E,w,d\right)  $ (by the definition of
$\operatorname*{mcs}\left(  E,w,d\right)  $). Consider this $C$. Thus,
$\left\vert C\right\vert =\operatorname*{mcs}\left(  E,w,d\right)
>\max\left\{  m,n_{1},n_{2},\ldots,n_{m}\right\}  \geq0$. Hence, the set $C$
is nonempty. In other words, there exists some $a\in C$. Consider this $a$.

But recall that $E_{1},E_{2},\ldots,E_{m}$ form a set partition of $E$. Thus,
$E_{1}\cup E_{2}\cup\cdots\cup E_{m}=E$. Now, $a\in C\subseteq E=E_{1}\cup
E_{2}\cup\cdots\cup E_{m}$. In other words, $a\in E_{i}$ for some
$i\in\left\{  1,2,\ldots,m\right\}  $. Consider this $i$.

We are in one of the following two cases:

\textit{Case 1:} We have $C\subseteq E_{i}$.

\textit{Case 2:} We have $C\not \subseteq E_{i}$.

Let us first consider Case 1. In this case, we have $C\subseteq E_{i}$. In
other words, $C$ is a subset of $E_{i}$.

Recall that $C$ is a clique of $\left(  E,w,d\right)  $. Since $C$ is a subset
of $E_{i}$, we can thus conclude that $C$ is a clique of the $\mathbb{V}%
$-ultra triple $\left(  E_{i},w_{i},d_{i}\right)  $ (since Proposition
\ref{prop.cliques.sub} \textbf{(a)} (applied to $\left(  F,w^{\prime
},d^{\prime}\right)  =\left(  E_{i},w_{i},d_{i}\right)  $) shows that if a
subset of $E_{i}$ is a clique of $\left(  E,w,d\right)  $, then this subset is
also a clique of the $\mathbb{V}$-ultra triple $\left(  E_{i},w_{i}%
,d_{i}\right)  $). But the maximum size of such a clique is
$\operatorname*{mcs}\left(  E_{i},w_{i},d_{i}\right)  $ (by the definition of
$\operatorname*{mcs}\left(  E_{i},w_{i},d_{i}\right)  $). Hence, $\left\vert
C\right\vert \leq\operatorname*{mcs}\left(  E_{i},w_{i},d_{i}\right)  =n_{i}$.
This contradicts $\left\vert C\right\vert >\max\left\{  m,n_{1},n_{2}%
,\ldots,n_{m}\right\}  \geq n_{i}$. Thus, we have obtained a contradiction in
Case 1.

Let us next consider Case 2. In this case, we have $C\not \subseteq E_{i}$. In
other words, there exists some $b\in C$ such that $b\notin E_{i}$. Consider
this $b$. From $b\in C\subseteq E=E_{1}\cup E_{2}\cup\cdots\cup E_{m}$, we
conclude that $b\in E_{j}$ for some $j\in\left\{  1,2,\ldots,m\right\}  $.
Consider this $j$. If we had $i=j$, then we would have $b\in E_{j}=E_{i}$
(since $j=i$), which would contradict $b\notin E_{i}$. Hence, we cannot have
$i=j$. Hence, $i\neq j$. Thus, Proposition \ref{prop.cliques.divide}
\textbf{(e)} yields that $a\neq b$ and $d\left(  a,b\right)  =\alpha$. Now,
$a\neq b$ shows that $a$ and $b$ are two distinct elements of $C$. But $C$ is
a clique, and thus is a $\gamma$-clique for some $\gamma\in\mathbb{V}$.
Consider this $\gamma$. Since $C$ is a $\gamma$-clique, we have $d\left(
a,b\right)  =\gamma$ (since $a$ and $b$ are two distinct elements of $C$).
Hence, $\gamma=d\left(  a,b\right)  =\alpha$. Thus, $C$ is an $\alpha$-clique
(since $C$ is a $\gamma$-clique). Thus, $\left\vert C\right\vert \leq m$
(since every $\alpha$-clique has size $\leq m$). This contradicts $\left\vert
C\right\vert >\max\left\{  m,n_{1},n_{2},\ldots,n_{m}\right\}  \geq m$. Thus,
we have obtained a contradiction in Case 2.

We have thus found a contradiction in each of the two Cases 1 and 2. Thus, we
always have a contradiction. This completes the proof of Proposition
\ref{prop.cliques.divide} \textbf{(f)}.
\end{proof}

\begin{remark}
Here are some additional observations on Proposition \ref{prop.cliques.divide}%
, which we will not need (and thus will not prove):

\begin{enumerate}
\item[\textbf{(a)}] The set partition $\left\{  E_{1},E_{2},\ldots
,E_{m}\right\}  $ constructed in Proposition \ref{prop.cliques.divide}
\textbf{(b)} does not depend on the choice of $e_{1},e_{2},\ldots,e_{m}$.
Indeed, $E_{1},E_{2},\ldots,E_{m}$ are precisely the maximal (with respect to
inclusion) subsets $F$ of $E$ satisfying $\left(  d\left(  a,b\right)
<\alpha\text{ for any distinct }a,b\in F\right)  $.

\item[\textbf{(b)}] Applying Proposition \ref{prop.cliques.divide}
iteratively, we can see that a $\mathbb{V}$-ultra triple $\left(
E,w,d\right)  $ with finite $E$ has a recursive structure governed by a tree.
This idea is not new; see \cite{Lemin03} and \cite[\S 2--\S 3]{PetDov14} for
related results.
\end{enumerate}
\end{remark}

\section{\label{sect.val-rep}Valadic representation of $\mathbb{V}$-ultra
triples}

For the whole Section \ref{sect.val-rep}, we fix a field $\mathbb{K}$, and we
let $\mathbb{V}_{\geq0}$, $\mathbb{L}$, $\mathbb{L}_{+}$ and $t_{\alpha}$ be
as in Section \ref{sect.p-adic}. We also recall Definition \ref{def.padic-ut}.

\begin{definition}
Let $\gamma\in\mathbb{V}$ and $u\in\mathbb{L}$. We say that a valadic
$\mathbb{V}$-ultra triple $\left(  E,w,d\right)  $ is $\left(  \gamma
,u\right)  $\textit{-positioned} if%
\[
E\subseteq u+t_{-\gamma}\mathbb{L}_{+}.
\]

\end{definition}

In other words, a valadic $\mathbb{V}$-ultra triple $\left(  E,w,d\right)  $
is $\left(  \gamma,u\right)  $-positioned if and only if each element of $E$
has the form $u+\sum\limits_{\substack{\beta\in\mathbb{V};\\\beta\geq-\gamma
}}p_{\beta}t_{\beta}$ for some $p_{\beta}\in\mathbb{K}$.

\begin{theorem}
\label{thm.approx-p}Let $\left(  E,w,d\right)  $ be a $\mathbb{V}$-ultra
triple such that the set $E$ is finite. Let $\gamma\in\mathbb{V}$ be such that%
\begin{equation}
d\left(  a,b\right)  \leq\gamma\ \ \ \ \ \ \ \ \ \ \text{for each }\left(
a,b\right)  \in E\timesu E. \label{eq.thm.approx-p.ass}%
\end{equation}
Let $u\in\mathbb{L}$.

Assume that $\left\vert \mathbb{K}\right\vert \geq\operatorname*{mcs}\left(
E,w,d\right)  $. Then, there exists a $\left(  \gamma,u\right)  $-positioned
valadic $\mathbb{V}$-ultra triple isomorphic to $\left(  E,w,d\right)  $.
\end{theorem}

\begin{proof}
[Proof of Theorem \ref{thm.approx-p}.]We proceed by strong induction on
$\left\vert E\right\vert $.

If $\left\vert E\right\vert =1$, then this is clear (just take the obvious
valadic $\mathbb{V}$-ultra triple on the set $\left\{  u\right\}
\subseteq\mathbb{L}$, which is clearly $\left(  \gamma,u\right)
$-positioned). The case $\left\vert E\right\vert =0$ is even more obvious.
Thus, WLOG assume that $\left\vert E\right\vert >1$. Thus, the set $E\timesu
E$ is nonempty. Hence, $d\left(  E\timesu E\right)  $ is a nonempty finite
subset of $\mathbb{V}$, and therefore has a largest element. In other words,
$\max\left(  d\left(  E\timesu E\right)  \right)  $ is well-defined.

Let $\alpha=\max\left(  d\left(  E\timesu E\right)  \right)  $. Thus,
$\alpha\leq\gamma$ (by (\ref{eq.thm.approx-p.ass})).

Pick any maximum-size $\alpha$-clique, and write it in the form $\left\{
e_{1},e_{2},\ldots,e_{m}\right\}  $ for some distinct elements $e_{1}%
,e_{2},\ldots,e_{m}$ of $E$. For each $i\in\left\{  1,2,\ldots,m\right\}  $,
let $E_{i}$ be the open ball $B_{\alpha}^{\circ}\left(  e_{i}\right)  $, and
let $\left(  E_{i},w_{i},d_{i}\right)  $ be the $\mathbb{V}$-ultra triple
$\left(  E_{i},w\mid_{E_{i}},d\mid_{E_{i}\timesu E_{i}}\right)  $.

Then, Proposition \ref{prop.cliques.divide} \textbf{(a)} shows that $m>1$.
Moreover, Proposition \ref{prop.cliques.divide} \textbf{(b)} shows that the
$m$ sets $E_{1},E_{2},\ldots,E_{m}$ form a set partition of $E$. Hence,
$E=E_{1}\sqcup E_{2}\sqcup\cdots\sqcup E_{m}$ (an internal disjoint union).
Each set $E_{i}$ is the open ball $B_{\alpha}^{\circ}\left(  e_{i}\right)  $
and thus contains $e_{i}$.

Let $n_{i}=\operatorname*{mcs}\left(  E_{i},w_{i},d_{i}\right)  $ for each
$i\in\left\{  1,2,\ldots,m\right\}  $. Then,%
\[
\left\vert \mathbb{K}\right\vert \geq\operatorname*{mcs}\left(  E,w,d\right)
=\max\left\{  m,n_{1},n_{2},\ldots,n_{m}\right\}  \ \ \ \ \ \ \ \ \ \ \left(
\text{by Proposition \ref{prop.cliques.divide} \textbf{(f)}}\right)  .
\]

For any subset $Z$ of $\mathbb{L}$, we define a distance function
$\overline{d}_{Z}:Z\timesu Z\rightarrow\mathbb{V}$ by setting%
\begin{equation}
\overline{d}_{Z}\left(  a,b\right)  =-\operatorname*{ord}\left(  a-b\right)
\ \ \ \ \ \ \ \ \ \ \text{for all }\left(  a,b\right)  \in Z\timesu Z.
\label{pf.thm.approx-p.dZ}%
\end{equation}
Note that the distance function of any valadic $\mathbb{V}$-ultra triple is
precisely $\overline{d}_{Z}$, where $Z$ is the ground set of this $\mathbb{V}%
$-ultra triple.

We have $\left\vert \mathbb{K}\right\vert \geq\max\left\{  m,n_{1}%
,n_{2},\ldots,n_{m}\right\}  \geq m$. Hence, there exist $m$ distinct elements
$\lambda_{1},\lambda_{2},\ldots,\lambda_{m}$ of $\mathbb{K}$. Fix $m$ such
elements. Define $m$ elements $u_{1},u_{2},\ldots,u_{m}$ of $\mathbb{L}$ by%
\begin{equation}
u_{i}=u+\lambda_{i}t_{-\alpha}\ \ \ \ \ \ \ \ \ \ \text{for each }i\in\left\{
1,2,\ldots,m\right\}  \text{.} \label{pf.thm.approx-p.ui=}%
\end{equation}

Let $\mathbb{L}_{++}$ denote the $\mathbb{K}$-submodule of $\mathbb{L}_{+}$
generated by $t_{\delta}$ for all positive $\delta\in\mathbb{V}$. (Of course,
$\delta\in\mathbb{V}$ is said to be positive if and only if $\delta>0$.) It is
easy to see that $\mathbb{L}_{++}$ is an ideal of $\mathbb{L}_{+}$. (Actually,
$\mathbb{L}_{++}=\operatorname*{Ker}\pi$, where $\pi$ is as defined in Lemma
\ref{lem.piLK}.) Hence, $\mathbb{L}_{++}\mathbb{L}_{+}\subseteq\mathbb{L}%
_{++}$.

Let $i\in\left\{  1,2,\ldots,m\right\}  $. We shall work under the assumption
that $\left\vert E_{i}\right\vert >1$; we will later explain how to proceed
without it.

The set $E_{i}\timesu E_{i}$ is nonempty (since $\left\vert E_{i}\right\vert
>1$) and finite (since $E_{i}$ is finite). Hence, the set $d\left(
E_{i}\timesu E_{i}\right)  $ is nonempty and finite as well. Thus, it has a
well-defined largest element $\max\left(  d\left(  E_{i}\timesu E_{i}\right)
\right)  $. Let us denote this largest element by $\beta$.

From Proposition \ref{prop.cliques.divide} \textbf{(d)}, it follows easily see
that $\beta<\alpha$. [\textit{Proof:} We defined $\beta$ by $\beta=\max\left(
d\left(  E_{i}\timesu E_{i}\right)  \right)  $. Thus, $\beta\in d\left(
E_{i}\timesu E_{i}\right)  $. In other words, $\beta=d\left(  a,b\right)  $
for some two distinct elements $a$ and $b$ of $E_{i}$. However, Proposition
\ref{prop.cliques.divide} \textbf{(d)} yields that these two elements $a$ and
$b$ satisfy $d\left(  a,b\right)  <\alpha$, and thus we conclude that
$\beta=d\left(  a,b\right)  <\alpha$.]

From $\beta<\alpha$, we obtain $\alpha-\beta>0$ and therefore $t_{\alpha
-\beta}\in\mathbb{L}_{++}$. Hence,%
\begin{equation}
t_{-\beta}=t_{-\alpha}\underbrace{t_{\alpha-\beta}}_{\in\mathbb{L}_{++}}\in
t_{-\alpha}\mathbb{L}_{++}.\label{pf.thm.approx-p.t-beta}%
\end{equation}

For any $\left(  a,b\right)  \in E_{i}\timesu E_{i}$, we have $d\left(
a,b\right)  \in d\left(  E_{i}\timesu E_{i}\right)  $ and thus%
\[
d\left(  a,b\right)  \leq\left(  \text{the largest element of }d\left(
E_{i}\timesu E_{i}\right)  \right)  =\max\left(  d\left(  E_{i}\timesu
E_{i}\right)  \right)  =\beta,
\]
and therefore%
\begin{align*}
d_{i}\left(  a,b\right)    & =d\left(  a,b\right)  \ \ \ \ \ \ \ \ \ \ \left(
\text{since }d_{i}=d\mid_{E_{i}\timesu E_{i}}\right)  \\
& \leq\beta.
\end{align*}
We thus have proved that%
\[
d_{i}\left(  a,b\right)  \leq\beta\ \ \ \ \ \ \ \ \ \ \text{for each }\left(
a,b\right)  \in E_{i}\timesu E_{i}.
\]
Also, $\left\vert E_{i}\right\vert <\left\vert E\right\vert $ (by Proposition
\ref{prop.cliques.divide} \textbf{(c)}) and%
\[
\left\vert \mathbb{K}\right\vert \geq\max\left\{  m,n_{1},n_{2},\ldots
,n_{m}\right\}  \geq n_{i}=\operatorname*{mcs}\left(  E_{i},w_{i}%
,d_{i}\right)  .
\]
Hence, the induction hypothesis shows that we can apply Theorem
\ref{thm.approx-p} to $\beta$, $u_{i}$ and $\left(  E_{i},w_{i},d_{i}\right)
$ instead of $\gamma$, $u$ and $\left(  E,w,d\right)  $. We thus conclude that
there exists a $\left(  \beta,u_{i}\right)  $-positioned valadic $\mathbb{V}%
$-ultra triple $\left(  \overline{E}_{i},\overline{w}_{i},\overline{d}%
_{i}\right)  $ isomorphic to $\left(  E_{i},w_{i},d_{i}\right)  $. Consider
this $\left(  \overline{E}_{i},\overline{w}_{i},\overline{d}_{i}\right)  $.
The $\mathbb{V}$-ultra triple $\left(  E_{i},w_{i},d_{i}\right)  $ is
isomorphic to $\left(  \overline{E}_{i},\overline{w}_{i},\overline{d}%
_{i}\right)  $ (since $\left(  \overline{E}_{i},\overline{w}_{i},\overline
{d}_{i}\right)  $ is isomorphic to $\left(  E_{i},w_{i},d_{i}\right)  $, but
being isomorphic is a symmetric relation). In other words, there exists an
isomorphism $f_{i}:E_{i}\rightarrow\overline{E}_{i}$ of $\mathbb{V}$-ultra
triples from $\left(  E_{i},w_{i},d_{i}\right)  $ to $\left(  \overline{E}%
_{i},\overline{w}_{i},\overline{d}_{i}\right)  $. Consider this $f_{i}$. Since
the $\mathbb{V}$-ultra triple $\left(  \overline{E}_{i},\overline{w}%
_{i},\overline{d}_{i}\right)  $ is $\left(  \beta,u_{i}\right)  $-positioned,
we have%
\[
\overline{E}_{i}\subseteq u_{i}+\underbrace{t_{-\beta}}_{\substack{\in
t_{-\alpha}\mathbb{L}_{++}\\\text{(by (\ref{pf.thm.approx-p.t-beta}))}%
}}\mathbb{L}_{+}\subseteq u_{i}+t_{-\alpha}\underbrace{\mathbb{L}%
_{++}\mathbb{L}_{+}}_{\subseteq\mathbb{L}_{++}}\subseteq u_{i}+t_{-\alpha
}\mathbb{L}_{++}.
\]

Hence, we have found a $\mathbb{V}$-ultra triple $\left(  E_{i},w_{i}%
,d_{i}\right)  $ and a valadic $\mathbb{V}$-ultra triple $\left(  \overline
{E}_{i},\overline{w}_{i},\overline{d}_{i}\right)  $ satisfying
\[
\overline{E}_{i}\subseteq u_{i}+t_{-\alpha}\mathbb{L}_{++},
\]
and an isomorphism $f_{i}:E_{i}\rightarrow\overline{E}_{i}$ of $\mathbb{V}%
$-ultra triples from $\left(  E_{i},w_{i},d_{i}\right)  $ to $\left(
\overline{E}_{i},\overline{w}_{i},\overline{d}_{i}\right)  $. We have done so
assuming that $\left\vert E_{i}\right\vert >1$; but this is even easier when
$\left\vert E_{i}\right\vert $ is $\leq1$ instead\footnote{\textit{Proof.}
Assume that $\left\vert E_{i}\right\vert \leq1$. Since $e_{i}$ is an element
of $E_{i}$ (because $E_{i}=B_{\alpha}^{\circ}\left(  e_{i}\right)  $), we thus
have $E_{i}=\left\{  e_{i}\right\}  $. Now, set $\overline{E}_{i}=\left\{
u_{i}\right\}  $; let $\overline{w}_{i}:\overline{E}_{i}\rightarrow\mathbb{V}$
be the map that sends $u_{i}$ to $w_{i}\left(  e_{i}\right)  $; let
$\overline{d}_{i}:\overline{E}_{i}\timesu\overline{E}_{i}\rightarrow
\mathbb{V}$ be the distance function $\overline{d}_{\left\{  u_{i}\right\}  }%
$; and let $f_{i}:E_{i}\rightarrow\overline{E}_{i}$ be the map that sends
$e_{i}$ to $u_{i}$. Then, it is easy to see that $\left(  \overline{E}%
_{i},\overline{w}_{i},\overline{d}_{i}\right)  $ is a valadic $\mathbb{V}%
$-ultra triple satisfying $\overline{E}_{i}\subseteq u_{i}+t_{-\alpha
}\mathbb{L}_{++}$ (since $\overline{E}_{i}=\left\{  u_{i}\right\}
=u_{i}+\underbrace{0}_{\subseteq t_{-\alpha}\mathbb{L}_{++}}\subseteq
u_{i}+t_{-\alpha}\mathbb{L}_{++}$), and that the map $f_{i}:E_{i}%
\rightarrow\overline{E}_{i}$ is an isomorphism of $\mathbb{V}$-ultra triples
from $\left(  E_{i},w_{i},d_{i}\right)  $ to $\left(  \overline{E}%
_{i},\overline{w}_{i},\overline{d}_{i}\right)  $ (since $E_{i}=\left\{
e_{i}\right\}  $ and $\overline{E}_{i}=\left\{  u_{i}\right\}  $, so that the
maps $d_{i}$ and $\overline{d}_{i}$ both have no values, whereas the map
$\overline{w}_{i}$ is defined in such a way that $\overline{w}_{i}\left(
u_{i}\right)  =w_{i}\left(  e_{i}\right)  $ and thus $\overline{w}_{i}\circ
f_{i}=w_{i}$). Thus, everything we constructed above still exists when
$\left\vert E_{i}\right\vert \leq1$.}.

Forget that we fixed $i$. Thus, for each $i\in\left\{  1,2,\ldots,m\right\}
$, we have constructed a $\mathbb{V}$-ultra triple $\left(  E_{i},w_{i}%
,d_{i}\right)  $ and a valadic $\mathbb{V}$-ultra triple $\left(  \overline
{E}_{i},\overline{w}_{i},\overline{d}_{i}\right)  $ satisfying
\[
\overline{E}_{i}\subseteq u_{i}+t_{-\alpha}\mathbb{L}_{++}%
\]
and an isomorphism $f_{i}:E_{i}\rightarrow\overline{E}_{i}$ of $\mathbb{V}%
$-ultra triples from $\left(  E_{i},w_{i},d_{i}\right)  $ to $\left(
\overline{E}_{i},\overline{w}_{i},\overline{d}_{i}\right)  $.

For later use, let us observe the following:

\begin{statement}
\textit{Claim 1:} Let $i$ and $j$ be two distinct elements of $\left\{
1,2,\ldots,m\right\}  $. Let $a\in\overline{E}_{i}$ and $b\in\overline{E}_{j}%
$. Then, $a-b\neq0$ and $\operatorname*{ord}\left(  a-b\right)  =-\alpha$.
\end{statement}

[\textit{Proof of Claim 1:} From $i\neq j$, we conclude that $\lambda_{i}$ and
$\lambda_{j}$ are two distinct elements of $\mathbb{K}$ (since $\lambda
_{1},\lambda_{2},\ldots,\lambda_{m}$ are distinct elements of $\mathbb{K}$).
Hence, $\lambda_{i}-\lambda_{j}$ is a nonzero element of $\mathbb{K}$.

We have $a\in\overline{E}_{i}\subseteq u_{i}+t_{-\alpha}\mathbb{L}_{++}$, so
that $a-u_{i}\in t_{-\alpha}\mathbb{L}_{++}$. Similarly, $b-u_{j}\in
t_{-\alpha}\mathbb{L}_{++}$. Also, (\ref{pf.thm.approx-p.ui=}) yields
$u_{i}=u+\lambda_{i}t_{-\alpha}$. Likewise, $u_{j}=u+\lambda_{j}t_{-\alpha}$.
Subtracting the latter two equalities from one another, we obtain%
\[
u_{i}-u_{j}=\left(  u+\lambda_{i}t_{-\alpha}\right)  -\left(  u+\lambda
_{j}t_{-\alpha}\right)  =\left(  \lambda_{i}-\lambda_{j}\right)  t_{-\alpha}.
\]
Now,%
\begin{align*}
a-b  &  =\underbrace{\left(  u_{i}-u_{j}\right)  }_{=\left(  \lambda
_{i}-\lambda_{j}\right)  t_{-\alpha}}+\underbrace{\left(  a-u_{i}\right)
}_{\in t_{-\alpha}\mathbb{L}_{++}}-\underbrace{\left(  b-u_{j}\right)  }_{\in
t_{-\alpha}\mathbb{L}_{++}}\\
&  \in\left(  \lambda_{i}-\lambda_{j}\right)  t_{-\alpha}%
+\underbrace{t_{-\alpha}\mathbb{L}_{++}-t_{-\alpha}\mathbb{L}_{++}}_{\subseteq
t_{-\alpha}\mathbb{L}_{++}}\subseteq\left(  \lambda_{i}-\lambda_{j}\right)
t_{-\alpha}+t_{-\alpha}\mathbb{L}_{++}.
\end{align*}
In other words, $a-b$ can be written in the form
\[
a-b=\left(  \lambda_{i}-\lambda_{j}\right)  t_{-\alpha}+\left(  \text{a
}\mathbb{K}\text{-linear combination of }t_{\eta}\text{ with }\eta
>-\alpha\right)
\]
(since the elements of $t_{-\alpha}\mathbb{L}_{++}$ are precisely the
$\mathbb{K}$-linear combinations of $t_{\eta}$ with $\eta>-\alpha$).

From this, we obtain two things: First, we obtain that $\left[  t_{-\alpha
}\right]  \left(  a-b\right)  =\lambda_{i}-\lambda_{j}$ (since $\lambda
_{i}-\lambda_{j}\in\mathbb{K}$), hence $\left[  t_{-\alpha}\right]  \left(
a-b\right)  =\lambda_{i}-\lambda_{j}\neq0$ (since $\lambda_{i}-\lambda_{j}$ is
nonzero), and thus $a-b\neq0$. Furthermore, we obtain that
$\operatorname*{ord}\left(  a-b\right)  =-\alpha$ (again since $\lambda
_{i}-\lambda_{j}\neq0$). Thus, Claim 1 is proved.] \medskip

Let $\overline{E}$ denote the subset
\[
\overline{E}_{1}\cup\overline{E}_{2}\cup\cdots\cup\overline{E}_{m}%
\]
of $\mathbb{L}$. Note that the sets $\overline{E}_{1},\overline{E}_{2}%
,\ldots,\overline{E}_{m}$ are disjoint\footnote{\textit{Proof.} Assume the
contrary. Thus, there exist some distinct $i,j\in\left\{  1,2,\ldots
,m\right\}  $ such that $\overline{E}_{i}\cap\overline{E}_{j}\neq\varnothing$.
Consider these $i$ and $j$. There exists some $a\in\overline{E}_{i}%
\cap\overline{E}_{j}$ (since $\overline{E}_{i}\cap\overline{E}_{j}%
\neq\varnothing$). Consider this $a$. We have $a\in\overline{E}_{i}%
\cap\overline{E}_{j}\subseteq\overline{E}_{i}$ and similarly $a\in\overline
{E}_{j}$. Hence, Claim 1 (applied to $b=a$) yields $a-a\neq0$ and
$\operatorname*{ord}\left(  a-a\right)  =-\alpha$. But $a-a\neq0$ clearly
contradicts $a-a=0$. This contradiction proves that our assumption was false,
qed.}. Hence, $\overline{E}=\overline{E}_{1}\sqcup\overline{E}_{2}\sqcup
\cdots\sqcup\overline{E}_{m}$ (an internal disjoint union). Moreover, each
$i\in\left\{  1,2,\ldots,m\right\}  $ satisfies%
\begin{align}
\overline{E}_{i} &  \subseteq\underbrace{u_{i}}_{\substack{=u+\lambda
_{i}t_{-\alpha}\\\text{(by (\ref{pf.thm.approx-p.ui=}))}}}+\,t_{-\alpha
}\mathbb{L}_{++}=u+\underbrace{\lambda_{i}t_{-\alpha}+t_{-\alpha}%
\mathbb{L}_{++}}_{=t_{-\alpha}\left(  \lambda_{i}+\mathbb{L}_{++}\right)
}\nonumber\\
&  =u+\underbrace{t_{-\alpha}}_{\substack{\in t_{-\gamma}\mathbb{L}%
_{+}\\\text{(since }-\alpha\geq-\gamma\\\text{(because }\alpha\leq
\gamma\text{))}}}\ \ \underbrace{\left(  \lambda_{i}+\mathbb{L}_{++}\right)
}_{\substack{\subseteq\mathbb{L}_{+}\\\text{(since }\lambda_{i}\in
\mathbb{K}\subseteq\mathbb{L}_{+}\text{)}}}\subseteq u+t_{-\gamma
}\underbrace{\mathbb{L}_{+}\mathbb{L}_{+}}_{\subseteq\mathbb{L}_{+}%
}\nonumber\\
&  \subseteq u+t_{-\gamma}\mathbb{L}_{+}.\label{pf.thm.approx-p.Eiin3}%
\end{align}
Now,%
\[
\overline{E}=\overline{E}_{1}\cup\overline{E}_{2}\cup\cdots\cup\overline
{E}_{m}=\bigcup_{i=1}^{m}\ \ \underbrace{\overline{E}_{i}}%
_{\substack{\subseteq u+t_{-\gamma}\mathbb{L}_{+}\\\text{(by
(\ref{pf.thm.approx-p.Eiin3}))}}}\subseteq\bigcup_{i=1}^{m}\left(
u+t_{-\gamma}\mathbb{L}_{+}\right)  \subseteq u+t_{-\gamma}\mathbb{L}_{+}.
\]

Now, recall that $E=E_{1}\sqcup E_{2}\sqcup\cdots\sqcup E_{m}$ and
$\overline{E}=\overline{E}_{1}\sqcup\overline{E}_{2}\sqcup\cdots
\sqcup\overline{E}_{m}$. Thus, we can glue the bijections $f_{i}%
:E_{i}\rightarrow\overline{E}_{i}$ together to a single bijection%
\begin{align*}
f:E  &  \rightarrow\overline{E},\\
a  &  \mapsto f_{i}\left(  a\right)  ,\ \ \ \ \ \ \ \ \ \ \text{where }%
i\in\left\{  1,2,\ldots,m\right\}  \text{ is such that }a\in E_{i}.
\end{align*}
Consider this $f$. Define a weight function $\overline{w}:\overline
{E}\rightarrow\mathbb{V}$ by setting $\overline{w}=w\circ f^{-1}$. Then,
$\left(  \overline{E},\overline{w},\overline{d}_{\overline{E}}\right)  $ is a
$\left(  \gamma,u\right)  $-positioned valadic $\mathbb{V}$-ultra triple. (It
is $\left(  \gamma,u\right)  $-positioned, since $\overline{E}\subseteq
u+t_{-\gamma}\mathbb{L}_{+}$.) Moreover, it is easy to see that every
$i,j\in\left\{  1,2,\ldots,m\right\}  $ and any two distinct elements
$a\in\overline{E}_{i}$ and $b\in\overline{E}_{j}$ satisfy%
\begin{equation}
\overline{d}_{\overline{E}}\left(  a,b\right)  =%
\begin{cases}
\alpha, & \text{if }i\neq j;\\
\overline{d}_{\overline{E}_{i}}\left(  a,b\right)  , & \text{if }i=j
\end{cases}
\label{pf.thm.approx-p.dE5}%
\end{equation}
\footnote{\textit{Proof of (\ref{pf.thm.approx-p.dE5}):} Let $i,j\in\left\{
1,2,\ldots,m\right\}  $, and let $a\in\overline{E}_{i}$ and $b\in\overline
{E}_{j}$ be two distinct elements. We must prove (\ref{pf.thm.approx-p.dE5}).
It clearly suffices to verify the equality
\[
\operatorname*{ord}\left(  a-b\right)  =%
\begin{cases}
-\alpha, & \text{if }i\neq j;\\
\operatorname*{ord}\left(  a-b\right)  , & \text{if }i=j
\end{cases}
\]
(since both $\overline{d}_{\overline{E}}$ and $\overline{d}_{\overline{E}_{i}%
}$ are given by the formula (\ref{pf.thm.approx-p.dZ})). This equality is
obviously true in the case when $i=j$; on the other hand, it follows from
Claim 1 in the case when $i\neq j$. Hence, (\ref{pf.thm.approx-p.dE5}) is
proved.}. From this, it is easy to see that%
\begin{equation}
\overline{d}_{\overline{E}}\left(  f\left(  a\right)  ,f\left(  b\right)
\right)  =d\left(  a,b\right)  \ \ \ \ \ \ \ \ \ \ \text{for all }\left(
a,b\right)  \in E\timesu E \label{pf.thm.approx-p.dE6}%
\end{equation}
\footnote{\textit{Proof of (\ref{pf.thm.approx-p.dE6}):} Let $\left(
a,b\right)  \in E\timesu E$. Thus, $a$ and $b$ are two distinct elements of
$E$. Let $i,j\in\left\{  1,2,\ldots,m\right\}  $ be such that $a\in E_{i}$ and
$b\in E_{j}$. (These $i$ and $j$ clearly exist, because $a$ and $b$ both
belong to $E=E_{1}\sqcup E_{2}\sqcup\cdots\sqcup E_{m}$.) Then, the definition
of $f$ yields $f\left(  a\right)  =f_{i}\left(  a\right)  $ and $f\left(
b\right)  =f_{j}\left(  b\right)  $. Hence,%
\begin{align}
\overline{d}_{\overline{E}}\left(  f\left(  a\right)  ,f\left(  b\right)
\right)   &  =\overline{d}_{\overline{E}}\left(  f_{i}\left(  a\right)
,f_{j}\left(  b\right)  \right) \nonumber\\
&  =%
\begin{cases}
\alpha, & \text{if }i\neq j;\\
\overline{d}_{\overline{E}_{i}}\left(  f_{i}\left(  a\right)  ,f_{j}\left(
b\right)  \right)  , & \text{if }i=j
\end{cases}
\label{pf.thm.approx-p.dE6.pf.1}%
\end{align}
(by (\ref{pf.thm.approx-p.dE5}), applied to $f_{i}\left(  a\right)  $ and
$f_{j}\left(  b\right)  $ instead of $a$ and $b$). If $i\neq j$, then this
becomes%
\[
\overline{d}_{\overline{E}}\left(  f\left(  a\right)  ,f\left(  b\right)
\right)  =\alpha=d\left(  a,b\right)  \ \ \ \ \ \ \ \ \ \ \left(  \text{since
Proposition \ref{prop.cliques.divide} \textbf{(e)} yields }d\left(
a,b\right)  =\alpha\right)  ,
\]
and thus (\ref{pf.thm.approx-p.dE6}) is proven in this case. Hence, for the
rest of this proof of (\ref{pf.thm.approx-p.dE6}), we WLOG assume that $i=j$.
Hence, $b\in E_{j}=E_{i}$ (since $j=i$). Recall that $f_{i}:E_{i}%
\rightarrow\overline{E}_{i}$ is an isomorphism of $\mathbb{V}$-ultra triples
from $\left(  E_{i},w_{i},d_{i}\right)  $ to $\left(  \overline{E}%
_{i},\overline{w}_{i},\overline{d}_{i}\right)  $. Hence,
\[
\overline{d}_{i}\left(  f_{i}\left(  a\right)  ,f_{i}\left(  b\right)
\right)  =d_{i}\left(  a,b\right)  =d\left(  a,b\right)
\ \ \ \ \ \ \ \ \ \ \left(  \text{by the definition of }d_{i}\right)  .
\]
Also, $\overline{d}_{i}=\overline{d}_{\overline{E}_{i}}$ (since the
$\mathbb{V}$-ultra triple $\left(  \overline{E}_{i},\overline{w}_{i}%
,\overline{d}_{i}\right)  $ is valadic). Now, (\ref{pf.thm.approx-p.dE6.pf.1})
becomes%
\begin{align*}
\overline{d}_{\overline{E}}\left(  f\left(  a\right)  ,f\left(  b\right)
\right)   &  =%
\begin{cases}
\alpha, & \text{if }i\neq j;\\
\overline{d}_{\overline{E}_{i}}\left(  f_{i}\left(  a\right)  ,f_{j}\left(
b\right)  \right)  , & \text{if }i=j
\end{cases}
\\
&  =\overline{d}_{\overline{E}_{i}}\left(  f_{i}\left(  a\right)
,f_{j}\left(  b\right)  \right)  \ \ \ \ \ \ \ \ \ \ \left(  \text{since
}i=j\right) \\
&  =\underbrace{\overline{d}_{\overline{E}_{i}}}_{=\overline{d}_{i}}\left(
f_{i}\left(  a\right)  ,f_{i}\left(  b\right)  \right)
\ \ \ \ \ \ \ \ \ \ \left(  \text{since }j=i\right) \\
&  =\overline{d}_{i}\left(  f_{i}\left(  a\right)  ,f_{i}\left(  b\right)
\right)  =d\left(  a,b\right)  .
\end{align*}
Thus, (\ref{pf.thm.approx-p.dE6}) is proven.}. Moreover, $\overline{w}\circ
f=w$ (since $\overline{w}=w\circ f^{-1}$). Hence, the bijection
$f:E\rightarrow\overline{E}$ is an isomorphism of $\mathbb{V}$-ultra triples
from $\left(  E,w,d\right)  $ to $\left(  \overline{E},\overline{w}%
,\overline{d}_{\overline{E}}\right)  $. Hence, there exists a $\left(
\gamma,u\right)  $-positioned valadic $\mathbb{V}$-ultra triple isomorphic to
$\left(  E,w,d\right)  $ (namely, $\left(  \overline{E},\overline{w}%
,\overline{d}_{\overline{E}}\right)  $). This proves Theorem
\ref{thm.approx-p} for our $\left(  E,w,d\right)  $. Thus, the induction step
is complete, and Theorem \ref{thm.approx-p} is proven.
\end{proof}

\section{\label{sect.main-proof}Proof of the main theorem}

We can now prove Theorem \ref{thm.bh-geg-mcs} (from which we will immediately
obtain Theorem \ref{thm.bh-geg}).

\begin{proof}
[Proof of Theorem \ref{thm.bh-geg-mcs}.]Pick some $\gamma\in\mathbb{V}$ such
that%
\[
d\left(  a,b\right)  \leq\gamma\ \ \ \ \ \ \ \ \ \ \text{for each }\left(
a,b\right)  \in E\timesu E.
\]
(Such a $\gamma$ clearly exists, since the set $E\timesu E$ is finite and the
set $\mathbb{V}$ is totally ordered.)

Theorem \ref{thm.approx-p} (applied to $u=0$) thus yields that there exists a
$\left(  \gamma,0\right)  $-positioned valadic $\mathbb{V}$-ultra triple
isomorphic to $\left(  E,w,d\right)  $. Consider this valadic $\mathbb{V}%
$-ultra triple, and denote it by $\left(  F,v,c\right)  $. Let $\mathcal{E}$
denote the Bhargava greedoid of this $\mathbb{V}$-ultra triple $\left(
F,v,c\right)  $. The set $F$ is a subset of $\mathbb{L}$ (since $\left(
F,v,c\right)  $ is a valadic $\mathbb{V}$-ultra triple) and is finite (since
$\left(  F,v,c\right)  $ is isomorphic to $\left(  E,w,d\right)  $, whence
$\left\vert F\right\vert =\left\vert E\right\vert <\infty$). Moreover, the
distance function $c$ of the $\mathbb{V}$-ultra triple $\left(  F,v,c\right)
$ is the function $d$ from Definition \ref{def.padic-ut} (since $\left(
F,v,c\right)  $ is a valadic $\mathbb{V}$-ultra triple). Hence, Theorem
\ref{thm.poly-ut-greed} (applied to $\left(  F,v,c\right)  $ instead of
$\left(  E,w,d\right)  $) yields that the Bhargava greedoid of $\left(
F,v,c\right)  $ is the Gaussian elimination greedoid of a vector family over
$\mathbb{K}$. In other words, $\mathcal{E}$ is the Gaussian elimination
greedoid of a vector family over $\mathbb{K}$ (since the Bhargava greedoid of
$\left(  F,v,c\right)  $ is $\mathcal{E}$).

But Proposition \ref{prop.iso.bh-gr} yields that the Bhargava greedoids of
$\left(  E,w,d\right)  $ and $\left(  F,v,c\right)  $ are isomorphic. In other
words, the set systems $\mathcal{F}$ and $\mathcal{E}$ are isomorphic (since
$\mathcal{F}$ and $\mathcal{E}$ are the Bhargava greedoids of $\left(
E,w,d\right)  $ and $\left(  F,v,c\right)  $, respectively). In other words,
the set systems $\mathcal{E}$ and $\mathcal{F}$ are isomorphic.

Hence, Proposition \ref{prop.iso.gauss} shows that $\mathcal{F}$ is the
Gaussian elimination greedoid of a vector family over $\mathbb{K}$. This
proves Theorem \ref{thm.bh-geg-mcs}.
\end{proof}

\begin{proof}
[Proof of Theorem \ref{thm.bh-geg}.]We have $\left\vert \mathbb{K}\right\vert
\geq\left\vert E\right\vert \geq\operatorname*{mcs}\left(  E,w,d\right)  $.
Thus, Theorem \ref{thm.bh-geg-mcs} yields that $\mathcal{F}$ is the Gaussian
elimination greedoid of a vector family over $\mathbb{K}$. This proves Theorem
\ref{thm.bh-geg}.
\end{proof}

\section{\label{sect.converse}Proof of Theorem \ref{thm.converse1}}

For the rest of Section \ref{sect.converse}, we fix a $\mathbb{V}$-ultra
triple $\left(  E,w,d\right)  $.

Our next goal is to prove Theorem \ref{thm.converse1}. We have to build
several tools to this purpose.

\subsection{Closed balls}

We will use a counterpart to the concept of open balls: the notion of
\textit{closed balls}. To wit, it is defined as follows:

\begin{definition}
\label{def.closed-ball}Let $\alpha\in\mathbb{V}$ and $e\in E$. The
\textit{closed ball }$B_{\alpha}\left(  e\right)  $ is defined to be the
subset%
\[
\left\{  f\in E\ \mid\ f=e\text{ or else }d\left(  f,e\right)  \leq
\alpha\right\}
\]
of $E$.
\end{definition}

Clearly, for each $\alpha\in\mathbb{V}$ and each $e\in E$, we have $e\in
B_{\alpha}\left(  e\right)  $, so that the closed ball $B_{\alpha}\left(
e\right)  $ contains at least the element $e$.

Most properties of open balls have analogues for closed balls. In particular,
here is an analogue of Proposition \ref{prop.Bo.self-rad}:

\begin{proposition}
\label{prop.Bc.self-rad}Let $\alpha\in\mathbb{V}$ and $e,f\in E$ be such that
$e\neq f$ and $d\left(  e,f\right)  \leq\alpha$. Then, $B_{\alpha}\left(
e\right)  =B_{\alpha}\left(  f\right)  $.
\end{proposition}

\begin{proof}
This can be proved by a straightforward modification of the above proof of
Proposition \ref{prop.Bo.self-rad} (namely, all \textquotedblleft%
$<$\textquotedblright\ signs are replaced by \textquotedblleft$\leq
$\textquotedblright\ signs).
\end{proof}

Next comes an analogue of Corollary \ref{cor.Bo.self-rad2}:

\begin{corollary}
\label{cor.Bc.self-rad2}Let $\alpha\in\mathbb{V}$ and $e\in E$. Let $f\in
B_{\alpha}\left(  e\right)  $. Then, $B_{\alpha}\left(  e\right)  =B_{\alpha
}\left(  f\right)  $.
\end{corollary}

\begin{proof}
This can be proved by a straightforward modification of the above proof of
Corollary \ref{cor.Bo.self-rad2} (namely, all \textquotedblleft$<$%
\textquotedblright\ signs are replaced by \textquotedblleft$\leq
$\textquotedblright\ signs).
\end{proof}

Knowing these properties, we can easily obtain the following lemma:

\begin{lemma}
\label{lem.clique-and-clball}Let $\beta\in\mathbb{V}$. Let $C$ be a $\beta$-clique.

\begin{enumerate}
\item[\textbf{(a)}] The closed balls $B_{\beta}\left(  c\right)  $ for all
$c\in C$ are identical.
\end{enumerate}

Now, let $B=B_{\beta}\left(  c\right)  $ for some $c\in C$. Then:

\begin{enumerate}
\item[\textbf{(b)}] We have $C\subseteq B$.

\item[\textbf{(c)}] For any distinct elements $p,q\in B$, we have $d\left(
p,q\right)  \leq\beta$.

\item[\textbf{(d)}] For any $n\in E\setminus B$ and any $p,q\in B$, we have
$d\left(  n,p\right)  =d\left(  n,q\right)  $.
\end{enumerate}
\end{lemma}

(Intuitively, it helps to think of a clique $C$ as an ultrametric analogue of
a sphere, and of the set $B$ constructed in Lemma \ref{lem.clique-and-clball}
as being the whole closed ball whose boundary is this sphere. Of course, this
must not be taken literally; in particular, every point in this ball serves as
the \textquotedblleft center\textquotedblright\ of this ball, so to speak.)

\begin{proof}
[Proof of Lemma \ref{lem.clique-and-clball}.]We know that $C$ is a $\beta
$-clique. In other words, $C$ is a subset of $E$ such that%
\begin{equation}
\text{any two distinct elements }a,b\in C\text{ satisfy }d\left(  a,b\right)
=\beta\label{pf.lem.clique-and-clball.1}%
\end{equation}
(by the definition of a \textquotedblleft$\beta$-clique\textquotedblright).
\medskip

\textbf{(a)} We must prove that $B_{\beta}\left(  e\right)  =B_{\beta}\left(
f\right)  $ for any $e,f\in C$.

So let $e,f\in C$. We must prove that $B_{\beta}\left(  e\right)  =B_{\beta
}\left(  f\right)  $. If $e=f$, then this is obvious. Hence, we WLOG assume
that $e\neq f$. Thus, (\ref{pf.lem.clique-and-clball.1}) (applied to $a=e$ and
$b=f$) yields $d\left(  e,f\right)  =\beta$. Thus, the \textquotedblleft
symmetry\textquotedblright\ axiom in Definition \ref{def.Vultra} yields that
$d\left(  f,e\right)  =d\left(  e,f\right)  =\beta\leq\beta$. Hence, $f=e$ or
else $d\left(  f,e\right)  \leq\beta$. In other words, $f\in B_{\beta}\left(
e\right)  $ (by the definition of $B_{\beta}\left(  e\right)  $). Hence,
Corollary \ref{cor.Bc.self-rad2} (applied to $\alpha=\beta$) yields $B_{\beta
}\left(  e\right)  =B_{\beta}\left(  f\right)  $.

Hence, we have proved $B_{\beta}\left(  e\right)  =B_{\beta}\left(  f\right)
$; thus, our proof of Lemma \ref{lem.clique-and-clball} \textbf{(a)} is
complete. \medskip

In preparation for the proofs of parts (\textbf{b)}, \textbf{(c)} and
\textbf{(d)}, let us observe the following:

We have defined $B$ to be $B_{\beta}\left(  c\right)  $ for some $c\in C$.
Consider this $c$. Thus, $B=B_{\beta}\left(  c\right)  $.

Lemma \ref{lem.clique-and-clball} \textbf{(a)} says that
\begin{equation}
B_{\beta}\left(  a\right)  =B_{\beta}\left(  b\right)
\ \ \ \ \ \ \ \ \ \ \text{for each }a,b\in C.
\label{pf.lem.clique-and-clball.2}%
\end{equation}
Thus, for each $a\in C$, we have
\begin{align}
B_{\beta}\left(  a\right)   &  =B_{\beta}\left(  c\right)
\ \ \ \ \ \ \ \ \ \ \left(  \text{by (\ref{pf.lem.clique-and-clball.2}),
applied to }b=c\right) \nonumber\\
&  =B\ \ \ \ \ \ \ \ \ \ \left(  \text{since }B=B_{\beta}\left(  c\right)
\right)  . \label{pf.lem.clique-and-clball.3}%
\end{align}

\textbf{(b)} Let $a\in C$. Then, $a\in B_{\beta}\left(  a\right)  $ (by the
definition of $B_{\beta}\left(  a\right)  $, since $a=a$). Hence, $a\in
B_{\beta}\left(  a\right)  =B$ (by (\ref{pf.lem.clique-and-clball.3})).

Forget that we fixed $a$. We thus have proved that $a\in B$ for each $a\in C$.
In other words, $C\subseteq B$. This proves Lemma \ref{lem.clique-and-clball}
\textbf{(b)}. \medskip

\textbf{(c)} Let $p,q\in B$ be distinct. We must prove that $d\left(
p,q\right)  \leq\beta$.

We have $p\in B=B_{\beta}\left(  c\right)  $. Hence, Corollary
\ref{cor.Bc.self-rad2} (applied to $\alpha=\beta$, $e=c$ and $f=p$) yields
$B_{\beta}\left(  c\right)  =B_{\beta}\left(  p\right)  $. Now, $q\in
B=B_{\beta}\left(  c\right)  =B_{\beta}\left(  p\right)  $. In other words, we
have $q=p$ or else $d\left(  q,p\right)  \leq\beta$ (by the definition of
$B_{\beta}\left(  p\right)  $). Since $q=p$ is impossible (because $p$ and $q$
are distinct), we thus obtain $d\left(  q,p\right)  \leq\beta$. Thus, the
\textquotedblleft symmetry\textquotedblright\ axiom in Definition
\ref{def.Vultra} yields that $d\left(  p,q\right)  =d\left(  q,p\right)
\leq\beta$. This proves Lemma \ref{lem.clique-and-clball} \textbf{(c)}.
\medskip

\textbf{(d)} Let $n\in E\setminus B$ and $p,q\in B$. We must prove that
$d\left(  n,p\right)  =d\left(  n,q\right)  $. If $p=q$, then this is obvious.
Thus, we WLOG assume that $p\neq q$. Hence, Lemma \ref{lem.clique-and-clball}
\textbf{(c)} yields $d\left(  p,q\right)  \leq\beta$.

We have $n\in E\setminus B$. In other words, $n\in E$ and $n\notin B$.

But we have $q\in B=B_{\beta}\left(  c\right)  $. Hence, Corollary
\ref{cor.Bc.self-rad2} (applied to $\alpha=\beta$, $e=c$ and $f=q$) yields
$B_{\beta}\left(  c\right)  =B_{\beta}\left(  q\right)  $. Now, $n\notin
B=B_{\beta}\left(  c\right)  =B_{\beta}\left(  q\right)  $. In other words, we
don't have $\left(  n=q\text{ or else }d\left(  n,q\right)  \leq\beta\right)
$ (by the definition of $B_{\beta}\left(  q\right)  $). In other words, we
have $n\neq q$ and $d\left(  n,q\right)  >\beta$. Thus, $d\left(  n,q\right)
>\beta\geq d\left(  p,q\right)  $ (since $d\left(  p,q\right)  \leq\beta$).

We have $p\neq n$ (since $p\in B$ but $n\notin B$) and similarly $q\neq n$.
Hence, the elements $n$, $p$ and $q$ of $E$ are distinct (since $p\neq n$ and
$q\neq n$ and $p\neq q$). The ultrametric triangle inequality (applied to $n$,
$p$ and $q$ instead of $a$, $b$ and $c$) thus yields%
\[
d\left(  n,p\right)  \leq\max\left\{  d\left(  n,q\right)  ,d\left(
p,q\right)  \right\}  =d\left(  n,q\right)  \ \ \ \ \ \ \ \ \ \ \left(
\text{since }d\left(  n,q\right)  >d\left(  p,q\right)  \right)  .
\]
The same argument (with the roles of $p$ and $q$ interchanged) yields
$d\left(  n,q\right)  \leq d\left(  n,p\right)  $. Combining these two
inequalities, we obtain $d\left(  n,p\right)  =d\left(  n,q\right)  $. This
proves Lemma \ref{lem.clique-and-clball} \textbf{(d)}.
\end{proof}

\subsection{Exchange results for sets intersecting a ball}

From now on, for the rest of Section \ref{sect.converse}, we assume that $E$
is finite.

\begin{corollary}
\label{cor.clique-and-clball2}Let $\beta\in\mathbb{V}$. Let $C$ be a $\beta
$-clique. Let $B=B_{\beta}\left(  c\right)  $ for some $c\in C$.

Let $N$ be a subset of $E\setminus B$.

Let $P$ and $Q$ be two subsets of $B$ such that $\left\vert P\right\vert
=\left\vert Q\right\vert $. Then:

\begin{enumerate}
\item[\textbf{(a)}] We have $\operatorname*{PER}\left(  N\cup Q\right)
-\operatorname*{PER}\left(  N\cup P\right)  =\operatorname*{PER}\left(
Q\right)  -\operatorname*{PER}\left(  P\right)  $.

\item[\textbf{(b)}] Assume that the map $w:E\rightarrow\mathbb{V}$ is
constant. Assume further that $Q$ is a subset of $C$. Then,
$\operatorname*{PER}\left(  N\cup Q\right)  \geq\operatorname*{PER}\left(
N\cup P\right)  $.
\end{enumerate}
\end{corollary}

\begin{proof}
[Proof of Corollary \ref{cor.clique-and-clball2}.]Let $m=\left\vert
P\right\vert =\left\vert Q\right\vert $. Let $p_{1},p_{2},\ldots,p_{m}$ be all
the $m$ elements of $P$ (listed without repetition). Let $q_{1},q_{2}%
,\ldots,q_{m}$ be all the $m$ elements of $Q$ (listed without repetition).
\medskip

\textbf{(a)} We have%
\begin{equation}
d\left(  n,p_{i}\right)  =d\left(  n,q_{i}\right)
\ \ \ \ \ \ \ \ \ \ \text{for each }n\in N\text{ and }i\in\left\{
1,2,\ldots,m\right\}  . \label{pf.cor.clique-and-clball2.1}%
\end{equation}

[\textit{Proof of (\ref{pf.cor.clique-and-clball2.1}):} Let $n\in N$ and
$i\in\left\{  1,2,\ldots,m\right\}  $. Then, $n\in N\subseteq E\setminus B$
and $p_{i}\in P\subseteq B$ and $q_{i}\in Q\subseteq B$. Thus, Lemma
\ref{lem.clique-and-clball} \textbf{(d)} (applied to $p=p_{i}$ and $q=q_{i}$)
yields $d\left(  n,p_{i}\right)  =d\left(  n,q_{i}\right)  $. This proves
(\ref{pf.cor.clique-and-clball2.1}).]

The set $N$ is disjoint from $B$ (since $N$ is a subset of $E\setminus B$),
and thus disjoint from $P$ as well (since $P\subseteq B$). Hence, the
definition of perimeter yields%
\begin{align*}
\operatorname*{PER}\left(  N\cup P\right)   &  =\operatorname*{PER}\left(
N\right)  +\operatorname*{PER}\left(  P\right)  +\sum_{n\in N}\underbrace{\sum
_{p\in P}d\left(  n,p\right)  }_{\substack{=\sum_{i=1}^{m}d\left(
n,p_{i}\right)  \\\text{(since }p_{1},p_{2},\ldots,p_{m}\text{ are
all}\\\text{the }m\text{ elements of }P\\\text{(listed without repetition))}%
}}\\
&  =\operatorname*{PER}\left(  N\right)  +\operatorname*{PER}\left(  P\right)
+\sum_{n\in N}\sum_{i=1}^{m}d\left(  n,p_{i}\right)  .
\end{align*}
Likewise,%
\[
\operatorname*{PER}\left(  N\cup Q\right)  =\operatorname*{PER}\left(
N\right)  +\operatorname*{PER}\left(  Q\right)  +\sum_{n\in N}\sum_{i=1}%
^{m}d\left(  n,q_{i}\right)  .
\]
Subtracting the first of these two equalities from the second, we obtain%
\begin{align*}
&  \operatorname*{PER}\left(  N\cup Q\right)  -\operatorname*{PER}\left(
N\cup P\right) \\
&  =\left(  \operatorname*{PER}\left(  N\right)  +\operatorname*{PER}\left(
Q\right)  +\sum_{n\in N}\sum_{i=1}^{m}d\left(  n,q_{i}\right)  \right) \\
&  \ \ \ \ \ \ \ \ \ \ -\left(  \operatorname*{PER}\left(  N\right)
+\operatorname*{PER}\left(  P\right)  +\sum_{n\in N}\sum_{i=1}^{m}d\left(
n,p_{i}\right)  \right) \\
&  =\operatorname*{PER}\left(  Q\right)  -\operatorname*{PER}\left(  P\right)
+\sum_{n\in N}\sum_{i=1}^{m}d\left(  n,q_{i}\right)  -\sum_{n\in N}\sum
_{i=1}^{m}\underbrace{d\left(  n,p_{i}\right)  }_{\substack{=d\left(
n,q_{i}\right)  \\\text{(by (\ref{pf.cor.clique-and-clball2.1}))}}}\\
&  =\operatorname*{PER}\left(  Q\right)  -\operatorname*{PER}\left(  P\right)
+\sum_{n\in N}\sum_{i=1}^{m}d\left(  n,q_{i}\right)  -\sum_{n\in N}\sum
_{i=1}^{m}d\left(  n,q_{i}\right) \\
&  =\operatorname*{PER}\left(  Q\right)  -\operatorname*{PER}\left(  P\right)
.
\end{align*}
This proves Corollary \ref{cor.clique-and-clball2} \textbf{(a)}. \medskip

\textbf{(b)} We know that the map $w:E\rightarrow\mathbb{V}$ is constant.
Hence,
\begin{equation}
w\left(  a\right)  =w\left(  b\right)  \ \ \ \ \ \ \ \ \ \ \text{for any
}a,b\in E. \label{pf.cor.clique-and-clball2.wa=wb}%
\end{equation}

Each $i\in\left\{  1,2,\ldots,m\right\}  $ satisfies%
\begin{equation}
w\left(  p_{i}\right)  =w\left(  c\right)
\label{pf.cor.clique-and-clball2.wp=wc}%
\end{equation}
(by (\ref{pf.cor.clique-and-clball2.wa=wb}), applied to $a=p_{i}$ and $b=c$)
and%
\begin{equation}
w\left(  q_{i}\right)  =w\left(  c\right)
\label{pf.cor.clique-and-clball2.wq=wc}%
\end{equation}
(by (\ref{pf.cor.clique-and-clball2.wa=wb}), applied to $a=q_{i}$ and $b=c$).

We know that $C$ is a $\beta$-clique. In other words, $C$ is a subset of $E$
such that%
\begin{equation}
\text{any two distinct elements }a,b\in C\text{ satisfy }d\left(  a,b\right)
=\beta\label{pf.cor.clique-and-clball2.clique}%
\end{equation}
(by the definition of a \textquotedblleft$\beta$-clique\textquotedblright).

But $p_{1},p_{2},\ldots,p_{m}$ are $m$ distinct elements of $P$ (by their
definition). Hence, $p_{1},p_{2},\ldots,p_{m}$ are $m$ distinct elements of
$B$ (since $P\subseteq B$). Thus, if $i$ and $j$ are two distinct elements of
$\left\{  1,2,\ldots,m\right\}  $, then $p_{i}$ and $p_{j}$ are two distinct
elements of $B$, and therefore satisfy%
\begin{equation}
d\left(  p_{i},p_{j}\right)  \leq\beta\label{pf.cor.clique-and-clball2.dp}%
\end{equation}
(by Lemma \ref{lem.clique-and-clball} \textbf{(c)}, applied to $p=p_{i}$ and
$q=p_{j}$).

On the other hand, $q_{1},q_{2},\ldots,q_{m}$ are $m$ distinct elements of $Q$
(by their definition). Hence, $q_{1},q_{2},\ldots,q_{m}$ are $m$ distinct
elements of $C$ (since $Q\subseteq C$). Thus, if $i$ and $j$ are two distinct
elements of $\left\{  1,2,\ldots,m\right\}  $, then $q_{i}$ and $q_{j}$ are
two distinct elements of $C$, and therefore satisfy%
\begin{equation}
d\left(  q_{i},q_{j}\right)  =\beta\label{pf.cor.clique-and-clball2.dq}%
\end{equation}
(by (\ref{pf.cor.clique-and-clball2.clique}), applied to $a=q_{i}$ and
$b=q_{j}$).

Recall that $p_{1},p_{2},\ldots,p_{m}$ are all the $m$ elements of $P$ (listed
without repetition). Hence, the definition of perimeter yields%
\begin{align*}
\operatorname*{PER}\left(  P\right)   &  =\sum_{i=1}^{m}\underbrace{w\left(
p_{i}\right)  }_{\substack{=w\left(  c\right)  \\\text{(by
(\ref{pf.cor.clique-and-clball2.wp=wc}))}}}+\sum_{1\leq i<j\leq m}%
\underbrace{d\left(  p_{i},p_{j}\right)  }_{\substack{\leq\beta\\\text{(by
(\ref{pf.cor.clique-and-clball2.dp}))}}}\leq\sum_{i=1}^{m}\underbrace{w\left(
c\right)  }_{\substack{=w\left(  q_{i}\right)  \\\text{(by
(\ref{pf.cor.clique-and-clball2.wq=wc}))}}}+\sum_{1\leq i<j\leq m}%
\underbrace{\beta}_{\substack{=d\left(  q_{i},q_{j}\right)  \\\text{(by
(\ref{pf.cor.clique-and-clball2.dq}))}}}\\
&  =\sum_{i=1}^{m}w\left(  q_{i}\right)  +\sum_{1\leq i<j\leq m}d\left(
q_{i},q_{j}\right)  =\operatorname*{PER}\left(  Q\right)
\end{align*}
(by the definition of $Q$, since $q_{1},q_{2},\ldots,q_{m}$ are all the $m$
elements of $Q$ (listed without repetition)).

But Corollary \ref{cor.clique-and-clball2} \textbf{(a)} yields
\[
\operatorname*{PER}\left(  N\cup Q\right)  -\operatorname*{PER}\left(  N\cup
P\right)  =\operatorname*{PER}\left(  Q\right)  -\operatorname*{PER}\left(
P\right)  \geq0
\]
(since $\operatorname*{PER}\left(  P\right)  \leq\operatorname*{PER}\left(
Q\right)  $). In other words, $\operatorname*{PER}\left(  N\cup Q\right)
\geq\operatorname*{PER}\left(  N\cup P\right)  $. This proves Corollary
\ref{cor.clique-and-clball2} \textbf{(b)}.
\end{proof}

\begin{corollary}
\label{cor.clique-and-clball3}Let $\mathcal{F}$ be the Bhargava greedoid of
$\left(  E,w,d\right)  $. Assume that the map $w:E\rightarrow\mathbb{V}$ is constant.

Let $\beta\in\mathbb{V}$. Let $C$ be a $\beta$-clique. Let $B=B_{\beta}\left(
c\right)  $ for some $c\in C$.

Let $N$ be a subset of $E\setminus B$.

Let $P$ and $Q$ be two subsets of $B$ such that $\left\vert P\right\vert
=\left\vert Q\right\vert $ and $Q\subseteq C$ and $N\cup P\in\mathcal{F}$.
Then, $N\cup Q\in\mathcal{F}$.
\end{corollary}

\begin{proof}
[Proof of Corollary \ref{cor.clique-and-clball3}.]Corollary
\ref{cor.clique-and-clball2} \textbf{(b)} yields $\operatorname*{PER}\left(
N\cup Q\right)  \geq\operatorname*{PER}\left(  N\cup P\right)  $.

Also, the set $N$ is disjoint from $B$ (since $N\subseteq E\setminus B$), and
thus is disjoint from $P$ as well (since $P$ is a subset of $B$). Hence,
$\left\vert N\cup P\right\vert =\left\vert N\right\vert +\left\vert
P\right\vert $. The same argument (applied to $Q$ instead of $P$) shows that
$\left\vert N\cup Q\right\vert =\left\vert N\right\vert +\left\vert
Q\right\vert $. Hence, $\left\vert N\cup P\right\vert =\left\vert N\right\vert
+\underbrace{\left\vert P\right\vert }_{=\left\vert Q\right\vert }=\left\vert
N\right\vert +\left\vert Q\right\vert =\left\vert N\cup Q\right\vert $.

We know that $\mathcal{F}$ is the Bhargava greedoid of $\left(  E,w,d\right)
$. In other words,
\begin{equation}
\mathcal{F}=\left\{  A\subseteq E\ \mid\ A\text{ has maximum perimeter among
all }\left\vert A\right\vert \text{-subsets of }E\right\}
\label{pf.cor.clique-and-clball3.F=}%
\end{equation}
(by Definition \ref{def.bhar-greed}). Hence, from $N\cup P\in\mathcal{F}$, we
conclude that the set $N\cup P$ has maximum perimeter among all $\left\vert
N\cup P\right\vert $-subsets of $E$. In other words, the set $N\cup P$ has
maximum perimeter among all $\left\vert N\cup Q\right\vert $-subsets of $E$
(since $\left\vert N\cup P\right\vert =\left\vert N\cup Q\right\vert $). Since
$N\cup Q$ is a further $\left\vert N\cup Q\right\vert $-subset of $E$, we thus
conclude that $\operatorname*{PER}\left(  N\cup P\right)  \geq
\operatorname*{PER}\left(  N\cup Q\right)  $. Combining this with
$\operatorname*{PER}\left(  N\cup Q\right)  \geq\operatorname*{PER}\left(
N\cup P\right)  $, we obtain $\operatorname*{PER}\left(  N\cup Q\right)
=\operatorname*{PER}\left(  N\cup P\right)  $. In other words, the subsets
$N\cup Q$ and $N\cup P$ of $E$ have the same perimeter. Therefore, the set
$N\cup Q$ has maximum perimeter among all $\left\vert N\cup Q\right\vert
$-subsets of $E$ (because the set $N\cup P$ has maximum perimeter among all
$\left\vert N\cup Q\right\vert $-subsets of $E$). In view of
(\ref{pf.cor.clique-and-clball3.F=}), this entails that $N\cup Q\in
\mathcal{F}$. This proves Corollary \ref{cor.clique-and-clball3}.
\end{proof}

\subsection{Gaussian elimination greedoids in terms of determinants}

Next, we introduce some notations for matrices.

\begin{definition}
\label{def.submatrix}Let $n\in\mathbb{N}$ and $m\in\mathbb{N}$. Let $A=\left(
a_{i,j}\right)  _{1\leq i\leq n,\ 1\leq j\leq m}$ be an $n\times m$-matrix.
Let $i_{1},i_{2},\ldots,i_{u}$ be some elements of $\left\{  1,2,\ldots
,n\right\}  $; let $j_{1},j_{2},\ldots,j_{v}$ be some elements of $\left\{
1,2,\ldots,m\right\}  $. Then, we define $\operatorname*{sub}\nolimits_{i_{1}%
,i_{2},\ldots,i_{u}}^{j_{1},j_{2},\ldots,j_{v}}A$ to be the $u\times v$-matrix
$\left(  a_{i_{x},j_{y}}\right)  _{1\leq x\leq u,\ 1\leq y\leq v}$.

When $i_{1}<i_{2}<\cdots<i_{u}$ and $j_{1}<j_{2}<\cdots<j_{v}$, the matrix
$\operatorname*{sub}\nolimits_{i_{1},i_{2},\ldots,i_{u}}^{j_{1},j_{2}%
,\ldots,j_{v}}A$ can be obtained from $A$ by crossing out all rows other than
the $i_{1}$-th, the $i_{2}$-th, etc., the $i_{u}$-th row and crossing out all
columns other than the $j_{1}$-th, the $j_{2}$-th, etc., the $j_{v}$-th
column. Thus, in this case, $\operatorname*{sub}\nolimits_{i_{1},i_{2}%
,\ldots,i_{u}}^{j_{1},j_{2},\ldots,j_{v}}A$ is called a \textit{submatrix} of
$A$.
\end{definition}

\begin{example}
If $n=3$ and $m=4$ and $A=\left(
\begin{array}
[c]{cccc}%
a & b & c & d\\
e & f & g & h\\
i & j & k & \ell
\end{array}
\right)  $, then $\operatorname*{sub}\nolimits_{1,3}^{1,3,4}A=\left(
\begin{array}
[c]{ccc}%
a & c & d\\
i & k & \ell
\end{array}
\right)  $ (this is a submatrix of $A$) and $\operatorname*{sub}%
\nolimits_{2,3}^{3,2,1}A=\left(
\begin{array}
[c]{ccc}%
g & f & e\\
k & j & i
\end{array}
\right)  $ (this is not, in general, a submatrix of $A$).
\end{example}

We can now describe Gaussian elimination greedoids in terms of determinants:

\begin{lemma}
\label{lem.geg.through-det}Let $n\in\mathbb{N}$. Let $E$ be the set $\left\{
1,2,\ldots,n\right\}  $.

Let $m\in\mathbb{N}$ be such that $m\geq\left\vert E\right\vert $. Let
$\mathbb{K}$ be a field. For each $e\in E$, let $v_{e}\in\mathbb{K}^{m}$ be a
column vector. Let $A$ be the $m\times n$-matrix whose columns (from left to
right) are $v_{1},v_{2},\ldots,v_{n}$.

Let $\mathcal{G}$ be the Gaussian elimination greedoid of the vector family
$\left(  v_{e}\right)  _{e\in E}$.

Let $p\in\mathbb{N}$. Let $i_{1},i_{2},\ldots,i_{p}\in E$ be $p$ distinct
numbers. Let $I=\left\{  i_{1},i_{2},\ldots,i_{p}\right\}  $. Then,%
\[
I\in\mathcal{G}\text{ holds if and only if }\det\left(  \operatorname*{sub}%
\nolimits_{1,2,\ldots,p}^{i_{1},i_{2},\ldots,i_{p}}A\right)  \neq0.
\]

\end{lemma}

\begin{proof}
[Proof of Lemma \ref{lem.geg.through-det}.]We have $I=\left\{  i_{1}%
,i_{2},\ldots,i_{p}\right\}  $. Thus, $\left\vert I\right\vert =p$ (since
$i_{1},i_{2},\ldots,i_{p}$ are distinct) and $I\subseteq E$ (since
$i_{1},i_{2},\ldots,i_{p}\in E$).

Define the maps $\pi_{k}$ for all $k\in\left\{  0,1,\ldots,m\right\}  $ as in
Definition \ref{def.geg}. Since $I\subseteq E$, we have $\left\vert
I\right\vert \leq\left\vert E\right\vert $. Hence, $p=\left\vert I\right\vert
\leq\left\vert E\right\vert \leq m$ (since $m\geq\left\vert E\right\vert $),
so that $p\in\left\{  0,1,\ldots,m\right\}  $. Thus, the map $\pi
_{p}:\mathbb{K}^{m}\rightarrow\mathbb{K}^{p}$ is well-defined.

Write the $m\times n$-matrix $A$ in the form $A=\left(  a_{i,j}\right)
_{1\leq i\leq m,\ 1\leq j\leq n}$. Then, the definition of
$\operatorname*{sub}\nolimits_{1,2,\ldots,p}^{i_{1},i_{2},\ldots,i_{p}}A$
yields%
\[
\operatorname*{sub}\nolimits_{1,2,\ldots,p}^{i_{1},i_{2},\ldots,i_{p}%
}A=\left(  a_{x,i_{y}}\right)  _{1\leq x\leq p,\ 1\leq y\leq p}.
\]
Thus, for each $k\in\left\{  1,2,\ldots,p\right\}  $, we have%
\begin{equation}
\left(  \text{the }k\text{-th column of }\operatorname*{sub}%
\nolimits_{1,2,\ldots,p}^{i_{1},i_{2},\ldots,i_{p}}A\right)  =\left(
\begin{array}
[c]{c}%
a_{1,i_{k}}\\
a_{2,i_{k}}\\
\vdots\\
a_{p,i_{k}}%
\end{array}
\right)  . \label{pf.lem.geg.through-det.col1}%
\end{equation}

The columns of the matrix $A$ (from left to right) are $v_{1},v_{2}%
,\ldots,v_{n}$. In other words, the $\ell$-th column of $A$ is $v_{\ell}$ for
each $\ell\in\left\{  1,2,\ldots,n\right\}  $. In other words, for each
$\ell\in\left\{  1,2,\ldots,n\right\}  $, we have $\left(  \text{the }%
\ell\text{-th column of }A\right)  =v_{\ell}$. Thus, for each $\ell\in\left\{
1,2,\ldots,n\right\}  $, we have%
\[
v_{\ell}=\left(  \text{the }\ell\text{-th column of }A\right)  =\left(
\begin{array}
[c]{c}%
a_{1,\ell}\\
a_{2,\ell}\\
\vdots\\
a_{m,\ell}%
\end{array}
\right)
\]
(since $A=\left(  a_{i,j}\right)  _{1\leq i\leq m,\ 1\leq j\leq n}$) and
therefore%
\begin{equation}
\pi_{p}\left(  v_{\ell}\right)  =\pi_{p}\left(  \left(
\begin{array}
[c]{c}%
a_{1,\ell}\\
a_{2,\ell}\\
\vdots\\
a_{m,\ell}%
\end{array}
\right)  \right)  =\left(
\begin{array}
[c]{c}%
a_{1,\ell}\\
a_{2,\ell}\\
\vdots\\
a_{p,\ell}%
\end{array}
\right)  \label{pf.lem.geg.through-det.col3}%
\end{equation}
(by the definition of $\pi_{p}$). Hence, for each $k\in\left\{  1,2,\ldots
,p\right\}  $, we have%
\begin{align*}
\pi_{p}\left(  v_{i_{k}}\right)   &  =\left(
\begin{array}
[c]{c}%
a_{1,i_{k}}\\
a_{2,i_{k}}\\
\vdots\\
a_{p,i_{k}}%
\end{array}
\right)  \ \ \ \ \ \ \ \ \ \ \left(  \text{by
(\ref{pf.lem.geg.through-det.col3}), applied to }\ell=i_{k}\right) \\
&  =\left(  \text{the }k\text{-th column of }\operatorname*{sub}%
\nolimits_{1,2,\ldots,p}^{i_{1},i_{2},\ldots,i_{p}}A\right)
\ \ \ \ \ \ \ \ \ \ \left(  \text{by (\ref{pf.lem.geg.through-det.col1}%
)}\right)  .
\end{align*}
In other words, the vectors $\pi_{p}\left(  v_{i_{1}}\right)  ,\pi_{p}\left(
v_{i_{2}}\right)  ,\ldots,\pi_{p}\left(  v_{i_{p}}\right)  $ are the columns
of the matrix $\operatorname*{sub}\nolimits_{1,2,\ldots,p}^{i_{1},i_{2}%
,\ldots,i_{p}}A$.

We know that $\mathcal{G}$ is the Gaussian elimination greedoid of the vector
family $\left(  v_{e}\right)  _{e\in E}$. Thus, Definition \ref{def.geg} shows
that%
\[
\mathcal{G}=\left\{  F\subseteq E\ \mid\ \text{the family }\left(
\pi_{\left\vert F\right\vert }\left(  v_{e}\right)  \right)  _{e\in F}%
\in\left(  \mathbb{K}^{\left\vert F\right\vert }\right)  ^{F}\text{ is
linearly independent}\right\}  .
\]
Hence, we have the following chain of logical equivalences:%
\begin{align*}
&  \ \left(  I\in\mathcal{G}\right) \\
&  \Longleftrightarrow\ \left(  \text{the family }\left(  \pi_{\left\vert
I\right\vert }\left(  v_{e}\right)  \right)  _{e\in I}\in\left(
\mathbb{K}^{\left\vert I\right\vert }\right)  ^{I}\text{ is linearly
independent}\right) \\
&  \Longleftrightarrow\ \left(  \text{the family }\left(  \pi_{p}\left(
v_{e}\right)  \right)  _{e\in I}\in\left(  \mathbb{K}^{p}\right)  ^{I}\text{
is linearly independent}\right) \\
&  \ \ \ \ \ \ \ \ \ \ \ \ \ \ \ \ \ \ \ \ \left(  \text{since }\left\vert
I\right\vert =p\right) \\
&  \Longleftrightarrow\ \left(  \text{the vectors }\pi_{p}\left(
v_{e}\right)  \text{ for }e\in I\text{ are linearly independent}\right) \\
&  \Longleftrightarrow\ \left(  \text{the vectors }\pi_{p}\left(  v_{i_{1}%
}\right)  ,\pi_{p}\left(  v_{i_{2}}\right)  ,\ldots,\pi_{p}\left(  v_{i_{p}%
}\right)  \text{ are linearly independent}\right) \\
&  \ \ \ \ \ \ \ \ \ \ \ \ \ \ \ \ \ \ \ \ \left(
\begin{array}
[c]{c}%
\text{since the vectors }\pi_{p}\left(  v_{e}\right)  \text{ for }e\in I\text{
are precisely}\\
\text{the vectors }\pi_{p}\left(  v_{i_{1}}\right)  ,\pi_{p}\left(  v_{i_{2}%
}\right)  ,\ldots,\pi_{p}\left(  v_{i_{p}}\right)
\end{array}
\right) \\
&  \Longleftrightarrow\ \left(  \text{the columns of the matrix }%
\operatorname*{sub}\nolimits_{1,2,\ldots,p}^{i_{1},i_{2},\ldots,i_{p}}A\text{
are linearly independent}\right) \\
&  \ \ \ \ \ \ \ \ \ \ \ \ \ \ \ \ \ \ \ \ \left(
\begin{array}
[c]{c}%
\text{since the vectors }\pi_{p}\left(  v_{i_{1}}\right)  ,\pi_{p}\left(
v_{i_{2}}\right)  ,\ldots,\pi_{p}\left(  v_{i_{p}}\right)  \text{ are}\\
\text{the columns of the matrix }\operatorname*{sub}\nolimits_{1,2,\ldots
,p}^{i_{1},i_{2},\ldots,i_{p}}A
\end{array}
\right) \\
&  \Longleftrightarrow\ \left(  \text{the matrix }\operatorname*{sub}%
\nolimits_{1,2,\ldots,p}^{i_{1},i_{2},\ldots,i_{p}}A\text{ is invertible}%
\right) \\
&  \ \ \ \ \ \ \ \ \ \ \ \ \ \ \ \ \ \ \ \ \left(
\begin{array}
[c]{c}%
\text{since }\operatorname*{sub}\nolimits_{1,2,\ldots,p}^{i_{1},i_{2}%
,\ldots,i_{p}}A\text{ is a square matrix, and thus is invertible}\\
\text{if and only if its columns are linearly independent}%
\end{array}
\right) \\
&  \Longleftrightarrow\ \left(  \det\left(  \operatorname*{sub}%
\nolimits_{1,2,\ldots,p}^{i_{1},i_{2},\ldots,i_{p}}A\right)  \neq0\right)  .
\end{align*}
This proves Lemma \ref{lem.geg.through-det}.
\end{proof}

We can leverage Lemma \ref{lem.geg.through-det} to obtain a criterion that,
roughly speaking, says that if a Gaussian elimination greedoid over a field
$\mathbb{K}$ contains a certain \textquotedblleft
constellation\textquotedblright\ (in an appropriate sense), then $\left\vert
\mathbb{K}\right\vert $ must be $\geq$ to a certain value. Namely:

\begin{lemma}
\label{lem.geg.K-bound}Let $\mathbb{K}$ be a field. Let $E$ be a finite set.
Let $\mathcal{F}$ be the Gaussian elimination greedoid of a vector family
$\left(  v_{e}\right)  _{e\in E}$ over $\mathbb{K}$. Let $N$ and $C$ be two
disjoint subsets of $E$. Assume that the following three conditions hold:

\begin{enumerate}
\item[\textbf{(i)}] For any $i\in C$, we have $N\cup\left\{  i\right\}
\in\mathcal{F}$.

\item[\textbf{(ii)}] For any distinct $i,j\in C$, we have $N\cup\left\{
i,j\right\}  \in\mathcal{F}$.

\item[\textbf{(iii)}] For any $p\in N$ and any distinct $i,j\in C$, we have
$\left(  N\cup\left\{  i,j\right\}  \right)  \setminus\left\{  p\right\}
\notin\mathcal{F}$.
\end{enumerate}

Then, $\left\vert \mathbb{K}\right\vert \geq\left\vert C\right\vert $.
\end{lemma}

\begin{proof}
[Proof of Lemma \ref{lem.geg.K-bound}.]Let $m\in\mathbb{N}$ be such that the
vectors $v_{e}$ belong to $\mathbb{K}^{m}$. Then, $m\geq\left\vert
E\right\vert $ (since otherwise, the Gaussian elimination greedoid of the
vector family $\left(  v_{e}\right)  _{e\in E}$ would not be well-defined).

Let $n=\left\vert E\right\vert $ and $r=\left\vert N\right\vert $. Hence, $N$
is an $r$-element subset of the $n$-element set $E$. Thus, the set $E$
consists of the $r$ elements of $N$ and of the $n-r$ remaining elements of $E$.

Clearly, the claim we are proving will not change if we rename the elements of
$E$. Thus, we can rename the elements of $E$ arbitrarily. In particular, we
can rename them in such a way that the $r$ elements of the subset $N$ will
become $1,2,\ldots,r$ whereas all remaining $n-r$ elements of $E$ will become
$r+1,r+2,\ldots,n$. Thus, we can WLOG assume that the $r$ elements of $N$ are
$1,2,\ldots,r$ and the remaining $n-r$ elements of $E$ are $r+1,r+2,\ldots,n$.
Assume this. Hence, $N=\left\{  1,2,\ldots,r\right\}  $ and $E\setminus
N=\left\{  r+1,r+2,\ldots,n\right\}  $ and therefore $E=\left\{
1,2,\ldots,n\right\}  $.

Since the subsets $N$ and $C$ of $E$ are disjoint, we have $C\subseteq
E\setminus N=\left\{  r+1,r+2,\ldots,n\right\}  $.

We must prove that $\left\vert \mathbb{K}\right\vert \geq\left\vert
C\right\vert $. If $\left\vert C\right\vert \leq1$, then this is obvious
(since $\left\vert \mathbb{K}\right\vert \geq1$). Hence, we WLOG assume that
$\left\vert C\right\vert >1$ from now on. However, from $C\subseteq E\setminus
N$, we obtain
\begin{align*}
\left\vert C\right\vert  &  \leq\left\vert E\setminus N\right\vert
=\underbrace{\left\vert E\right\vert }_{=n}-\underbrace{\left\vert
N\right\vert }_{=r}\ \ \ \ \ \ \ \ \ \ \left(  \text{since }N\subseteq
E\right) \\
&  =n-r.
\end{align*}
Thus, $n-r\geq\left\vert C\right\vert >1$, so that $n-r\geq2$ and thus $n\geq
r+2$. Thus, $r+2\leq n=\left\vert E\right\vert \leq m$ (since $m\geq\left\vert
E\right\vert $).

Let $A$ be the $m\times n$-matrix whose columns (from left to right) are
$v_{1},v_{2},\ldots,v_{n}$. Write this matrix $A$ in the form $A=\left(
a_{i,j}\right)  _{1\leq i\leq m,\ 1\leq j\leq n}$.

We now make a few observations:

\begin{statement}
\textit{Claim 1:} We have $\det\left(  \operatorname*{sub}%
\nolimits_{1,2,\ldots,r+1}^{1,2,\ldots,r,i}A\right)  \neq0$ for each $i\in C$.
\end{statement}

[\textit{Proof of Claim 1:} Let $i\in C$. Then, $i\in C\subseteq\left\{
r+1,r+2,\ldots,n\right\}  $. Thus, the $r+1$ numbers $1,2,\ldots,r,i$ are
distinct. Also, $N\cup\left\{  i\right\}  \in\mathcal{F}$ (by condition
\textbf{(i)} in Lemma \ref{lem.geg.K-bound}). But $N\cup\left\{  i\right\}
=\left\{  1,2,\ldots,r,i\right\}  $ (since $N=\left\{  1,2,\ldots,r\right\}
$). Thus, Lemma \ref{lem.geg.through-det} (applied to $\mathcal{F}$,
$N\cup\left\{  i\right\}  $, $r+1$ and $\left(  1,2,\ldots,r,i\right)  $
instead of $\mathcal{G}$, $I$, $p$ and $\left(  i_{1},i_{2},\ldots
,i_{p}\right)  $) yields that $N\cup\left\{  i\right\}  \in\mathcal{F}$ holds
if and only if $\det\left(  \operatorname*{sub}\nolimits_{1,2,\ldots
,r+1}^{1,2,\ldots,r,i}A\right)  \neq0.$ Thus, we have $\det\left(
\operatorname*{sub}\nolimits_{1,2,\ldots,r+1}^{1,2,\ldots,r,i}A\right)  \neq0$
(since $N\cup\left\{  i\right\}  \in\mathcal{F}$). This proves Claim 1.]
\medskip

Now, for each $i\in C$, we define a scalar $r_{i}\in\mathbb{K}$ by%
\begin{equation}
r_{i}=\dfrac{a_{r+2,i}}{\det\left(  \operatorname*{sub}\nolimits_{1,2,\ldots
,r+1}^{1,2,\ldots,r,i}A\right)  }. \label{pf.lem.geg.K-bound.1}%
\end{equation}
This is well-defined, since Claim 1 yields $\det\left(  \operatorname*{sub}%
\nolimits_{1,2,\ldots,r+1}^{1,2,\ldots,r,i}A\right)  \neq0$ (and since we have
$r+2\leq m$).

\begin{statement}
\textit{Claim 2:} The scalars $r_{i}$ for all $i\in C$ are distinct.
\end{statement}

[\textit{Proof of Claim 2:} We need to prove that $r_{i}\neq r_{j}$ for any
two distinct $i,j\in C$. So let us fix two distinct $i,j\in C$. We must prove
that $r_{i}\neq r_{j}$.

We have $i,j\in C\subseteq\left\{  r+1,r+2,\ldots,n\right\}  $. Thus, the
$r+2$ numbers $1,2,\ldots,r,i,j$ are distinct (since $i$ and $j$ are
distinct). Also, $N\cup\left\{  i,j\right\}  \in\mathcal{F}$ (by condition
\textbf{(ii)} in Lemma \ref{lem.geg.K-bound}). But $N\cup\left\{  i,j\right\}
=\left\{  1,2,\ldots,r,i,j\right\}  $ (since $N=\left\{  1,2,\ldots,r\right\}
$). Thus, Lemma \ref{lem.geg.through-det} (applied to $\mathcal{F}$,
$N\cup\left\{  i,j\right\}  $, $r+2$ and $\left(  1,2,\ldots,r,i,j\right)  $
instead of $\mathcal{G}$, $I$, $p$ and $\left(  i_{1},i_{2},\ldots
,i_{p}\right)  $) yields that $N\cup\left\{  i,j\right\}  \in\mathcal{F}$
holds if and only if $\det\left(  \operatorname*{sub}\nolimits_{1,2,\ldots
,r+2}^{1,2,\ldots,r,i,j}A\right)  \neq0.$ Thus, we have
\begin{equation}
\det\left(  \operatorname*{sub}\nolimits_{1,2,\ldots,r+2}^{1,2,\ldots
,r,i,j}A\right)  \neq0 \label{pf.lem.geg.K-bound.c2.pf.neq0}%
\end{equation}
(since $N\cup\left\{  i,j\right\}  \in\mathcal{F}$).

Let us agree to use the following notation: If $p\in\left\{  1,2,\ldots
,r\right\}  $ is arbitrary, then \textquotedblleft$1,2,\ldots,\widehat{p}%
,\ldots,r,i,j$\textquotedblright\ will denote the list $1,2,\ldots,r,i,j$ with
the entry $p$ omitted (i.e., the list $1,2,\ldots,p-1,p+1,p+2,\ldots,r,i,j$).

Now, let us use Laplace expansion to expand the determinant of the $\left(
r+2\right)  \times\left(  r+2\right)  $-matrix $\operatorname*{sub}%
\nolimits_{1,2,\ldots,r+2}^{1,2,\ldots,r,i,j}A$ along its last row. We thus
obtain%
\begin{align}
&  \det\left(  \operatorname*{sub}\nolimits_{1,2,\ldots,r+2}^{1,2,\ldots
,r,i,j}A\right) \nonumber\\
&  =\sum_{p=1}^{r}\left(  -1\right)  ^{\left(  r+2\right)  +p}a_{r+2,p}%
\det\left(  \operatorname*{sub}\nolimits_{1,2,\ldots,r+1}^{1,2,\ldots
,\widehat{p},\ldots,r,i,j}A\right) \nonumber\\
&  \ \ \ \ \ \ \ \ \ \ +\left(  -1\right)  ^{\left(  r+2\right)  +\left(
r+1\right)  }a_{r+2,i}\det\left(  \operatorname*{sub}\nolimits_{1,2,\ldots
,r+1}^{1,2,\ldots,r,j}A\right) \nonumber\\
&  \ \ \ \ \ \ \ \ \ \ +\left(  -1\right)  ^{\left(  r+2\right)  +\left(
r+2\right)  }a_{r+2,j}\det\left(  \operatorname*{sub}\nolimits_{1,2,\ldots
,r+1}^{1,2,\ldots,r,i}A\right)  . \label{pf.lem.geg.K-bound.c2.pf.1}%
\end{align}
(Indeed, the entries of the last row of $\operatorname*{sub}%
\nolimits_{1,2,\ldots,r+2}^{1,2,\ldots,r,i,j}A$ are
\[
\underbrace{a_{r+2,1},a_{r+2,2},\ldots,a_{r+2,r}}_{\text{these are the
}a_{r+2,p}\text{ for all }p\in\left\{  1,2,\ldots,r\right\}  },a_{r+2,i}%
,a_{r+2,j},
\]
and the cofactors corresponding to the first $r$ of these entries are%
\[
\left(  -1\right)  ^{\left(  r+2\right)  +p}\det\left(  \operatorname*{sub}%
\nolimits_{1,2,\ldots,r+1}^{1,2,\ldots,\widehat{p},\ldots,r,i,j}A\right)
\ \ \ \ \ \ \ \ \ \ \text{for all }p\in\left\{  1,2,\ldots,r\right\}  ,
\]
whereas the cofactors corresponding to the last two entries are
\[
\left(  -1\right)  ^{\left(  r+2\right)  +\left(  r+1\right)  }\det\left(
\operatorname*{sub}\nolimits_{1,2,\ldots,r+1}^{1,2,\ldots,r,j}A\right)
\ \ \ \ \ \ \ \ \ \ \text{and}\ \ \ \ \ \ \ \ \ \ \left(  -1\right)  ^{\left(
r+2\right)  +\left(  r+2\right)  }\det\left(  \operatorname*{sub}%
\nolimits_{1,2,\ldots,r+1}^{1,2,\ldots,r,i}A\right)  .
\]
)

Now, let us simplify the entries in the $\sum_{p=1}^{r}$ sum on the right hand
side of (\ref{pf.lem.geg.K-bound.c2.pf.1}).

Let $p\in\left\{  1,2,\ldots,r\right\}  $. Then, $p\in\left\{  1,2,\ldots
,r\right\}  =N$. Hence, $\left(  N\cup\left\{  i,j\right\}  \right)
\setminus\left\{  p\right\}  \notin\mathcal{F}$ (by condition \textbf{(iii)}
in Lemma \ref{lem.geg.K-bound}). In other words, we \textbf{don't} have
$\left(  N\cup\left\{  i,j\right\}  \right)  \setminus\left\{  p\right\}
\in\mathcal{F}$. But from $N\cup\left\{  i,j\right\}  =\left\{  1,2,\ldots
,r,i,j\right\}  $, we obtain $\left(  N\cup\left\{  i,j\right\}  \right)
\setminus\left\{  p\right\}  =\left\{  1,2,\ldots,r,i,j\right\}
\setminus\left\{  p\right\}  =\left\{  1,2,\ldots,\widehat{p},\ldots
,r,i,j\right\}  $ (since the $r+2$ numbers $1,2,\ldots,r,i,j$ are distinct).
Of course, the $r+1$ numbers $1,2,\ldots,\widehat{p},\ldots,r,i,j$ are
distinct (since the $r+2$ numbers $1,2,\ldots,r,i,j$ are distinct). Thus,
Lemma \ref{lem.geg.through-det} (applied to $\mathcal{F}$, $\left(
N\cup\left\{  i,j\right\}  \right)  \setminus\left\{  p\right\}  $, $r+1$ and
$\left(  1,2,\ldots,\widehat{p},\ldots,r,i,j\right)  $ instead of
$\mathcal{G}$, $I$, $p$ and $\left(  i_{1},i_{2},\ldots,i_{p}\right)  $)
yields that $\left(  N\cup\left\{  i,j\right\}  \right)  \setminus\left\{
p\right\}  \in\mathcal{F}$ holds if and only if $\det\left(
\operatorname*{sub}\nolimits_{1,2,\ldots,r+1}^{1,2,\ldots,\widehat{p}%
,\ldots,r,i,j}A\right)  \neq0.$ Thus, we don't have $\det\left(
\operatorname*{sub}\nolimits_{1,2,\ldots,r+1}^{1,2,\ldots,\widehat{p}%
,\ldots,r,i,j}A\right)  \neq0$ (since we don't have $\left(  N\cup\left\{
i,j\right\}  \right)  \setminus\left\{  p\right\}  \in\mathcal{F}$). In other
words, we have%
\begin{equation}
\det\left(  \operatorname*{sub}\nolimits_{1,2,\ldots,r+1}^{1,2,\ldots
,\widehat{p},\ldots,r,i,j}A\right)  =0. \label{pf.lem.geg.K-bound.c2.pf.2}%
\end{equation}

Forget that we fixed $p$. We thus have proved
(\ref{pf.lem.geg.K-bound.c2.pf.2}) for each $p\in\left\{  1,2,\ldots
,r\right\}  $. Now, (\ref{pf.lem.geg.K-bound.c2.pf.1}) becomes%
\begin{align*}
&  \det\left(  \operatorname*{sub}\nolimits_{1,2,\ldots,r+2}^{1,2,\ldots
,r,i,j}A\right) \\
&  =\sum_{p=1}^{r}\left(  -1\right)  ^{\left(  r+2\right)  +p}a_{r+2,p}%
\underbrace{\det\left(  \operatorname*{sub}\nolimits_{1,2,\ldots
,r+1}^{1,2,\ldots,\widehat{p},\ldots,r,i,j}A\right)  }%
_{\substack{=0\\\text{(by (\ref{pf.lem.geg.K-bound.c2.pf.2}))}}}\\
&  \ \ \ \ \ \ \ \ \ \ +\underbrace{\left(  -1\right)  ^{\left(  r+2\right)
+\left(  r+1\right)  }}_{=-1}a_{r+2,i}\det\left(  \operatorname*{sub}%
\nolimits_{1,2,\ldots,r+1}^{1,2,\ldots,r,j}A\right) \\
&  \ \ \ \ \ \ \ \ \ \ +\underbrace{\left(  -1\right)  ^{\left(  r+2\right)
+\left(  r+2\right)  }}_{=1}a_{r+2,j}\det\left(  \operatorname*{sub}%
\nolimits_{1,2,\ldots,r+1}^{1,2,\ldots,r,i}A\right) \\
&  =\underbrace{\sum_{p=1}^{r}\left(  -1\right)  ^{\left(  r+2\right)
+p}a_{r+2,p}0}_{=0}-a_{r+2,i}\det\left(  \operatorname*{sub}%
\nolimits_{1,2,\ldots,r+1}^{1,2,\ldots,r,j}A\right)  +a_{r+2,j}\det\left(
\operatorname*{sub}\nolimits_{1,2,\ldots,r+1}^{1,2,\ldots,r,i}A\right) \\
&  =-\underbrace{a_{r+2,i}}_{\substack{=\det\left(  \operatorname*{sub}%
\nolimits_{1,2,\ldots,r+1}^{1,2,\ldots,r,i}A\right)  \cdot r_{i}\\\text{(by
(\ref{pf.lem.geg.K-bound.1}))}}}\det\left(  \operatorname*{sub}%
\nolimits_{1,2,\ldots,r+1}^{1,2,\ldots,r,j}A\right) \\
&  \ \ \ \ \ \ \ \ \ \ +\underbrace{a_{r+2,j}}_{\substack{=\det\left(
\operatorname*{sub}\nolimits_{1,2,\ldots,r+1}^{1,2,\ldots,r,j}A\right)  \cdot
r_{j}\\\text{(since (\ref{pf.lem.geg.K-bound.1}) (applied to }j\text{ instead
of }i\text{)}\\\text{yields }r_{j}=\dfrac{a_{r+2,j}}{\det\left(
\operatorname*{sub}\nolimits_{1,2,\ldots,r+1}^{1,2,\ldots,r,j}A\right)
}\text{)}}}\det\left(  \operatorname*{sub}\nolimits_{1,2,\ldots,r+1}%
^{1,2,\ldots,r,i}A\right) \\
&  =-\det\left(  \operatorname*{sub}\nolimits_{1,2,\ldots,r+1}^{1,2,\ldots
,r,i}A\right)  \cdot r_{i}\cdot\det\left(  \operatorname*{sub}%
\nolimits_{1,2,\ldots,r+1}^{1,2,\ldots,r,j}A\right) \\
&  \ \ \ \ \ \ \ \ \ \ +\det\left(  \operatorname*{sub}\nolimits_{1,2,\ldots
,r+1}^{1,2,\ldots,r,j}A\right)  \cdot r_{j}\cdot\det\left(
\operatorname*{sub}\nolimits_{1,2,\ldots,r+1}^{1,2,\ldots,r,i}A\right) \\
&  =\det\left(  \operatorname*{sub}\nolimits_{1,2,\ldots,r+1}^{1,2,\ldots
,r,i}A\right)  \det\left(  \operatorname*{sub}\nolimits_{1,2,\ldots
,r+1}^{1,2,\ldots,r,j}A\right)  \cdot\left(  r_{j}-r_{i}\right)  .
\end{align*}
Hence,%
\[
\det\left(  \operatorname*{sub}\nolimits_{1,2,\ldots,r+1}^{1,2,\ldots
,r,i}A\right)  \det\left(  \operatorname*{sub}\nolimits_{1,2,\ldots
,r+1}^{1,2,\ldots,r,j}A\right)  \cdot\left(  r_{j}-r_{i}\right)  =\det\left(
\operatorname*{sub}\nolimits_{1,2,\ldots,r+2}^{1,2,\ldots,r,i,j}A\right)
\neq0
\]
(by (\ref{pf.lem.geg.K-bound.c2.pf.neq0})). Thus, $r_{j}-r_{i}\neq0$, so that
$r_{i}\neq r_{j}$. This proves Claim 2.] \medskip

Claim 2 shows that the scalars $r_{i}$ for all $i\in C$ are distinct. Thus, we
have found $\left\vert C\right\vert $ distinct elements of $\mathbb{K}$
(namely, these scalars $r_{i}$ for all $i\in C$). Therefore, $\mathbb{K}$ must
have at least $\left\vert C\right\vert $ elements. In other words, $\left\vert
\mathbb{K}\right\vert \geq\left\vert C\right\vert $. Thus, Lemma
\ref{lem.geg.K-bound} is proved.
\end{proof}

\subsection{Proving the theorem}

We shall need one more lemma about Gaussian elimination greedoids:

\begin{lemma}
\label{lem.geg.strong-ax-ii}Let $\mathbb{K}$ be a field. Let $\mathcal{F}$ be
the Gaussian elimination greedoid of a vector family over $\mathbb{K}$. If
$B\in\mathcal{F}$ satisfies $\left\vert B\right\vert >0$, then there exists
$b\in B$ such that $B\setminus\left\{  b\right\}  \in\mathcal{F}$.
\end{lemma}

\begin{proof}
[Proof of Lemma \ref{lem.geg.strong-ax-ii}.]We shall use the notion of a
\textquotedblleft strong greedoid\textquotedblright, as defined in Definition
\ref{def.sg}.

Theorem \ref{thm.geg.strong} shows that every Gaussian elimination greedoid is
a strong greedoid. Hence, $\mathcal{F}$ is a strong greedoid. In other words,
$\mathcal{F}$ is a set system that satisfies the four axioms \textbf{(i)},
\textbf{(ii)}, \textbf{(iii)} and \textbf{(iv)} of Definition \ref{def.sg}.
Thus, in particular, $\mathcal{F}$ satisfies axiom \textbf{(ii)}. But this
axiom says precisely that if $B\in\mathcal{F}$ satisfies $\left\vert
B\right\vert >0$, then there exists $b\in B$ such that $B\setminus\left\{
b\right\}  \in\mathcal{F}$. Thus, Lemma \ref{lem.geg.strong-ax-ii} is proved.
\end{proof}

We can now prove Theorem \ref{thm.converse1}:

\begin{proof}
[Proof of Theorem \ref{thm.converse1}.]The definition of $\operatorname*{mcs}%
\left(  E,w,d\right)  $ shows that $\operatorname*{mcs}\left(  E,w,d\right)  $
is the maximum size of a clique of $\left(  E,w,d\right)  $. Thus, there
exists a clique $C$ of size $\operatorname*{mcs}\left(  E,w,d\right)  $.
Consider this $C$. Thus, $\operatorname*{mcs}\left(  E,w,d\right)  =\left\vert
C\right\vert $ (since $C$ has size $\operatorname*{mcs}\left(  E,w,d\right)  $).

We must prove that $\left\vert \mathbb{K}\right\vert \geq\operatorname*{mcs}%
\left(  E,w,d\right)  $. In other words, we must prove that $\left\vert
\mathbb{K}\right\vert \geq\left\vert C\right\vert $ (since
$\operatorname*{mcs}\left(  E,w,d\right)  =\left\vert C\right\vert $). If
$\left\vert C\right\vert \leq1$, then this is obvious (since $\left\vert
\mathbb{K}\right\vert \geq1$). Thus, we WLOG assume that $\left\vert
C\right\vert >1$. Hence, $\left\vert C\right\vert \geq2$.

The set $C$ is a clique. In other words, there exists a $\beta\in\mathbb{V}$
such that $C$ is a $\beta$-clique. Consider this $\beta$.

Choose some $c\in C$. (We can do this, since $\left\vert C\right\vert \geq
2>0$.) Set $B=B_{\beta}\left(  c\right)  $. Lemma \ref{lem.clique-and-clball}
\textbf{(b)} yields $C\subseteq B$, so that $B\supseteq C$.

A \textit{secant set} shall mean a subset $S$ of $E$ satisfying $S\in
\mathcal{F}$ and $\left\vert S\cap B\right\vert \geq2$.

The set $E$ itself satisfies $E\in\mathcal{F}$ (by Remark
\ref{rmk.bhargava.EinF}) and $\left\vert E\cap B\right\vert \geq
2$\ \ \ \ \footnote{\textit{Proof.} Since $B$ is a subset of $E$, we have
$E\cap B=B\supseteq C$. Thus, $\left\vert E\cap B\right\vert \geq\left\vert
C\right\vert \geq2$.}. In other words, $E$ is a secant set. Hence, there
exists at least one secant set. Thus, there exists a secant set of smallest
size (since there are only finitely many secant sets). Consider such a secant
set, and call it $S$.

Thus, $S$ is a secant set of smallest size. Hence, $S$ is a secant set; in
other words, $S$ is a subset of $E$ satisfying $S\in\mathcal{F}$ and
$\left\vert S\cap B\right\vert \geq2$. Since $S\cap B$ is a subset of $S$, we
have $\left\vert S\cap B\right\vert \leq\left\vert S\right\vert $, so that
$\left\vert S\right\vert \geq\left\vert S\cap B\right\vert \geq2>0$. Thus,
Lemma \ref{lem.geg.strong-ax-ii} (applied to $B=S$) yields that there exists
$b\in S$ such that $S\setminus\left\{  b\right\}  \in\mathcal{F}$. Consider
this $b$.

We now shall show several claims:

\begin{statement}
\textit{Claim 1:} We have $\left\vert \left(  S\cap B\right)  \setminus
\left\{  b\right\}  \right\vert <2$.
\end{statement}

[\textit{Proof of Claim 1:} The set $S\setminus\left\{  b\right\}  $ has
smaller size than $S$ (since $b\in S$), and thus cannot be a secant set (since
$S$ is a secant set of smallest size). In other words, $S\setminus\left\{
b\right\}  $ cannot be a subset of $E$ satisfying $S\setminus\left\{
b\right\}  \in\mathcal{F}$ and $\left\vert \left(  S\setminus\left\{
b\right\}  \right)  \cap B\right\vert \geq2$ (by the definition of
\textquotedblleft secant set\textquotedblright). Hence, we cannot have
$\left\vert \left(  S\setminus\left\{  b\right\}  \right)  \cap B\right\vert
\geq2$ (because $S\setminus\left\{  b\right\}  $ is a subset of $E$ satisfying
$S\setminus\left\{  b\right\}  \in\mathcal{F}$). In other words, we have
$\left\vert \left(  S\setminus\left\{  b\right\}  \right)  \cap B\right\vert
<2$.

But it is known from basic set theory that $\left(  X\cap Y\right)  \setminus
Z=\left(  X\setminus Z\right)  \cap Y$ for any three sets $X$, $Y$ and $Z$.
Applying this to $X=S$, $Y=B$ and $Z=\left\{  b\right\}  $, we obtain $\left(
S\cap B\right)  \setminus\left\{  b\right\}  =\left(  S\setminus\left\{
b\right\}  \right)  \cap B$. Hence, $\left\vert \left(  S\cap B\right)
\setminus\left\{  b\right\}  \right\vert =\left\vert \left(  S\setminus
\left\{  b\right\}  \right)  \cap B\right\vert <2$. This proves Claim 1.]
\medskip

\begin{statement}
\textit{Claim 2:} We have $b\in S\cap B$.
\end{statement}

[\textit{Proof of Claim 2:} Claim 1 yields $\left\vert \left(  S\cap B\right)
\setminus\left\{  b\right\}  \right\vert <2\leq\left\vert S\cap B\right\vert $
(since $\left\vert S\cap B\right\vert \geq2$). Hence, $\left\vert \left(
S\cap B\right)  \setminus\left\{  b\right\}  \right\vert \neq\left\vert S\cap
B\right\vert $, so that $\left(  S\cap B\right)  \setminus\left\{  b\right\}
\neq S\cap B$. Therefore, $b\in S\cap B$. This proves Claim 2.] \medskip

\begin{statement}
\textit{Claim 3:} We have $\left\vert S\cap B\right\vert =2$.
\end{statement}

[\textit{Proof of Claim 3:} Claim 1 says that $\left\vert \left(  S\cap
B\right)  \setminus\left\{  b\right\}  \right\vert <2$. But Claim 2 yields
$b\in S\cap B$. Thus, $\left\vert \left(  S\cap B\right)  \setminus\left\{
b\right\}  \right\vert =\left\vert S\cap B\right\vert -1$. Hence, $\left\vert
S\cap B\right\vert =\underbrace{\left\vert \left(  S\cap B\right)
\setminus\left\{  b\right\}  \right\vert }_{<2}+1<2+1=3$. In other words,
$\left\vert S\cap B\right\vert \leq2$ (since $\left\vert S\cap B\right\vert $
is an integer). Combining this with $\left\vert S\cap B\right\vert \geq2$, we
find $\left\vert S\cap B\right\vert =2$. This proves Claim 3.] \medskip

Claim 3 shows that the set $S\cap B$ has exactly two elements. One of these
two elements is $b$ (since Claim 2 says that $b\in S\cap B$); let $a$ be the
other element. Thus, $a\neq b$ and $S\cap B=\left\{  a,b\right\}  $. Hence,
$\left\{  a,b\right\}  =S\cap B\subseteq B$. Also, $\left\vert \left\{
a,b\right\}  \right\vert =2$ (since $a\neq b$).

Let $N=S\setminus B$. Then, it is easy to see that $S=N\cup\left\{
a,b\right\}  $\ \ \ \ \footnote{\textit{Proof.} Any two sets $X$ and $Y$
satisfy $X=\left(  X\setminus Y\right)  \cup\left(  X\cap Y\right)  $.
Applying this to $X=S$ and $Y=B$, we obtain $S=\underbrace{\left(  S\setminus
B\right)  }_{=N}\cup\underbrace{\left(  S\cap B\right)  }_{=\left\{
a,b\right\}  }=N\cup\left\{  a,b\right\}  $.} and $N\cap\left\{  a,b\right\}
=\varnothing$\ \ \ \ \footnote{\textit{Proof.} We have $\left\{  a,b\right\}
\subseteq B$. Hence, $\underbrace{N}_{=S\setminus B}\cap\underbrace{\left\{
a,b\right\}  }_{\subseteq B}\subseteq\left(  S\setminus B\right)  \cap
B=\varnothing$, so that $N\cap\left\{  a,b\right\}  =\varnothing$.}. From
$S=N\cup\left\{  a,b\right\}  $, we obtain
\begin{align}
\left\vert S\right\vert  &  =\left\vert N\cup\left\{  a,b\right\}  \right\vert
=\left\vert N\right\vert +\underbrace{\left\vert \left\{  a,b\right\}
\right\vert }_{=2}\ \ \ \ \ \ \ \ \ \ \left(  \text{since }N\cap\left\{
a,b\right\}  =\varnothing\right) \nonumber\\
&  =\left\vert N\right\vert +2. \label{pf.thm.converse-strong.4}%
\end{align}

Moreover,
\[
N=\underbrace{S}_{\subseteq E}\setminus B\subseteq E\setminus\underbrace{B}%
_{\supseteq C}\subseteq E\setminus C.
\]
Hence, the two subsets $N$ and $C$ of $E$ are disjoint.

Now, we have the following:

\begin{statement}
\textit{Claim 4:} Let $i\in C$. Then, $N\cap\left\{  i\right\}  =\varnothing$
and $N\cup\left\{  i\right\}  \in\mathcal{F}$.
\end{statement}

[\textit{Proof of Claim 4:} If we had $i\in N$, then we would have $i\notin B$
(since $i\in N=S\setminus B$), which would contradict $i\in C\subseteq B$.
Hence, we cannot have $i\in N$. Thus, $N\cap\left\{  i\right\}  =\varnothing$.
It remains to prove that $N\cup\left\{  i\right\}  \in\mathcal{F}$.

If we had $b\in N$, then we would have $b\notin B$ (since $b\in N=S\setminus
B$), which would contradict $b\in\left\{  a,b\right\}  \subseteq B$. Hence, we
cannot have $b\in N$. Thus, $b\notin N$, so that $N\setminus\left\{
b\right\}  =N$.

We have $S=N\cup\left\{  a,b\right\}  $, thus%
\[
S\setminus\left\{  b\right\}  =\left(  N\cup\left\{  a,b\right\}  \right)
\setminus\left\{  b\right\}  =\underbrace{\left(  N\setminus\left\{
b\right\}  \right)  }_{=N}\cup\underbrace{\left(  \left\{  a,b\right\}
\setminus\left\{  b\right\}  \right)  }_{\substack{=\left\{  a\right\}
\\\text{(since }a\neq b\text{)}}}=N\cup\left\{  a\right\}  .
\]
Hence, $N\cup\left\{  a\right\}  =S\setminus\left\{  b\right\}  \in
\mathcal{F}$. But $\left\{  a\right\}  $ and $\left\{  i\right\}  $ are
subsets of $B$ (since $a\in\left\{  a,b\right\}  \subseteq B$ and $i\in
C\subseteq B$), and satisfy $\left\vert \left\{  a\right\}  \right\vert
=\left\vert \left\{  i\right\}  \right\vert $ (since both $\left\vert \left\{
a\right\}  \right\vert $ and $\left\vert \left\{  i\right\}  \right\vert $
equal $1$) and $\left\{  i\right\}  \subseteq C$ (since $i\in C$) and
$N\cup\left\{  a\right\}  \in\mathcal{F}$. Hence, Corollary
\ref{cor.clique-and-clball3} (applied to $P=\left\{  a\right\}  $ and
$Q=\left\{  i\right\}  $) yields $N\cup\left\{  i\right\}  \in\mathcal{F}$.
This finishes the proof of Claim 4.] \medskip

\begin{statement}
\textit{Claim 5:} Let $i,j\in C$ be distinct. Then, $N\cap\left\{
i,j\right\}  =\varnothing$ and $N\cup\left\{  i,j\right\}  \in\mathcal{F}$.
\end{statement}

[\textit{Proof of Claim 5:} If we had $i\in N$, then we would have $i\notin B$
(since $i\in N=S\setminus B$), which would contradict $i\in C\subseteq B$.
Hence, we cannot have $i\in N$. In other words, $i\notin N$. Likewise,
$j\notin N$. Combining $i\notin N$ with $j\notin N$, we obtain $N\cap\left\{
i,j\right\}  =\varnothing$. It remains to prove that $N\cup\left\{
i,j\right\}  \in\mathcal{F}$.

Note that $\left\vert \left\{  a,b\right\}  \right\vert =2$ and $\left\vert
\left\{  i,j\right\}  \right\vert =2$ (since $i$ and $j$ are distinct). Hence,
$\left\vert \left\{  a,b\right\}  \right\vert =2=\left\vert \left\{
i,j\right\}  \right\vert $. Also, $i,j\in C$, so that $\left\{  i,j\right\}
\subseteq C\subseteq B$.

We have $S=N\cup\left\{  a,b\right\}  $, thus $N\cup\left\{  a,b\right\}
=S\in\mathcal{F}$. But $\left\{  a,b\right\}  $ and $\left\{  i,j\right\}  $
are subsets of $B$ (since $\left\{  a,b\right\}  \subseteq B$ and $\left\{
i,j\right\}  \subseteq B$), and satisfy $\left\vert \left\{  a,b\right\}
\right\vert =\left\vert \left\{  i,j\right\}  \right\vert $ and $\left\{
i,j\right\}  \subseteq C$ and $N\cup\left\{  a,b\right\}  \in\mathcal{F}$.
Hence, Corollary \ref{cor.clique-and-clball3} (applied to $P=\left\{
a,b\right\}  $ and $Q=\left\{  i,j\right\}  $) yields $N\cup\left\{
i,j\right\}  \in\mathcal{F}$. This finishes the proof of Claim 5.] \medskip

\begin{statement}
\textit{Claim 6:} Let $p\in N$. Let $i,j\in C$ be distinct. Then, $\left(
N\cup\left\{  i,j\right\}  \right)  \setminus\left\{  p\right\}
\notin\mathcal{F}$.
\end{statement}

[\textit{Proof of Claim 6:} Assume the contrary. Hence, $\left(  N\cup\left\{
i,j\right\}  \right)  \setminus\left\{  p\right\}  \in\mathcal{F}$. Also,
$\left\vert \left(  \left(  N\cup\left\{  i,j\right\}  \right)  \setminus
\left\{  p\right\}  \right)  \cap B\right\vert \geq2$%
\ \ \ \ \footnote{\textit{Proof.} We have $p\in N=S\setminus B$, so that
$p\notin B$. Hence, $B\setminus\left\{  p\right\}  =B$. Also, $i,j\in C$, so
that $\left\{  i,j\right\}  \subseteq C\subseteq B$.
\par
But it is known from set theory that $\left(  X\setminus Y\right)  \cap
Z=X\cap\left(  Z\setminus Y\right)  $ for any three sets $X$, $Y$ and $Z$.
Applying this to $X=N\cup\left\{  i,j\right\}  $, $Y=\left\{  p\right\}  $ and
$Z=B$, we obtain
\[
\left(  \left(  N\cup\left\{  i,j\right\}  \right)  \setminus\left\{
p\right\}  \right)  \cap B=\underbrace{\left(  N\cup\left\{  i,j\right\}
\right)  }_{\supseteq\left\{  i,j\right\}  }\cap\underbrace{\left(
B\setminus\left\{  p\right\}  \right)  }_{=B}\supseteq\left\{  i,j\right\}
\cap B=\left\{  i,j\right\}
\]
(since $\left\{  i,j\right\}  \subseteq B$). Hence, $\left\vert \left(
\left(  N\cup\left\{  i,j\right\}  \right)  \setminus\left\{  p\right\}
\right)  \cap B\right\vert \geq\left\vert \left\{  i,j\right\}  \right\vert
=2$ (since $i$ and $j$ are distinct). Qed.}. Therefore, $\left(  N\cup\left\{
i,j\right\}  \right)  \setminus\left\{  p\right\}  $ is a secant set (by the
definition of \textquotedblleft secant set\textquotedblright). Hence,
\begin{equation}
\left\vert \left(  N\cup\left\{  i,j\right\}  \right)  \setminus\left\{
p\right\}  \right\vert \geq\left\vert S\right\vert
\label{pf.thm.converse-strong.c6.pf.1}%
\end{equation}
(since $S$ was defined to be a secant set of smallest size).

We have $\left\vert \left\{  i,j\right\}  \right\vert =2$ (since $i$ and $j$
are distinct). But Claim 5 yields $N\cap\left\{  i,j\right\}  =\varnothing$,
so that $\left\vert N\cup\left\{  i,j\right\}  \right\vert =\left\vert
N\right\vert +\underbrace{\left\vert \left\{  i,j\right\}  \right\vert }%
_{=2}=\left\vert N\right\vert +2=\left\vert S\right\vert $ (by
(\ref{pf.thm.converse-strong.4})). But $p\in N\subseteq N\cup\left\{
i,j\right\}  $ and thus $\left\vert \left(  N\cup\left\{  i,j\right\}
\right)  \setminus\left\{  p\right\}  \right\vert =\underbrace{\left\vert
N\cup\left\{  i,j\right\}  \right\vert }_{=\left\vert S\right\vert
}-1=\left\vert S\right\vert -1<\left\vert S\right\vert $. This contradicts
(\ref{pf.thm.converse-strong.c6.pf.1}). This contradiction shows that our
assumption was false. Hence, Claim 6 is proved.] \medskip

But recall that $\mathcal{F}$ is the Gaussian elimination greedoid of a vector
family over $\mathbb{K}$ (by assumption). Let $\left(  v_{e}\right)  _{e\in
E}$ be this vector family. Recall that $N$ and $C$ are two disjoint subsets of
$E$. Moreover, the following facts hold:

\begin{enumerate}
\item[\textbf{(i)}] For any $i\in C$, we have $N\cup\left\{  i\right\}
\in\mathcal{F}$ (by Claim 4).

\item[\textbf{(ii)}] For any distinct $i,j\in C$, we have $N\cup\left\{
i,j\right\}  \in\mathcal{F}$ (by Claim 5).

\item[\textbf{(iii)}] For any $p\in N$ and any distinct $i,j\in C$, we have
$\left(  N\cup\left\{  i,j\right\}  \right)  \setminus\left\{  p\right\}
\notin\mathcal{F}$ (by Claim 6).
\end{enumerate}

Hence, Lemma \ref{lem.geg.K-bound} shows that $\left\vert \mathbb{K}%
\right\vert \geq\left\vert C\right\vert $. This proves Theorem
\ref{thm.converse1}.
\end{proof}

\section{\label{sect.plucker}Appendix: Gaussian elimination greedoids are
strong}

We shall now prove Theorem \ref{thm.geg.strong}. In order to do so, we will
first need to recall the definition of a strong greedoid (see, e.g.,
\cite[\S 6.1]{GriPet19}):

\begin{definition}
\label{def.sg}Let $E$ be a finite set. A set system $\mathcal{F}$ on ground
set $E$ is said to be a \textit{greedoid} if it satisfies the following three axioms:

\begin{enumerate}
\item[\textbf{(i)}] We have $\varnothing\in\mathcal{F}$.

\item[\textbf{(ii)}] If $B\in\mathcal{F}$ satisfies $\left\vert B\right\vert
>0$, then there exists $b\in B$ such that $B\setminus\left\{  b\right\}
\in\mathcal{F}$.

\item[\textbf{(iii)}] If $A,B\in\mathcal{F}$ satisfy $\left\vert B\right\vert
=\left\vert A\right\vert +1$, then there exists $b\in B\setminus A$ such that
$A\cup\left\{  b\right\}  \in\mathcal{F}$.
\end{enumerate}

\noindent Furthermore, $\mathcal{F}$ is said to be a \textit{strong greedoid}
if it satisfies the following axiom (in addition to axioms \textbf{(i)},
\textbf{(ii)} and \textbf{(iii)}):

\begin{enumerate}
\item[\textbf{(iv)}] If $A,B\in\mathcal{F}$ satisfy $\left\vert B\right\vert
=\left\vert A\right\vert +1$, then there exists $b\in B\setminus A$ such that
$A\cup\left\{  b\right\}  \in\mathcal{F}$ and $B\setminus\left\{  b\right\}
\in\mathcal{F}$.
\end{enumerate}
\end{definition}

Note that axiom \textbf{(iii)} in Definition \ref{def.sg} clearly follows from
axiom \textbf{(iv)}; thus, only axioms \textbf{(i)}, \textbf{(ii)} and
\textbf{(iv)} need to be checked in order to convince ourselves that a set
system is a strong greedoid.

We shall furthermore use a determinantal identity due to Pl\"{u}cker (one of
several facts known as the Pl\"{u}cker identities). To state this identity, we
will use the following notations:

\begin{definition}
\label{def.addcol-uncol}Let $\mathbb{K}$ be a commutative ring. Let
$n\in\mathbb{N}$ and $m\in\mathbb{N}$. Let $A\in\mathbb{K}^{n\times m}$ be an
$n\times m$-matrix. (Here and in the following, $\mathbb{K}^{n\times m}$
denotes the set of all $n\times m$-matrices with entries in $\mathbb{K}$.)

\begin{enumerate}
\item[\textbf{(a)}] If $v\in\mathbb{K}^{n\times1}$ is a column vector with $n$
entries, then $\left(  A\mid v\right)  $ will denote the $n\times\left(
m+1\right)  $-matrix whose $m+1$ columns are $A_{\bullet,1},A_{\bullet
,2},\ldots,A_{\bullet,m},v$ (from left to right). (Informally speaking,
$\left(  A\mid v\right)  $ is the matrix obtained when the column vector $v$
is \textquotedblleft attached\textquotedblright\ to $A$ at the right edge.)

\item[\textbf{(b)}] If $j\in\left\{  1,2,\ldots,m\right\}  $, then
$A_{\bullet,j}$ shall mean the $j$-th column of the matrix $A$. This is a
column vector with $n$ entries, i.e., an $n\times1$-matrix.

\item[\textbf{(c)}] If $i\in\left\{  1,2,\ldots,n\right\}  $, then $A_{\sim
i,\bullet}$ shall mean the matrix obtained from the matrix $A$ by removing the
$i$-th row. This is an $\left(  n-1\right)  \times m$-matrix.

\item[\textbf{(d)}] If $i\in\left\{  1,2,\ldots,n\right\}  $ and $j\in\left\{
1,2,\ldots,m\right\}  $, then $A_{\sim i,\sim j}$ shall mean the matrix
obtained from $A$ by removing the $i$-th row and the $j$-th column. This is an
$\left(  n-1\right)  \times\left(  m-1\right)  $-matrix.
\end{enumerate}
\end{definition}

We can now state our Pl\"{u}cker identity:

\begin{proposition}
\label{prop.desnanot.XY}Let $\mathbb{K}$ be a commutative ring. Let $n$ be a
positive integer. Let $X\in\mathbb{K}^{n\times\left(  n-1\right)  }$ and
$Y\in\mathbb{K}^{n\times n}$. Let $i\in\left\{  1,2,\ldots,n\right\}  $. Then,%
\[
\det\left(  X_{\sim i,\bullet}\right)  \det Y=\sum_{q=1}^{n}\left(  -1\right)
^{n+q}\det\left(  X\mid Y_{\bullet,q}\right)  \det\left(  Y_{\sim i,\sim
q}\right)  .
\]

\end{proposition}

\begin{example}
If we set $n=3$, $X=\left(
\begin{array}
[c]{cc}%
a & a^{\prime}\\
b & b^{\prime}\\
c & c^{\prime}%
\end{array}
\right)  $, $Y=\left(
\begin{array}
[c]{ccc}%
x & x^{\prime} & x^{\prime\prime}\\
y & y^{\prime} & y^{\prime\prime}\\
z & z^{\prime} & z^{\prime\prime}%
\end{array}
\right)  $ and $i=2$, then Proposition \ref{prop.desnanot.XY} states that%
\begin{align*}
&  \det\left(
\begin{array}
[c]{cc}%
a & a^{\prime}\\
c & c^{\prime}%
\end{array}
\right)  \det\left(
\begin{array}
[c]{ccc}%
x & x^{\prime} & x^{\prime\prime}\\
y & y^{\prime} & y^{\prime\prime}\\
z & z^{\prime} & z^{\prime\prime}%
\end{array}
\right) \\
&  =\det\left(
\begin{array}
[c]{ccc}%
a & a^{\prime} & x\\
b & b^{\prime} & y\\
c & c^{\prime} & z
\end{array}
\right)  \det\left(
\begin{array}
[c]{cc}%
x^{\prime} & x^{\prime\prime}\\
z^{\prime} & z^{\prime\prime}%
\end{array}
\right)  -\det\left(
\begin{array}
[c]{ccc}%
a & a^{\prime} & x^{\prime}\\
b & b^{\prime} & y^{\prime}\\
c & c^{\prime} & z^{\prime}%
\end{array}
\right)  \det\left(
\begin{array}
[c]{cc}%
x & x^{\prime\prime}\\
z & z^{\prime\prime}%
\end{array}
\right) \\
&  \ \ \ \ \ \ \ \ \ \ +\det\left(
\begin{array}
[c]{ccc}%
a & a^{\prime} & x^{\prime\prime}\\
b & b^{\prime} & y^{\prime\prime}\\
c & c^{\prime} & z^{\prime\prime}%
\end{array}
\right)  \det\left(
\begin{array}
[c]{cc}%
x & x^{\prime}\\
z & z^{\prime}%
\end{array}
\right)  .
\end{align*}

\end{example}

Proposition \ref{prop.desnanot.XY} is precisely \cite[Proposition
6.137]{detnotes} (with $A$, $C$ and $v$ renamed as $X$, $Y$ and $i$).

\begin{proof}
[Proof of Theorem \ref{thm.geg.strong}.]We need to prove that $\mathcal{G}$ is
a strong greedoid. In other words, we need to prove that the axioms
\textbf{(i)}, \textbf{(ii)}, \textbf{(iii)} and \textbf{(iv)} from Definition
\ref{def.sg} are satisfied for $\mathcal{F}=\mathcal{G}$. Let us do this:
\medskip

\textit{Proof of the axiom \textbf{(i)} for }$\mathcal{F}=\mathcal{G}%
$\textit{:} The vector family $\left(  \pi_{\left\vert \varnothing\right\vert
}\left(  v_{e}\right)  \right)  _{e\in\varnothing}\in\left(  \mathbb{K}%
^{\left\vert \varnothing\right\vert }\right)  ^{\varnothing}$ is empty, and
thus is linearly independent. In other words, $\varnothing\in\mathcal{G}$.
This proves that the axiom \textbf{(i)} from Definition \ref{def.sg} is
satisfied for $\mathcal{F}=\mathcal{G}$. \medskip

\textit{Proof of the axiom \textbf{(ii)} for }$\mathcal{F}=\mathcal{G}%
$\textit{:} Let $B\in\mathcal{G}$ satisfy $\left\vert B\right\vert >0$. We
shall show that there exists $b\in B$ such that $B\setminus\left\{  b\right\}
\in\mathcal{G}$.

Let $n=\left\vert B\right\vert $. Thus, $n=\left\vert B\right\vert >0$.

We have $B\in\mathcal{G}$. In other words, the family $\left(  \pi_{\left\vert
B\right\vert }\left(  v_{e}\right)  \right)  _{e\in B}\in\left(
\mathbb{K}^{\left\vert B\right\vert }\right)  ^{B}$ is linearly independent
(by the definition of $\mathcal{G}$). In other words, the family $\left(
\pi_{n}\left(  v_{e}\right)  \right)  _{e\in B}\in\left(  \mathbb{K}%
^{n}\right)  ^{B}$ is linearly independent (since $\left\vert B\right\vert =n$).

Let $Y$ be the $n\times n$-matrix whose columns are the vectors $\pi
_{n}\left(  v_{e}\right)  $ for all $e\in B$ (in some order, with no
repetition). The columns of this $n\times n$-matrix $Y$ are thus linearly
independent (since the family $\left(  \pi_{n}\left(  v_{e}\right)  \right)
_{e\in B}\in\left(  \mathbb{K}^{\left\vert B\right\vert }\right)  ^{B}$ is
linearly independent). Hence, this $n\times n$-matrix $Y$ is invertible. In
other words, $\det Y\neq0$.

Now, let $y_{i,j}$ be the $\left(  i,j\right)  $-th entry of the matrix $Y$
for each $i,j\in\left\{  1,2,\ldots,n\right\}  $. Then, Laplace expansion
along the $n$-th row yields%
\[
\det Y=\sum_{q=1}^{n}\left(  -1\right)  ^{n+q}y_{n,q}\det\left(  Y_{\sim
n,\sim q}\right)  .
\]
If every $q\in\left\{  1,2,\ldots,n\right\}  $ satisfied $\det\left(  Y_{\sim
n,\sim q}\right)  =0$, then this would become%
\[
\det Y=\sum_{q=1}^{n}\left(  -1\right)  ^{n+q}y_{n,q}\underbrace{\det\left(
Y_{\sim n,\sim q}\right)  }_{=0}=0,
\]
which would contradict $\det Y\neq0$. Hence, not every $q\in\left\{
1,2,\ldots,n\right\}  $ satisfies $\det\left(  Y_{\sim n,\sim q}\right)  =0$.
In other words, there exists at least one $q\in\left\{  1,2,\ldots,n\right\}
$ such that $\det\left(  Y_{\sim n,\sim q}\right)  \neq0$. Consider this $q$.
The $q$-th column of $Y$ is the vector $\pi_{n}\left(  v_{f}\right)  $ for
some $f\in B$ (by the definition of $Y$); consider this $f$. Note that
$\left\vert B\setminus\left\{  f\right\}  \right\vert =\left\vert B\right\vert
-1=n-1$ (since $\left\vert B\right\vert =n$).

The $\left(  n-1\right)  \times\left(  n-1\right)  $-matrix $Y_{\sim n,\sim
q}$ is invertible (since $\det\left(  Y_{\sim n,\sim q}\right)  \neq0$); thus,
its $n-1$ columns are linearly independent. But these $n-1$ columns are
precisely the $n-1$ vectors $\pi_{n-1}\left(  v_{e}\right)  $ for all $e\in
B\setminus\left\{  f\right\}  $ (by the definition of $Y$ and the choice of
$f$). Hence, the $n-1$ vectors $\pi_{n-1}\left(  v_{e}\right)  $ for all $e\in
B\setminus\left\{  f\right\}  $ are linearly independent. In other words, the
family $\left(  \pi_{n-1}\left(  v_{e}\right)  \right)  _{e\in B\setminus
\left\{  f\right\}  }$ is linearly independent. In other words, the family
$\left(  \pi_{\left\vert B\setminus\left\{  f\right\}  \right\vert }\left(
v_{e}\right)  \right)  _{e\in B\setminus\left\{  f\right\}  }$ is linearly
independent (since $\left\vert B\setminus\left\{  f\right\}  \right\vert
=n-1$). In other words, $B\setminus\left\{  f\right\}  \in\mathcal{G}$ (by the
definition of $\mathcal{G}$). Hence, there exists $b\in B$ such that
$B\setminus\left\{  b\right\}  \in\mathcal{G}$ (namely, $b=f$). This shows
that the axiom \textbf{(ii)} from Definition \ref{def.sg} is satisfied for
$\mathcal{F}=\mathcal{G}$. \medskip

\textit{Proof of the axiom \textbf{(iv)} for }$\mathcal{F}=\mathcal{G}%
$\textit{:} Let $A,B\in\mathcal{G}$ satisfy $\left\vert B\right\vert
=\left\vert A\right\vert +1$. We shall show that there exists $b\in B\setminus
A$ such that $A\cup\left\{  b\right\}  \in\mathcal{G}$ and $B\setminus\left\{
b\right\}  \in\mathcal{G}$.

Let $n=\left\vert B\right\vert $. Thus, $n=\left\vert B\right\vert =\left\vert
A\right\vert +1\geq1>0$. Also, from $n=\left\vert A\right\vert +1$, we obtain
$\left\vert A\right\vert =n-1$.

We have $B\in\mathcal{G}$. In other words, the family $\left(  \pi_{\left\vert
B\right\vert }\left(  v_{e}\right)  \right)  _{e\in B}\in\left(
\mathbb{K}^{\left\vert B\right\vert }\right)  ^{B}$ is linearly independent
(by the definition of $\mathcal{G}$). In other words, the family $\left(
\pi_{n}\left(  v_{e}\right)  \right)  _{e\in B}\in\left(  \mathbb{K}%
^{n}\right)  ^{B}$ is linearly independent (since $\left\vert B\right\vert
=n$). Similarly, the family $\left(  \pi_{n-1}\left(  v_{e}\right)  \right)
_{e\in A}\in\left(  \mathbb{K}^{n-1}\right)  ^{A}$ is linearly independent.

Let $Y$ be the $n\times n$-matrix whose columns are the vectors $\pi
_{n}\left(  v_{e}\right)  $ for all $e\in B$ (in some order, with no
repetition). The columns of this $n\times n$-matrix $Y$ are thus linearly
independent (since the family $\left(  \pi_{n}\left(  v_{e}\right)  \right)
_{e\in B}\in\left(  \mathbb{K}^{\left\vert B\right\vert }\right)  ^{B}$ is
linearly independent). Hence, this $n\times n$-matrix $Y$ is invertible. In
other words, $\det Y\neq0$.

Let $X$ be the $n\times\left(  n-1\right)  $-matrix whose columns are the
vectors $\pi_{n}\left(  v_{e}\right)  $ for all $e\in A$ (in some order, with
no repetition). Then, the columns of the $\left(  n-1\right)  \times\left(
n-1\right)  $-matrix $X_{\sim n,\bullet}$ are the vectors $\pi_{n-1}\left(
v_{e}\right)  $ for all $e\in A$; thus, they are linearly independent (since
the family $\left(  \pi_{n-1}\left(  v_{e}\right)  \right)  _{e\in A}%
\in\left(  \mathbb{K}^{n-1}\right)  ^{A}$ is linearly independent). Hence,
this $\left(  n-1\right)  \times\left(  n-1\right)  $-matrix $X_{\sim
n,\bullet}$ is invertible. In other words, $\det\left(  X_{\sim n,\bullet
}\right)  \neq0$.

But $\mathbb{K}$ is a field and thus an integral domain. Hence, from
$\det\left(  X_{\sim n,\bullet}\right)  \neq0$ and $\det Y\neq0$, we obtain%
\begin{equation}
\det\left(  X_{\sim n,\bullet}\right)  \det Y\neq0.
\label{pf.thm.gog.strong.iv.3}%
\end{equation}

If every $q\in\left\{  1,2,\ldots,n\right\}  $ satisfied $\det\left(  X\mid
Y_{\bullet,q}\right)  \det\left(  Y_{\sim n,\sim q}\right)  =0$, then we would
have%
\begin{align*}
\det\left(  X_{\sim n,\bullet}\right)  \det Y  &  =\sum_{q=1}^{n}\left(
-1\right)  ^{n+q}\underbrace{\det\left(  X\mid Y_{\bullet,q}\right)
\det\left(  Y_{\sim n,\sim q}\right)  }_{=0}\\
&  \ \ \ \ \ \ \ \ \ \ \left(  \text{by Proposition \ref{prop.desnanot.XY},
applied to }i=n\right) \\
&  =0,
\end{align*}
which would contradict (\ref{pf.thm.gog.strong.iv.3}). Hence, not every
$q\in\left\{  1,2,\ldots,n\right\}  $ satisfies \newline$\det\left(  X\mid
Y_{\bullet,q}\right)  \det\left(  Y_{\sim n,\sim q}\right)  =0$. In other
words, some $q\in\left\{  1,2,\ldots,n\right\}  $ satisfies $\det\left(  X\mid
Y_{\bullet,q}\right)  \det\left(  Y_{\sim n,\sim q}\right)  \neq0$. Consider
this $q$. The $q$-th column of $Y$ is the vector $\pi_{n}\left(  v_{f}\right)
$ for some $f\in B$ (by the definition of $Y$); consider this $f$. Thus,
$Y_{\bullet,q}=\pi_{n}\left(  v_{f}\right)  $.

Note that $\left\vert B\setminus\left\{  f\right\}  \right\vert =\left\vert
B\right\vert -1=n-1$ (since $\left\vert B\right\vert =n$).

From $\det\left(  X\mid Y_{\bullet,q}\right)  \det\left(  Y_{\sim n,\sim
q}\right)  \neq0$, we obtain $\det\left(  X\mid Y_{\bullet,q}\right)  \neq0$
and $\det\left(  Y_{\sim n,\sim q}\right)  \neq0$.

The $\left(  n-1\right)  \times\left(  n-1\right)  $-matrix $Y_{\sim n,\sim
q}$ is invertible (since $\det\left(  Y_{\sim n,\sim q}\right)  \neq0$); thus,
its $n-1$ columns are linearly independent. But these $n-1$ columns are
precisely the $n-1$ vectors $\pi_{n-1}\left(  v_{e}\right)  $ for all $e\in
B\setminus\left\{  f\right\}  $ (by the definition of $Y$ and the choice of
$f$). Hence, the $n-1$ vectors $\pi_{n-1}\left(  v_{e}\right)  $ for all $e\in
B\setminus\left\{  f\right\}  $ are linearly independent. In other words, the
family $\left(  \pi_{n-1}\left(  v_{e}\right)  \right)  _{e\in B\setminus
\left\{  f\right\}  }$ is linearly independent. In other words, the family
$\left(  \pi_{\left\vert B\setminus\left\{  f\right\}  \right\vert }\left(
v_{e}\right)  \right)  _{e\in B\setminus\left\{  f\right\}  }$ is linearly
independent (since $\left\vert B\setminus\left\{  f\right\}  \right\vert
=n-1$). In other words, $B\setminus\left\{  f\right\}  \in\mathcal{G}$ (by the
definition of $\mathcal{G}$).

The $n\times n$-matrix $\left(  X\mid Y_{\bullet,q}\right)  $ is invertible
(since $\det\left(  X\mid Y_{\bullet,q}\right)  \neq0$); thus, its $n$ columns
are linearly independent. But these $n$ columns are precisely the $n-1$
vectors $\pi_{n}\left(  v_{e}\right)  $ for all $e\in A$ as well as the extra
vector $Y_{\bullet,q}=\pi_{n}\left(  v_{f}\right)  $. Thus, if we had $f\in
A$, then these $n$ columns would contain two equal vectors (namely, the extra
vector $Y_{\bullet,q}=\pi_{n}\left(  v_{f}\right)  $ would be equal to one of
the vectors $\pi_{n}\left(  v_{e}\right)  $ with $e\in A$), and thus would be
linearly dependent. But this would contradict the fact that these $n$ columns
are linearly independent. Hence, we cannot have $f\in A$. Thus, $f\notin A$.
Combining $f\in B$ with $f\notin A$, we obtain $f\in B\setminus A$. From
$f\notin A$, we also obtain $\left\vert A\cup\left\{  f\right\}  \right\vert
=\left\vert A\right\vert +1=n$ (since $\left\vert A\right\vert =n-1$).

Recall that the $n$ columns of the $n\times n$-matrix $\left(  X\mid
Y_{\bullet,q}\right)  $ are the $n-1$ vectors $\pi_{n}\left(  v_{e}\right)  $
for all $e\in A$ as well as the extra vector $Y_{\bullet,q}=\pi_{n}\left(
v_{f}\right)  $. In other words, they are the $n$ vectors $\pi_{n}\left(
v_{e}\right)  $ for all $e\in A\cup\left\{  f\right\}  $ (since $f\notin A$).
Hence, the $n$ vectors $\pi_{n}\left(  v_{e}\right)  $ for all $e\in
A\cup\left\{  f\right\}  $ are linearly independent (since the $n$ columns of
the $n\times n$-matrix $\left(  X\mid Y_{\bullet,q}\right)  $ are linearly
independent). In other words, the family $\left(  \pi_{n}\left(  e\right)
\right)  _{e\in A\cup\left\{  f\right\}  }$ is linearly independent. In other
words, the family $\left(  \pi_{\left\vert A\cup\left\{  f\right\}
\right\vert }\left(  e\right)  \right)  _{e\in A\cup\left\{  f\right\}  }$ is
linearly independent (since $n=\left\vert A\cup\left\{  f\right\}  \right\vert
$). In other words, $A\cup\left\{  f\right\}  \in\mathcal{G}$ (by the
definition of $\mathcal{G}$).

We now know that $f\in B\setminus A$ and $A\cup\left\{  f\right\}
\in\mathcal{G}$ and $B\setminus\left\{  f\right\}  \in\mathcal{G}$. Hence,
there exists $b\in B\setminus A$ such that $A\cup\left\{  b\right\}
\in\mathcal{G}$ and $B\setminus\left\{  b\right\}  \in\mathcal{G}$ (namely,
$b=f$). This shows that the axiom \textbf{(iv)} from Definition \ref{def.sg}
is satisfied for $\mathcal{F}=\mathcal{G}$. \medskip

\textit{Proof of the axiom \textbf{(iii)} for }$\mathcal{F}=\mathcal{G}%
$\textit{:} Clearly, axiom \textbf{(iii)} from Definition \ref{def.sg} follows
from axiom \textbf{(iv)}. Thus, axiom \textbf{(iii)} holds for $\mathcal{F}%
=\mathcal{G}$ (since we already know that axiom \textbf{(iv)} holds for
$\mathcal{F}=\mathcal{G}$). \medskip

We have now showed that the system $\mathcal{G}$ satisfies the four axioms
\textbf{(i)}, \textbf{(ii)}, \textbf{(iii)} and \textbf{(iv)} from Definition
\ref{def.sg}. Thus, $\mathcal{G}$ is a strong greedoid. This proves Theorem
\ref{thm.geg.strong}.
\end{proof}

\section{\label{sect.matroid-rep}Appendix: Proof of Proposition
\ref{prop.geg.matroid-rep}}

We shall now prove Proposition \ref{prop.geg.matroid-rep}. This proof is
mostly bookkeeping.

\begin{proof}
[Proof of Proposition \ref{prop.geg.matroid-rep} (sketched).]If $\mathcal{G}%
_{k}$ is empty, then the claim is obvious. Thus, we WLOG assume that
$\mathcal{G}_{k}$ is nonempty.

We assumed that $\mathcal{G}$ is a Gaussian elimination greedoid. In other
words, there exist a field $\mathbb{K}$ and a vector family $\left(
v_{e}\right)  _{e\in E}$ over $\mathbb{K}$ such that $\mathcal{G}$ is the
Gaussian elimination greedoid of this vector family $\left(  v_{e}\right)
_{e\in E}$. Consider these $\mathbb{K}$ and $\left(  v_{e}\right)  _{e\in E}$.
Let $X$ be the $k\times\left\vert E\right\vert $-matrix whose columns are the
vectors $\pi_{k}\left(  v_{e}\right)  \in\mathbb{K}^{k}$ for all $e\in E$ (in
some order). Then, there is a bijection $\phi:\left\{  1,2,\ldots,\left\vert
E\right\vert \right\}  \rightarrow E$ such that the columns of the matrix $X$
are $\pi_{k}\left(  v_{\phi\left(  1\right)  }\right)  ,\pi_{k}\left(
v_{\phi\left(  2\right)  }\right)  ,\ldots,\pi_{k}\left(  v_{\phi\left(
\left\vert E\right\vert \right)  }\right)  $ (in this order). Consider this
$\phi$. Hence,%
\begin{equation}
\left(  f\text{-th column of }X\right)  =\pi_{k}\left(  v_{\phi\left(
f\right)  }\right)  \label{pf.prop.geg.matroid-rep.fcolX=}%
\end{equation}
for each $f\in\left\{  1,2,\ldots,\left\vert E\right\vert \right\}  $. Let
$\mathcal{V}$ be the vector matroid\footnote{See \cite[\S 1.1]{Oxley11} for
the definition of a vector matroid.} of the matrix $X$. This is a matroid on
the set $\left\{  1,2,\ldots,\left\vert E\right\vert \right\}  $; its rank is
at most $k$ (since $X$ has $k$ rows and therefore rank $\leq k$).

The bijection $\phi$ can be used to transport the matroid $\mathcal{V}$ onto
the set $E$. More precisely, we can define a matroid $\mathcal{V}^{\prime}$ on
the set $E$ by%
\[
\mathcal{V}^{\prime}=\left\{  \phi\left(  I\right)  \ \mid\ I\in
\mathcal{V}\right\}
\]
(where we regard a matroid as a collection of independent sets). The
independent sets of this matroid $\mathcal{V}^{\prime}$ are the subsets $I$ of
$E$ for which $\phi^{-1}\left(  I\right)  $ is an independent set of
$\mathcal{V}$. Moreover, the map $\phi$ is an isomorphism from the matroid
$\mathcal{V}$ to the matroid $\mathcal{V}^{\prime}$. Thus, the matroid
$\mathcal{V}^{\prime}$ is isomorphic to the vector matroid $\mathcal{V}$, and
hence is representable. Our goal is now to show that $\mathcal{G}_{k}$ is the
collection of bases of $\mathcal{V}$.

Let $F$ be a $k$-element subset of $E$. Then, we have the following chain of
logical equivalences:%
\begin{align}
&  \ \left(  F\in\mathcal{G}_{k}\right) \nonumber\\
&  \Longleftrightarrow\ \left(  F\in\mathcal{G}\right)
\ \ \ \ \ \ \ \ \ \ \left(  \text{since }\mathcal{G}_{k}\text{ is the set of
all }k\text{-element sets in }\mathcal{G}\right) \nonumber\\
&  \Longleftrightarrow\ \left(  \text{the family }\left(  \pi_{\left\vert
F\right\vert }\left(  v_{e}\right)  \right)  _{e\in F}\in\left(
\mathbb{K}^{\left\vert F\right\vert }\right)  ^{F}\text{ is linearly
independent}\right) \nonumber\\
&  \ \ \ \ \ \ \ \ \ \ \left(
\begin{array}
[c]{c}%
\text{by the definition of a Gaussian elimination greedoid,}\\
\text{since }\mathcal{G}\text{ is the Gaussian elimination greedoid of
}\left(  v_{e}\right)  _{e\in E}%
\end{array}
\right) \nonumber\\
&  \Longleftrightarrow\ \left(  \text{the family }\left(  \pi_{k}\left(
v_{e}\right)  \right)  _{e\in F}\in\left(  \mathbb{K}^{k}\right)  ^{F}\text{
is linearly independent}\right) \nonumber\\
&  \ \ \ \ \ \ \ \ \ \ \left(  \text{since }\left\vert F\right\vert =k\text{
(because }F\text{ is a }k\text{-element set)}\right) \nonumber\\
&  \Longleftrightarrow\ \left(  \text{the family }\left(  \pi_{k}\left(
v_{\phi\left(  f\right)  }\right)  \right)  _{f\in\phi^{-1}\left(  F\right)
}\in\left(  \mathbb{K}^{k}\right)  ^{\phi^{-1}\left(  F\right)  }\text{ is
linearly independent}\right) \nonumber\\
&  \ \ \ \ \ \ \ \ \ \ \left(
\begin{array}
[c]{c}%
\text{since }\phi:\left\{  1,2,\ldots,\left\vert E\right\vert \right\}
\rightarrow E\text{ is a bijection, and thus the}\\
\text{family }\left(  \pi_{k}\left(  v_{\phi\left(  f\right)  }\right)
\right)  _{f\in\phi^{-1}\left(  F\right)  }\text{ is just a reindexing of the
family }\left(  \pi_{k}\left(  v_{e}\right)  \right)  _{e\in F}%
\end{array}
\right) \nonumber\\
&  \Longleftrightarrow\ \left(  \text{the family }\left(  f\text{-th column of
}X\right)  _{f\in\phi^{-1}\left(  F\right)  }\in\left(  \mathbb{K}^{k}\right)
^{\phi^{-1}\left(  F\right)  }\text{ is linearly independent}\right)
\nonumber\\
&  \ \ \ \ \ \ \ \ \ \ \left(  \text{since
(\ref{pf.prop.geg.matroid-rep.fcolX=}) yields }\left(  f\text{-th column of
}X\right)  _{f\in\phi^{-1}\left(  F\right)  }=\left(  \pi_{k}\left(
v_{\phi\left(  f\right)  }\right)  \right)  _{f\in\phi^{-1}\left(  F\right)
}\right) \nonumber\\
&  \Longleftrightarrow\ \left(  \phi^{-1}\left(  F\right)  \text{ is an
independent set of the vector matroid of }X\right) \nonumber\\
&  \ \ \ \ \ \ \ \ \ \ \left(  \text{by the definition of a vector
matroid}\right) \nonumber\\
&  \Longleftrightarrow\ \left(  \phi^{-1}\left(  F\right)  \text{ is an
independent set of }\mathcal{V}\right)
\label{pf.prop.geg.matroid-rep.long-equiv}\\
&  \ \ \ \ \ \ \ \ \ \ \left(  \text{since }\mathcal{V}\text{ is the vector
matroid of }X\right)  .\nonumber
\end{align}

Forget that we fixed $F$. We thus have proved the equivalence
(\ref{pf.prop.geg.matroid-rep.long-equiv}) for each $k$-element subset $F$ of
$E$.

Recall that $\mathcal{G}_{k}$ is nonempty. In other words, there exists some
$B\in\mathcal{G}_{k}$. Consider this $B$. Thus, $B$ is a $k$-element set in
$\mathcal{G}$ (since $\mathcal{G}_{k}$ is the set of all $k$-element sets in
$\mathcal{G}$). Hence, the equivalence
(\ref{pf.prop.geg.matroid-rep.long-equiv}) (applied to $F=B$) yields the
equivalence%
\[
\left(  B\in\mathcal{G}_{k}\right)  \ \Longleftrightarrow\ \left(  \phi
^{-1}\left(  B\right)  \text{ is an independent set of }\mathcal{V}\right)  .
\]
Thus, $\phi^{-1}\left(  B\right)  $ is an independent set of $\mathcal{V}$
(since $B\in\mathcal{G}_{k}$). Since $\phi$ is a bijection, we have
$\left\vert \phi^{-1}\left(  B\right)  \right\vert =\left\vert B\right\vert
=k$ (since $B$ is a $k$-element set). In other words, $\phi^{-1}\left(
B\right)  $ is a $k$-element set. Thus, the matroid $\mathcal{V}$ has a
$k$-element independent set (namely, $\phi^{-1}\left(  B\right)  $).
Therefore, the rank of $\mathcal{V}$ is at least $k$. Since we also know that
the rank of $\mathcal{V}$ is at most $k$, we thus conclude that the rank of
$\mathcal{V}$ is exactly $k$. Hence, the rank of $\mathcal{V}^{\prime}$ is
exactly $k$ as well (since the matroid $\mathcal{V}^{\prime}$ is isomorphic to
the matroid $\mathcal{V}$). Therefore, the bases of $\mathcal{V}^{\prime}$ are
the $k$-element independent sets of $\mathcal{V}^{\prime}$. Thus, a
$k$-element set is an independent set of $\mathcal{V}^{\prime}$ if and only if
it is a basis of $\mathcal{V}^{\prime}$.

Now, let $F$ again be a $k$-element subset of $E$. Thus, $F$ is an independent
set of $\mathcal{V}^{\prime}$ if and only if $F$ is a basis of $\mathcal{V}%
^{\prime}$ (since a $k$-element set is an independent set of $\mathcal{V}%
^{\prime}$ if and only if it is a basis of $\mathcal{V}^{\prime}$). Also,
$\phi^{-1}\left(  F\right)  $ is an independent set of $\mathcal{V}$ if and
only if $F$ is an independent set of $\mathcal{V}^{\prime}$ (since the
independent sets of the matroid $\mathcal{V}^{\prime}$ are the subsets $I$ of
$E$ for which $\phi^{-1}\left(  I\right)  $ is an independent set of
$\mathcal{V}$).

Now, we have the following chain of logical equivalences:%
\begin{align*}
\left(  F\in\mathcal{G}_{k}\right)  \  &  \Longleftrightarrow\ \left(
\phi^{-1}\left(  F\right)  \text{ is an independent set of }\mathcal{V}%
\right)  \ \ \ \ \ \ \ \ \ \ \left(  \text{by
(\ref{pf.prop.geg.matroid-rep.long-equiv})}\right) \\
&  \Longleftrightarrow\ \left(  F\text{ is an independent set of }%
\mathcal{V}^{\prime}\right) \\
&  \ \ \ \ \ \ \ \ \ \ \left(
\begin{array}
[c]{c}%
\text{since }\phi^{-1}\left(  F\right)  \text{ is an independent set of
}\mathcal{V}\\
\text{if and only if }F\text{ is an independent set of }\mathcal{V}^{\prime}%
\end{array}
\right) \\
&  \Longleftrightarrow\ \left(  F\text{ is a basis of }\mathcal{V}^{\prime
}\right)
\end{align*}
(since $F$ is an independent set of $\mathcal{V}^{\prime}$ if and only if $F$
is a basis of $\mathcal{V}^{\prime}$).

Forget that we fixed $F$. We thus have proved that, for each $k$-element
subset $F$ of $E$, we have the equivalence%
\[
\left(  F\in\mathcal{G}_{k}\right)  \ \Longleftrightarrow\ \left(  F\text{ is
a basis of }\mathcal{V}^{\prime}\right)  .
\]
Hence, $\mathcal{G}_{k}$ is the collection of all bases of $\mathcal{V}%
^{\prime}$ (since all bases of $\mathcal{V}^{\prime}$ are $k$-element subsets
of $E$ (because $\mathcal{V}^{\prime}$ is a matroid of rank $k$)). Hence,
$\mathcal{G}_{k}$ is the collection of bases of a representable matroid on the
ground set $E$ (since $\mathcal{V}^{\prime}$ is a representable matroid on the
ground set $E$). This proves Proposition \ref{prop.geg.matroid-rep}.
\end{proof}

\section{\label{sect.ord-proofs}Appendix: Proofs of some properties of
$\mathbb{L}$}

We shall next prove four lemmas that were left unproven in Section
\ref{sect.p-adic}: namely, Lemma \ref{lem.ord.triangle}, Lemma
\ref{lem.valadic-ut.wd}, Lemma \ref{lem.piLK} and Lemma \ref{lem.pinot0}.

\begin{proof}
[Proof of Lemma \ref{lem.ord.triangle}.]\textbf{(a)} Let $a\in\mathbb{L}$ be a
nonzero element. We must prove that the element $a$ belongs to $\mathbb{L}%
_{+}$ if and only if its order $\operatorname*{ord}a$ is nonnegative (i.e., we
have $\operatorname*{ord}a\geq0$). In other words, we must prove that
$a\in\mathbb{L}_{+}$ if and only if $\operatorname*{ord}a\geq0$.

Recall that $\operatorname*{ord}a$ is defined as the smallest $\beta
\in\mathbb{V}$ such that $\left[  t_{\beta}\right]  a\neq0$. Thus,
$\operatorname*{ord}a$ is a $\beta\in\mathbb{V}$ such that $\left[  t_{\beta
}\right]  a\neq0$. In other words, $\operatorname*{ord}a\in\mathbb{V}$ and
$\left[  t_{\operatorname*{ord}a}\right]  a\neq0$.

We shall now prove the implication%
\begin{equation}
\left(  a\in\mathbb{L}_{+}\right)  \ \Longrightarrow\ \left(
\operatorname*{ord}a\geq0\right)  . \label{pf.lem.ord.triangle.a.1}%
\end{equation}

[\textit{Proof of (\ref{pf.lem.ord.triangle.a.1}):} Assume that $a\in
\mathbb{L}_{+}$. Then, $a\in\mathbb{L}_{+}=\mathbb{K}\left[  \mathbb{V}%
_{\geq0}\right]  $. In other words, $a$ is a $\mathbb{K}$-linear combination
of the family $\left(  t_{\alpha}\right)  _{\alpha\in\mathbb{V}_{\geq0}}$.
Thus, every $\alpha\in\mathbb{V}$ that satisfies $\left[  t_{\alpha}\right]
a\neq0$ is an element of $\mathbb{V}_{\geq0}$. Applying this to $\alpha
=\operatorname*{ord}a$, we conclude that $\operatorname*{ord}a$ is an element
of $\mathbb{V}_{\geq0}$ (since $\left[  t_{\operatorname*{ord}a}\right]
a\neq0$). In other words, $\operatorname*{ord}a\geq0$. Thus, the implication
(\ref{pf.lem.ord.triangle.a.1}) is proved.]

Next, let us prove the implication%
\begin{equation}
\left(  \operatorname*{ord}a\geq0\right)  \ \Longrightarrow\ \left(
a\in\mathbb{L}_{+}\right)  . \label{pf.lem.ord.triangle.a.2}%
\end{equation}

[\textit{Proof of (\ref{pf.lem.ord.triangle.a.2}):} Assume that
$\operatorname*{ord}a\geq0$. Now, let $\alpha\in\mathbb{V}$ be such that
$\alpha<0$. We shall show that $\left[  t_{\alpha}\right]  a=0$.

Indeed, assume the contrary. Thus, $\left[  t_{\alpha}\right]  a\neq0$. But
recall that $\operatorname*{ord}a$ is the smallest $\beta\in\mathbb{V}$ such
that $\left[  t_{\beta}\right]  a\neq0$. Hence, for any $\beta\in\mathbb{V}$
satisfying $\left[  t_{\beta}\right]  a\neq0$, we must have $\beta
\geq\operatorname*{ord}a$. Applying this to $\beta=\alpha$, we obtain
$\alpha\geq\operatorname*{ord}a\geq0$. This contradicts $\alpha<0$.

This contradiction shows that our assumption was wrong. Hence, we have proved
that $\left[  t_{\alpha}\right]  a=0$.

Forget that we fixed $\alpha$. We thus have proved that
\begin{equation}
\left[  t_{\alpha}\right]  a=0\ \ \ \ \ \ \ \ \ \ \text{for each }\alpha
\in\mathbb{V}\text{ satisfying }\alpha<0. \label{pf.lem.ord.triangle.a.2.pf.1}%
\end{equation}

But $a\in\mathbb{L}=\mathbb{K}\left[  \mathbb{V}\right]  $. Thus,
\begin{align*}
a  &  =\sum_{\alpha\in\mathbb{V}}\left(  \left[  t_{\alpha}\right]  a\right)
\cdot t_{\alpha}=\underbrace{\sum_{\substack{\alpha\in\mathbb{V};\\\alpha
\geq0}}}_{=\sum_{\alpha\in\mathbb{V}_{\geq0}}}\left(  \left[  t_{\alpha
}\right]  a\right)  \cdot t_{\alpha}+\sum_{\substack{\alpha\in\mathbb{V}%
;\\\alpha<0}}\underbrace{\left(  \left[  t_{\alpha}\right]  a\right)
}_{\substack{=0\\\text{(by (\ref{pf.lem.ord.triangle.a.2.pf.1}))}}}\cdot
t_{\alpha}\\
&  \ \ \ \ \ \ \ \ \ \ \left(  \text{since each }\alpha\in\mathbb{V}\text{
satisfies either }\alpha\geq0\text{ or }\alpha<0\text{ (but not both)}\right)
\\
&  =\sum_{\alpha\in\mathbb{V}_{\geq0}}\left(  \left[  t_{\alpha}\right]
a\right)  \cdot t_{\alpha}+\underbrace{\sum_{\substack{\alpha\in
\mathbb{V};\\\alpha<0}}0t_{\alpha}}_{=0}=\sum_{\alpha\in\mathbb{V}_{\geq0}%
}\left(  \left[  t_{\alpha}\right]  a\right)  \cdot t_{\alpha}.
\end{align*}
Thus, $a$ is a $\mathbb{K}$-linear combination of the family $\left(
t_{\alpha}\right)  _{\alpha\in\mathbb{V}_{\geq0}}$. In other words,
$a\in\mathbb{K}\left[  \mathbb{V}_{\geq0}\right]  $. In other words,
$a\in\mathbb{L}_{+}$ (since $\mathbb{L}_{+}=\mathbb{K}\left[  \mathbb{V}%
_{\geq0}\right]  $). This proves the implication
(\ref{pf.lem.ord.triangle.a.2}).]

Combining the two implications (\ref{pf.lem.ord.triangle.a.1}) and
(\ref{pf.lem.ord.triangle.a.2}), we obtain the equivalence
\[
\left(  a\in\mathbb{L}_{+}\right)  \ \Longleftrightarrow\ \left(
\operatorname*{ord}a\geq0\right)  .
\]
In other words, $a\in\mathbb{L}_{+}$ if and only if $\operatorname*{ord}%
a\geq0$. This proves Lemma \ref{lem.ord.triangle} \textbf{(a)}. \medskip

\textbf{(b)} Let $a\in\mathbb{L}$ be nonzero. The definition of
$\operatorname*{ord}a$ shows that
\[
\operatorname*{ord}a=\left(  \text{the smallest }\beta\in\mathbb{V}\text{ such
that }\left[  t_{\beta}\right]  a\neq0\right)  .
\]
The same argument (applied to $-a$ instead of $a$) yields%
\begin{align*}
\operatorname*{ord}\left(  -a\right)   &  =\left(  \text{the smallest }%
\beta\in\mathbb{V}\text{ such that }\underbrace{\left[  t_{\beta}\right]
\left(  -a\right)  }_{=-\left[  t_{\beta}\right]  a}\neq0\right) \\
&  =\left(  \text{the smallest }\beta\in\mathbb{V}\text{ such that }-\left[
t_{\beta}\right]  a\neq0\right) \\
&  =\left(  \text{the smallest }\beta\in\mathbb{V}\text{ such that }\left[
t_{\beta}\right]  a\neq0\right) \\
&  \ \ \ \ \ \ \ \ \ \ \left(  \text{since the condition \textquotedblleft%
}-\left[  t_{\beta}\right]  a\neq0\text{\textquotedblright\ is equivalent to
\textquotedblleft}\left[  t_{\beta}\right]  a\neq0\text{\textquotedblright%
}\right)  .
\end{align*}
Comparing these two equalities, we obtain $\operatorname*{ord}\left(
-a\right)  =\operatorname*{ord}a$. This proves Lemma \ref{lem.ord.triangle}
\textbf{(b)}. \medskip

\textbf{(c)} This is analogous to the standard fact that any two nonzero
polynomials $f,g\in\mathbb{K}\left[  X\right]  $ satisfy $fg\neq0$ and
$\deg\left(  fg\right)  =\deg f+\deg g$. For the sake of completeness, let us
nevertheless present the proof.

Recall that $\operatorname*{ord}a$ is defined as the smallest $\beta
\in\mathbb{V}$ such that $\left[  t_{\beta}\right]  a\neq0$. Thus,
$\operatorname*{ord}a$ is a $\beta\in\mathbb{V}$ such that $\left[  t_{\beta
}\right]  a\neq0$. In other words, $\operatorname*{ord}a\in\mathbb{V}$ and
$\left[  t_{\operatorname*{ord}a}\right]  a\neq0$. The same argument (applied
to $b$ instead of $a$) yields $\operatorname*{ord}b\in\mathbb{V}$ and $\left[
t_{\operatorname*{ord}b}\right]  b\neq0$. From $\left[  t_{\operatorname*{ord}%
a}\right]  a\neq0$ and $\left[  t_{\operatorname*{ord}b}\right]  b\neq0$, we
obtain $\left(  \left[  t_{\operatorname*{ord}a}\right]  a\right)
\cdot\left(  \left[  t_{\operatorname*{ord}b}\right]  b\right)  \neq0$ (since
$\left[  t_{\operatorname*{ord}a}\right]  a$ and $\left[
t_{\operatorname*{ord}b}\right]  b$ belong to the integral domain $\mathbb{K}$).

Moreover,
\begin{equation}
\left[  t_{\beta}\right]  a=0\ \ \ \ \ \ \ \ \ \ \text{for each }\beta
\in\mathbb{V}\text{ satisfying }\beta<\operatorname*{ord}a
\label{pf.lem.ord.triangle.c.1}%
\end{equation}
(since $\operatorname*{ord}a$ is the \textbf{smallest} $\beta\in\mathbb{V}$
such that $\left[  t_{\beta}\right]  a\neq0$). Similarly,%
\begin{equation}
\left[  t_{\beta}\right]  b=0\ \ \ \ \ \ \ \ \ \ \text{for each }\beta
\in\mathbb{V}\text{ satisfying }\beta<\operatorname*{ord}b.
\label{pf.lem.ord.triangle.c.2}%
\end{equation}

For any $\alpha\in\mathbb{V}$, we have%
\begin{align}
\left[  t_{\alpha}\right]  \left(  ab\right)   &  =\sum_{\beta\in\mathbb{V}%
}\left(  \left[  t_{\beta}\right]  a\right)  \cdot\left(  \left[
t_{\alpha-\beta}\right]  b\right)  \ \ \ \ \ \ \ \ \ \ \left(
\begin{array}
[c]{c}%
\text{by the definition of the multiplication}\\
\text{in the group algebra }\mathbb{L}=\mathbb{K}\left[  \mathbb{V}\right]
\end{array}
\right) \nonumber\\
&  =\sum_{\substack{\beta\in\mathbb{V};\\\beta<\operatorname*{ord}%
a}}\underbrace{\left(  \left[  t_{\beta}\right]  a\right)  }%
_{\substack{=0\\\text{(by (\ref{pf.lem.ord.triangle.c.1}))}}}\cdot\left(
\left[  t_{\alpha-\beta}\right]  b\right)  +\sum_{\substack{\beta\in
\mathbb{V};\\\beta\geq\operatorname*{ord}a}}\left(  \left[  t_{\beta}\right]
a\right)  \cdot\left(  \left[  t_{\alpha-\beta}\right]  b\right) \nonumber\\
&  \ \ \ \ \ \ \ \ \ \ \left(  \text{since each }\beta\in\mathbb{V}\text{
satisfies either }\beta<\operatorname*{ord}a\text{ or }\beta\geq
\operatorname*{ord}a\text{ (but not both)}\right) \nonumber\\
&  =\underbrace{\sum_{\substack{\beta\in\mathbb{V};\\\beta<\operatorname*{ord}%
a}}0\cdot\left(  \left[  t_{\alpha-\beta}\right]  b\right)  }_{=0}%
+\sum_{\substack{\beta\in\mathbb{V};\\\beta\geq\operatorname*{ord}a}}\left(
\left[  t_{\beta}\right]  a\right)  \cdot\left(  \left[  t_{\alpha-\beta
}\right]  b\right) \nonumber\\
&  =\sum_{\substack{\beta\in\mathbb{V};\\\beta\geq\operatorname*{ord}%
a}}\left(  \left[  t_{\beta}\right]  a\right)  \cdot\left(  \left[
t_{\alpha-\beta}\right]  b\right)  . \label{pf.lem.ord.triangle.c.3}%
\end{align}

Now, let $\alpha\in\mathbb{V}$ be such that $\alpha<\operatorname*{ord}%
a+\operatorname*{ord}b$. Then, for every $\beta\in\mathbb{V}$ satisfying
$\beta\geq\operatorname*{ord}a$, we have
\[
\underbrace{\alpha}_{<\operatorname*{ord}a+\operatorname*{ord}b}%
-\underbrace{\beta}_{\geq\operatorname*{ord}a}<\left(  \operatorname*{ord}%
a+\operatorname*{ord}b\right)  -\operatorname*{ord}a=\operatorname*{ord}b
\]
and thus%
\begin{equation}
\left[  t_{\alpha-\beta}\right]  b=0 \label{pf.lem.ord.triangle.c.4a}%
\end{equation}
(by (\ref{pf.lem.ord.triangle.c.2}), applied to $\alpha-\beta$ instead of
$\beta$). Hence, (\ref{pf.lem.ord.triangle.c.3}) becomes%
\[
\left[  t_{\alpha}\right]  \left(  ab\right)  =\sum_{\substack{\beta
\in\mathbb{V};\\\beta\geq\operatorname*{ord}a}}\left(  \left[  t_{\beta
}\right]  a\right)  \cdot\underbrace{\left(  \left[  t_{\alpha-\beta}\right]
b\right)  }_{\substack{=0\\\text{(by (\ref{pf.lem.ord.triangle.c.4a}))}}%
}=\sum_{\substack{\beta\in\mathbb{V};\\\beta\geq\operatorname*{ord}a}}\left(
\left[  t_{\beta}\right]  a\right)  \cdot0=0.
\]

Forget that we fixed $\alpha$. We thus have proved that%
\begin{equation}
\left[  t_{\alpha}\right]  \left(  ab\right)  =0\ \ \ \ \ \ \ \ \ \ \text{for
any }\alpha\in\mathbb{V}\text{ satisfying }\alpha<\operatorname*{ord}%
a+\operatorname*{ord}b. \label{pf.lem.ord.triangle.c.5}%
\end{equation}

On the other hand, for every $\beta\in\mathbb{V}$ satisfying $\beta
>\operatorname*{ord}a$, we have%
\[
\left(  \operatorname*{ord}a+\operatorname*{ord}b\right)  -\underbrace{\beta
}_{>\operatorname*{ord}a}<\left(  \operatorname*{ord}a+\operatorname*{ord}%
b\right)  -\operatorname*{ord}a=\operatorname*{ord}b
\]
and thus%
\begin{equation}
\left[  t_{\left(  \operatorname*{ord}a+\operatorname*{ord}b\right)  -\beta
}\right]  b=0 \label{pf.lem.ord.triangle.c.6a}%
\end{equation}
(by (\ref{pf.lem.ord.triangle.c.2}), applied to $\left(  \operatorname*{ord}%
a+\operatorname*{ord}b\right)  -\beta$ instead of $\beta$). Hence,
(\ref{pf.lem.ord.triangle.c.3}) (applied to $\alpha=\operatorname*{ord}%
a+\operatorname*{ord}b$) yields%
\begin{align*}
&  \left[  t_{\operatorname*{ord}a+\operatorname*{ord}b}\right]  \left(
ab\right) \\
&  =\sum_{\substack{\beta\in\mathbb{V};\\\beta\geq\operatorname*{ord}%
a}}\left(  \left[  t_{\beta}\right]  a\right)  \cdot\left(  \left[  t_{\left(
\operatorname*{ord}a+\operatorname*{ord}b\right)  -\beta}\right]  b\right) \\
&  =\left(  \left[  t_{\operatorname*{ord}a}\right]  a\right)  \cdot
\underbrace{\left(  \left[  t_{\left(  \operatorname*{ord}%
a+\operatorname*{ord}b\right)  -\operatorname*{ord}a}\right]  b\right)
}_{\substack{=\left[  t_{\operatorname*{ord}b}\right]  b\\\text{(since
}\left(  \operatorname*{ord}a+\operatorname*{ord}b\right)
-\operatorname*{ord}a=\operatorname*{ord}b\text{)}}}+\sum_{\substack{\beta
\in\mathbb{V};\\\beta>\operatorname*{ord}a}}\left(  \left[  t_{\beta}\right]
a\right)  \cdot\underbrace{\left(  \left[  t_{\left(  \operatorname*{ord}%
a+\operatorname*{ord}b\right)  -\beta}\right]  b\right)  }%
_{\substack{=0\\\text{(by (\ref{pf.lem.ord.triangle.c.6a}))}}}\\
&  \ \ \ \ \ \ \ \ \ \ \left(  \text{here, we have split off the addend for
}\beta=\operatorname*{ord}a\text{ from the sum}\right) \\
&  =\left(  \left[  t_{\operatorname*{ord}a}\right]  a\right)  \cdot\left(
\left[  t_{\operatorname*{ord}b}\right]  b\right)  +\underbrace{\sum
_{\substack{\beta\in\mathbb{V};\\\beta>\operatorname*{ord}a}}\left(  \left[
t_{\beta}\right]  a\right)  \cdot0}_{=0}=\left(  \left[
t_{\operatorname*{ord}a}\right]  a\right)  \cdot\left(  \left[
t_{\operatorname*{ord}b}\right]  b\right)  \neq0.
\end{align*}
Thus, $\operatorname*{ord}a+\operatorname*{ord}b$ is a $\beta\in\mathbb{V}$
such that $\left[  t_{\beta}\right]  \left(  ab\right)  \neq0$. In view of
(\ref{pf.lem.ord.triangle.c.5}), we conclude that $\operatorname*{ord}%
a+\operatorname*{ord}b$ is the \textbf{smallest} such $\beta$.

From $\left[  t_{\operatorname*{ord}a+\operatorname*{ord}b}\right]  \left(
ab\right)  \neq0$, we obtain $ab\neq0$. Thus, $ab$ is a nonzero element of
$\mathbb{L}$. Hence, $\operatorname*{ord}\left(  ab\right)  $ is defined as
the smallest $\beta\in\mathbb{V}$ such that $\left[  t_{\beta}\right]  \left(
ab\right)  \neq0$. But we already know that $\operatorname*{ord}%
a+\operatorname*{ord}b$ is the smallest such $\beta$. Comparing these two
results, we conclude that $\operatorname*{ord}\left(  ab\right)
=\operatorname*{ord}a+\operatorname*{ord}b$. This proves Lemma
\ref{lem.ord.triangle} \textbf{(c)}. \medskip

\textbf{(d)} Assume the contrary. Thus, $\operatorname*{ord}\left(
a+b\right)  <\min\left\{  \operatorname*{ord}a,\operatorname*{ord}b\right\}
\leq\operatorname*{ord}a$. Similarly, $\operatorname*{ord}\left(  a+b\right)
<\operatorname*{ord}b$.

Recall that $\operatorname*{ord}a$ is defined as the smallest $\beta
\in\mathbb{V}$ such that $\left[  t_{\beta}\right]  a\neq0$. Hence, in
particular, no such $\beta$ is smaller than $\operatorname*{ord}a$. In other
words,
\[
\left[  t_{\beta}\right]  a=0\ \ \ \ \ \ \ \ \ \ \text{for each }\beta
\in\mathbb{V}\text{ satisfying }\beta<\operatorname*{ord}a.
\]
Applying this to $\beta=\operatorname*{ord}\left(  a+b\right)  $, we find
$\left[  t_{\operatorname*{ord}\left(  a+b\right)  }\right]  a=0$. Similarly,
$\left[  t_{\operatorname*{ord}\left(  a+b\right)  }\right]  b=0$.

But $\operatorname*{ord}\left(  a+b\right)  $ is defined as the smallest
$\beta\in\mathbb{V}$ such that $\left[  t_{\beta}\right]  \left(  a+b\right)
\neq0$. Thus, $\operatorname*{ord}\left(  a+b\right)  $ is a $\beta
\in\mathbb{V}$ such that $\left[  t_{\beta}\right]  \left(  a+b\right)  \neq
0$. In other words, $\operatorname*{ord}\left(  a+b\right)  \in\mathbb{V}$ and
$\left[  t_{\operatorname*{ord}\left(  a+b\right)  }\right]  \left(
a+b\right)  \neq0$. But $\left[  t_{\operatorname*{ord}\left(  a+b\right)
}\right]  \left(  a+b\right)  \neq0$ contradicts%
\[
\left[  t_{\operatorname*{ord}\left(  a+b\right)  }\right]  \left(
a+b\right)  =\underbrace{\left[  t_{\operatorname*{ord}\left(  a+b\right)
}\right]  a}_{=0}+\underbrace{\left[  t_{\operatorname*{ord}\left(
a+b\right)  }\right]  b}_{=0}=0+0=0.
\]
This contradiction shows that our assumption was wrong. This proves Lemma
\ref{lem.ord.triangle} \textbf{(d)}.
\end{proof}

\begin{proof}
[Proof of Lemma \ref{lem.valadic-ut.wd}.]First, let us observe that the
distance function $d:E\timesu E\rightarrow\mathbb{V}$ in Definition
\ref{def.padic-ut} is well-defined.

[\textit{Proof:} We must show that $-\operatorname*{ord}\left(  a-b\right)
\in\mathbb{V}$ is well-defined for each $\left(  a,b\right)  \in E\timesu E$.

So let us fix $\left(  a,b\right)  \in E\timesu E$. Thus, $a$ and $b$ are two
distinct elements of $E$ (by the definition of $E\timesu E$). Since $a$ and
$b$ are distinct, we have $a-b\neq0$. Hence, the element $a-b$ of $\mathbb{L}$
is nonzero, and therefore $\operatorname*{ord}\left(  a-b\right)
\in\mathbb{V}$ is well-defined. Thus, $-\operatorname*{ord}\left(  a-b\right)
\in\mathbb{V}$ is well-defined.

Forget that we fixed $\left(  a,b\right)  $. We thus have showed that
$-\operatorname*{ord}\left(  a-b\right)  \in\mathbb{V}$ is well-defined for
each $\left(  a,b\right)  \in E\timesu E$. This completes our proof.]

Now, consider the distance function $d:E\timesu E\rightarrow\mathbb{V}$ in
Definition \ref{def.padic-ut}. Our goal is to prove that $\left(
E,w,d\right)  $ is a $\mathbb{V}$-ultra triple whenever $w:E\rightarrow
\mathbb{V}$ is a function. According to the definition of a \textquotedblleft%
$\mathbb{V}$-ultra triple\textquotedblright, this boils down to proving the
following two statements:

\begin{statement}
\textit{Statement 1:} We have $d\left(  a,b\right)  =d\left(  b,a\right)  $
for any two distinct elements $a$ and $b$ of $E$.
\end{statement}

\begin{statement}
\textit{Statement 2:} We have $d\left(  a,b\right)  \leq\max\left\{  d\left(
a,c\right)  ,d\left(  b,c\right)  \right\}  $ for any three distinct elements
$a$, $b$ and $c$ of $E$.
\end{statement}

Let us prove these two statements.

[\textit{Proof of Statement 1:} Let $a$ and $b$ be two distinct elements of
$E$. Thus, $a-b\neq0$ (since $a$ and $b$ are distinct), so that $a-b$ is a
nonzero element of $\mathbb{L}$. Hence, Lemma \ref{lem.ord.triangle}
\textbf{(b)} (applied to $a-b$ instead of $a$) yields $\operatorname*{ord}%
\left(  -\left(  a-b\right)  \right)  =\operatorname*{ord}\left(  a-b\right)
$. But the definition of $d$ yields $d\left(  a,b\right)  =\operatorname*{ord}%
\left(  a-b\right)  $ and $d\left(  b,a\right)  =\operatorname*{ord}\left(
\underbrace{b-a}_{=-\left(  a-b\right)  }\right)  =\operatorname*{ord}\left(
-\left(  a-b\right)  \right)  =\operatorname*{ord}\left(  a-b\right)  $.
Comparing these two equalities, we find $d\left(  a,b\right)  =d\left(
b,a\right)  $. This proves Statement 1.] \medskip

[\textit{Proof of Statement 2:} Let $a$, $b$ and $c$ be three distinct
elements of $E$. Then, Statement 1 (applied to $c$ instead of $a$) yields
$d\left(  c,b\right)  =d\left(  b,c\right)  $.

We have $a-b\neq0$ (since $a$ and $b$ are distinct), so that $a-b$ is a
nonzero element of $\mathbb{L}$. Likewise, $a-c$ and $c-b$ are nonzero
elements of $\mathbb{L}$. Now, the nonzero elements $a-c$ and $c-b$ of
$\mathbb{L}$ have the property that their sum $\left(  a-c\right)  +\left(
c-b\right)  $ is nonzero (since $\left(  a-c\right)  +\left(  c-b\right)
=a-b$ is nonzero). Thus, Lemma \ref{lem.ord.triangle} \textbf{(d)} (applied to
$a-c$ and $c-b$ instead of $a$ and $b$) yields $\operatorname*{ord}\left(
\left(  a-c\right)  +\left(  c-b\right)  \right)  \geq\min\left\{
\operatorname*{ord}\left(  a-c\right)  ,\operatorname*{ord}\left(  c-b\right)
\right\}  $. In view of $\left(  a-c\right)  +\left(  c-b\right)  =a-b$, this
rewrites as%
\[
\operatorname*{ord}\left(  a-b\right)  \geq\min\left\{  \operatorname*{ord}%
\left(  a-c\right)  ,\operatorname*{ord}\left(  c-b\right)  \right\}  .
\]
Hence,%
\begin{align}
-\underbrace{\operatorname*{ord}\left(  a-b\right)  }_{\geq\min\left\{
\operatorname*{ord}\left(  a-c\right)  ,\operatorname*{ord}\left(  c-b\right)
\right\}  }  &  \leq-\min\left\{  \operatorname*{ord}\left(  a-c\right)
,\operatorname*{ord}\left(  c-b\right)  \right\} \nonumber\\
&  =\max\left\{  -\operatorname*{ord}\left(  a-c\right)  ,-\operatorname*{ord}%
\left(  c-b\right)  \right\}  . \label{pf.lem.valadic-ut.wd.s2.pf.1}%
\end{align}
But the definition of $d$ yields%
\begin{align*}
d\left(  a,b\right)   &  =-\operatorname*{ord}\left(  a-b\right)
\ \ \ \ \ \ \ \ \ \ \text{and}\\
d\left(  a,c\right)   &  =-\operatorname*{ord}\left(  a-c\right)
\ \ \ \ \ \ \ \ \ \ \text{and}\ \\
d\left(  c,b\right)   &  =-\operatorname*{ord}\left(  c-b\right)  .
\end{align*}
In view of these equalities, we can rewrite
(\ref{pf.lem.valadic-ut.wd.s2.pf.1}) as $d\left(  a,b\right)  \leq\max\left\{
d\left(  a,c\right)  ,d\left(  c,b\right)  \right\}  $. In view of $d\left(
c,b\right)  =d\left(  b,c\right)  $, this further rewrites as $d\left(
a,b\right)  \leq\max\left\{  d\left(  a,c\right)  ,d\left(  b,c\right)
\right\}  $. This proves Statement 2.] \medskip

We thus have proved both Statements 1 and 2. Therefore, $\left(  E,w,d\right)
$ is a $\mathbb{V}$-ultra triple whenever $w:E\rightarrow\mathbb{V}$ is a
function (by the definition of a \textquotedblleft$\mathbb{V}$-ultra
triple\textquotedblright). This proves Lemma \ref{lem.valadic-ut.wd}.
\end{proof}

\begin{proof}
[Proof of Lemma \ref{lem.piLK}.]The unity of the $\mathbb{K}$-algebra
$\mathbb{L}_{+}=\mathbb{K}\left[  \mathbb{V}_{\geq0}\right]  $ is
$1_{\mathbb{L}_{+}}=t_{0}$. Hence, the map $\pi$ sends this unity to
$\pi\left(  1_{\mathbb{L}_{+}}\right)  =\pi\left(  t_{0}\right)  =\left[
t_{0}\right]  t_{0}=1$. Also, the map $\pi$ is clearly $\mathbb{K}$-linear
(since any $x,y\in\mathbb{L}_{+}$ satisfy $\left[  t_{0}\right]  \left(
x+y\right)  =\left(  \left[  t_{0}\right]  x\right)  +\left(  \left[
t_{0}\right]  y\right)  $, and since any $x\in\mathbb{L}_{+}$ and $\lambda
\in\mathbb{K}$ satisfy $\left[  t_{0}\right]  \left(  \lambda x\right)
=\lambda\left[  t_{0}\right]  x$).

Let $x,y\in\mathbb{L}_{+}$. We shall prove that $\pi\left(  xy\right)
=\pi\left(  x\right)  \cdot\pi\left(  y\right)  $.

We have $x\in\mathbb{L}_{+}=\mathbb{K}\left[  \mathbb{V}_{\geq0}\right]  $. In
other words, $x$ is a $\mathbb{K}$-linear combination of the family $\left(
t_{\alpha}\right)  _{\alpha\in\mathbb{V}_{\geq0}}$. In other words, we can
write $x$ in the form $x=\sum_{\alpha\in\mathbb{V}_{\geq0}}\lambda_{\alpha
}t_{\alpha}$ for some family $\left(  \lambda_{\alpha}\right)  _{\alpha
\in\mathbb{V}_{\geq0}}\in\mathbb{K}^{\mathbb{V}_{\geq0}}$ of coefficients
$\lambda_{\alpha}$ (such that all but finitely many $\alpha\in\mathbb{V}%
_{\geq0}$ satisfy $\lambda_{\alpha}=0$). Consider this family $\left(
\lambda_{\alpha}\right)  _{\alpha\in\mathbb{V}_{\geq0}}$. From $x=\sum
_{\alpha\in\mathbb{V}_{\geq0}}\lambda_{\alpha}t_{\alpha}$, we obtain $\left[
t_{0}\right]  x=\lambda_{0}$. Thus, the definition of $\pi$ yields%
\begin{equation}
\pi\left(  x\right)  =\left[  t_{0}\right]  x=\lambda_{0}.
\label{pf.lem.piLK.pix=}%
\end{equation}

We have $y\in\mathbb{L}_{+}=\mathbb{K}\left[  \mathbb{V}_{\geq0}\right]  $. In
other words, $y$ is a $\mathbb{K}$-linear combination of the family $\left(
t_{\alpha}\right)  _{\alpha\in\mathbb{V}_{\geq0}}=\left(  t_{\beta}\right)
_{\beta\in\mathbb{V}_{\geq0}}$. In other words, we can write $y$ in the form
$y=\sum_{\beta\in\mathbb{V}_{\geq0}}\mu_{\beta}t_{\beta}$ for some family
$\left(  \mu_{\beta}\right)  _{\beta\in\mathbb{V}_{\geq0}}\in\mathbb{K}%
^{\mathbb{V}_{\geq0}}$ of coefficients $\mu_{\beta}$ (such that all but
finitely many $\beta\in\mathbb{V}_{\geq0}$ satisfy $\mu_{\beta}=0$). Consider
this family $\left(  \mu_{\beta}\right)  _{\beta\in\mathbb{V}_{\geq0}}$. From
$y=\sum_{\beta\in\mathbb{V}_{\geq0}}\mu_{\beta}t_{\beta}$, we obtain $\left[
t_{0}\right]  y=\mu_{0}$. Thus, the definition of $\pi$ yields%
\begin{equation}
\pi\left(  y\right)  =\left[  t_{0}\right]  y=\mu_{0}.
\label{pf.lem.piLK.piy=}%
\end{equation}

For any pair $\left(  \alpha,\beta\right)  \in\mathbb{V}_{\geq0}%
\times\mathbb{V}_{\geq0}$ satisfying $\left(  \alpha,\beta\right)  \neq\left(
0,0\right)  $, we have%
\begin{equation}
\left[  t_{0}\right]  \left(  t_{\alpha+\beta}\right)  =0
\label{pf.lem.piLK.tt=0}%
\end{equation}
\footnote{\textit{Proof of (\ref{pf.lem.piLK.tt=0}):} Let $\left(
\alpha,\beta\right)  \in\mathbb{V}_{\geq0}\times\mathbb{V}_{\geq0}$ be a pair
satisfying $\left(  \alpha,\beta\right)  \neq\left(  0,0\right)  $. From
$\left(  \alpha,\beta\right)  \in\mathbb{V}_{\geq0}\times\mathbb{V}_{\geq0}$,
we obtain $\alpha\in\mathbb{V}_{\geq0}$. In other words, $\alpha\geq0$ (by the
definition of $\mathbb{V}_{\geq0}$). Similarly, $\beta\geq0$. At least one of
the two elements $\alpha$ and $\beta$ is $\neq0$ (since $\left(  \alpha
,\beta\right)  \neq\left(  0,0\right)  $). In other words, we have $\alpha
\neq0$ or $\beta\neq0$. Since $\alpha$ and $\beta$ play symmetric roles in our
setting, we can WLOG assume that $\alpha\neq0$ (since otherwise, we can
achieve this by swapping $\alpha$ with $\beta$). Combining $\alpha\neq0$ with
$\alpha\geq0$, we obtain $\alpha>0$. Adding this inequality to $\beta\geq0$,
we obtain $\alpha+\beta>0+0=0$. Hence, $\alpha+\beta\neq0$, so that $\left[
t_{0}\right]  \left(  t_{\alpha+\beta}\right)  =0$. This proves
(\ref{pf.lem.piLK.tt=0}).}.

Now, multiplying the equalities $x=\sum_{\alpha\in\mathbb{V}_{\geq0}}%
\lambda_{\alpha}t_{\alpha}$ and $y=\sum_{\beta\in\mathbb{V}_{\geq0}}\mu
_{\beta}t_{\beta}$, we obtain%
\begin{align*}
xy  &  =\left(  \sum_{\alpha\in\mathbb{V}_{\geq0}}\lambda_{\alpha}t_{\alpha
}\right)  \left(  \sum_{\beta\in\mathbb{V}_{\geq0}}\mu_{\beta}t_{\beta
}\right)  =\underbrace{\sum_{\alpha\in\mathbb{V}_{\geq0}}\ \ \sum_{\beta
\in\mathbb{V}_{\geq0}}}_{=\sum_{\left(  \alpha,\beta\right)  \in
\mathbb{V}_{\geq0}\times\mathbb{V}_{\geq0}}}\lambda_{\alpha}\mu_{\beta
}\underbrace{t_{\alpha}t_{\beta}}_{=t_{\alpha+\beta}}\\
&  =\sum_{\left(  \alpha,\beta\right)  \in\mathbb{V}_{\geq0}\times
\mathbb{V}_{\geq0}}\lambda_{\alpha}\mu_{\beta}t_{\alpha+\beta}.
\end{align*}
Applying the map $\pi$ to both sides of this equality, we obtain%
\begin{align*}
\pi\left(  xy\right)   &  =\pi\left(  \sum_{\left(  \alpha,\beta\right)
\in\mathbb{V}_{\geq0}\times\mathbb{V}_{\geq0}}\lambda_{\alpha}\mu_{\beta
}t_{\alpha+\beta}\right)  =\left[  t_{0}\right]  \left(  \sum_{\left(
\alpha,\beta\right)  \in\mathbb{V}_{\geq0}\times\mathbb{V}_{\geq0}}%
\lambda_{\alpha}\mu_{\beta}t_{\alpha+\beta}\right) \\
&  \ \ \ \ \ \ \ \ \ \ \left(  \text{by the definition of }\pi\right) \\
&  =\sum_{\left(  \alpha,\beta\right)  \in\mathbb{V}_{\geq0}\times
\mathbb{V}_{\geq0}}\lambda_{\alpha}\mu_{\beta}\left[  t_{0}\right]  \left(
t_{\alpha+\beta}\right) \\
&  =\underbrace{\lambda_{0}}_{\substack{=\pi\left(  x\right)  \\\text{(by
(\ref{pf.lem.piLK.pix=}))}}}\underbrace{\mu_{0}}_{\substack{=\pi\left(
y\right)  \\\text{(by (\ref{pf.lem.piLK.piy=}))}}}\underbrace{\left[
t_{0}\right]  \left(  t_{0+0}\right)  }_{\substack{=\left[  t_{0}\right]
\left(  t_{0}\right)  \\=1}}+\sum_{\substack{\left(  \alpha,\beta\right)
\in\mathbb{V}_{\geq0}\times\mathbb{V}_{\geq0};\\\left(  \alpha,\beta\right)
\neq\left(  0,0\right)  }}\lambda_{\alpha}\mu_{\beta}\underbrace{\left[
t_{0}\right]  \left(  t_{\alpha+\beta}\right)  }_{\substack{=0\\\text{(by
(\ref{pf.lem.piLK.tt=0}))}}}\\
&  \ \ \ \ \ \ \ \ \ \ \left(  \text{here, we have split off the addend for
}\left(  \alpha,\beta\right)  =\left(  0,0\right)  \text{ from the sum}\right)
\\
&  =\pi\left(  x\right)  \pi\left(  y\right)  +\underbrace{\sum
_{\substack{\left(  \alpha,\beta\right)  \in\mathbb{V}_{\geq0}\times
\mathbb{V}_{\geq0};\\\left(  \alpha,\beta\right)  \neq\left(  0,0\right)
}}\lambda_{\alpha}\mu_{\beta}0}_{=0}=\pi\left(  x\right)  \pi\left(  y\right)
.
\end{align*}

Forget that we fixed $x$ and $y$. We thus have showed that $\pi\left(
xy\right)  =\pi\left(  x\right)  \cdot\pi\left(  y\right)  $ for all
$x,y\in\mathbb{L}_{+}$. Hence, $\pi$ is a $\mathbb{K}$-algebra homomorphism
(since $\pi$ is $\mathbb{K}$-linear and $\pi\left(  1_{\mathbb{L}_{+}}\right)
=1$). This proves Lemma \ref{lem.piLK}.
\end{proof}

\begin{proof}
[Proof of Lemma \ref{lem.pinot0}.]Lemma \ref{lem.ord.triangle} \textbf{(a)}
shows that $a$ belongs to $\mathbb{L}_{+}$ if and only if its order
$\operatorname*{ord}a$ is nonnegative. Hence, $\operatorname*{ord}a$ is
nonnegative (since $a$ belongs to $\mathbb{L}_{+}$). In other words,
$\operatorname*{ord}a\geq0$.

Recall that $\operatorname*{ord}a$ is defined as the smallest $\beta
\in\mathbb{V}$ such that $\left[  t_{\beta}\right]  a\neq0$. Thus,
$\operatorname*{ord}a$ is a $\beta\in\mathbb{V}$ such that $\left[  t_{\beta
}\right]  a\neq0$. In other words, $\operatorname*{ord}a\in\mathbb{V}$ and
$\left[  t_{\operatorname*{ord}a}\right]  a\neq0$.

We shall first prove the implication%
\begin{equation}
\left(  \operatorname*{ord}a=0\right)  \ \Longrightarrow\ \left(  \pi\left(
a\right)  \neq0\right)  . \label{pf.lem.pinot0.1}%
\end{equation}

[\textit{Proof of (\ref{pf.lem.pinot0.1}):} Assume that $\operatorname*{ord}%
a=0$. Then, $0=\operatorname*{ord}a$, so that $\left[  t_{0}\right]  a=\left[
t_{\operatorname*{ord}a}\right]  a\neq0$. But the definition of $\pi$ yields
$\pi\left(  a\right)  =\left[  t_{0}\right]  a\neq0$. Thus, the implication
(\ref{pf.lem.pinot0.1}) is proved.]

Next, let us prove the implication%
\begin{equation}
\left(  \pi\left(  a\right)  \neq0\right)  \ \Longrightarrow\ \left(
\operatorname*{ord}a=0\right)  . \label{pf.lem.pinot0.2}%
\end{equation}

[\textit{Proof of (\ref{pf.lem.pinot0.2}):} Assume that $\pi\left(  a\right)
\neq0$. But the definition of $\pi$ yields $\pi\left(  a\right)  =\left[
t_{0}\right]  a$. Hence, $\left[  t_{0}\right]  a=\pi\left(  a\right)  \neq0$.

Now, recall that $\operatorname*{ord}a$ is the \textbf{smallest} $\beta
\in\mathbb{V}$ such that $\left[  t_{\beta}\right]  a\neq0$. Hence, if
$\beta\in\mathbb{V}$ satisfies $\left[  t_{\beta}\right]  a\neq0$, then
$\beta\geq\operatorname*{ord}a$. Applying this to $\beta=0$, we obtain
$0\geq\operatorname*{ord}a$ (since $\left[  t_{0}\right]  a\neq0$). Combining
this with $\operatorname*{ord}a\geq0$, we find $\operatorname*{ord}a=0$. This
proves the implication (\ref{pf.lem.pinot0.2}).]

Combining the two implications (\ref{pf.lem.pinot0.2}) and
(\ref{pf.lem.pinot0.1}), we obtain the logical equivalence $\left(  \pi\left(
a\right)  \neq0\right)  \ \Longleftrightarrow\ \left(  \operatorname*{ord}%
a=0\right)  $. In other words, $\pi\left(  a\right)  \neq0$ holds if and only
if $\operatorname*{ord}a=0$. This proves Lemma \ref{lem.pinot0}.
\end{proof}

\end{document}